\documentclass{amsart}
\usepackage{amssymb}
\usepackage{amsfonts}
\usepackage{geometry}

\newtheorem{theorem}{Theorem}
\theoremstyle{plain}

\newtheorem{corollary}[theorem]{Corollary}

\newtheorem{definition}[theorem]{Definition}

\newtheorem{lemma}[theorem]{Lemma}

\newtheorem{problem}[theorem]{Problem}
\newtheorem{proposition}[theorem]{Proposition}
\newtheorem{remark}[theorem]{Remark}

\numberwithin{equation}{section}
\input{tcilatex}
\makeatletter
\@namedef{subjclassname@2020}{\textup{2020} Mathematics Subject Classification}
\makeatother

\begin{document}
\title[Sums of squares III]{Sums of squares III: hypoellipticity in the
infinitely degenerate regime}
\author{Lyudmila Korobenko}
\address{Reed College, Portland, Oregon, USA, korobenko@reed.edu}
\author{Eric Sawyer}
\address{McMaster University, Hamilton, Ontario, Canada, sawyer@mcmaster.ca}
\thanks{The second author is partially supported by NSERC grant number 12409
and the McKay Research Chair grant at McMaster University}
\subjclass[2020]{35B65, 35J70, 35H10.}
\keywords{infinite degeneracy, elliptic operators, hypoellipticity, sums of
squares, vector fields.}

\begin{abstract}
This is the third paper in a series of three dealing with sums of squares
and hypoellipticity in the infinitely degenerate regime. We establish a $%
C^{2,\delta }$ generalization of M. Christ's smooth sum of squares theorem,
and then use a bootstrap argument with the sum of squares decomposition for
matrix functions, obtained in our second paper of this series, to prove a
hypoellipticity theorem that generalizes some cases of the results of
Christ, Hoshiro, Koike, Kusuoka and Stroock and Morimoto for sums of
squares, and of Fed\u{\i}i and Kohn for degeneracies not necessarily a sum
of squares.
\end{abstract}

\maketitle
\tableofcontents

\section{Introduction}

The regularity theory of second order \emph{subelliptic} linear equations
with smooth coefficients is well established, see e.g. H\"{o}rmander \cite%
{Ho} and Fefferman and Phong \cite{FePh}. In \cite{Ho}, H\"{o}rmander
obtained hypoellipticity of sums of squares of smooth vector fields plus a
lower order term, whose Lie algebra spans at every point. In \cite{FePh},
Fefferman and Phong considered general nonnegative semidefinite smooth
self-adjoint linear operators, and characterized subellipticity in terms of
a containment condition involving Euclidean balls and \textquotedblright
subunit\textquotedblright\ balls related to the geometry of the nonnegative
semidefinite form associated to the operator. Of course subelliptic
operators $L$ with smooth coefficients are hypoelliptic, namely every
distribution solution $u$ of $Lu=\phi $ is smooth when $\phi \,$is smooth.
In the converse direction, H\"{o}rmander also showed in \cite{Ho} that a sum
of squares of smooth vector fields in $\mathbb{R}^{n}$, with constant rank
Lie algebras, is hypoelliptic if and only if the rank is $n$. See Tr\`{e}ves 
\cite{Tre} for a treatment of further results on characterizing
hypoellipticity in certain special cases.

However, the question of hypoellipticity in general remains largely a
mystery. A possible form for a characterization involving the effective
symbol $\widetilde{\sigma }\left( x,\xi \right) $ (when it exists) is given
by Christ in \cite{Chr2}, motivated by his main hypoellipticity theorem for
sums of squares in the infinitely degenerate regime in \cite[see Main
Theorem 2.3]{Chr}. We will generalize this latter theorem of Christ to hold
for $C^{2,\delta }$ symbols, which will play a major role in Theorem \ref%
{HYP G} below on hypoellipticity in the infinitely degenerate regime.

Thus a basic obstacle to understanding hypoellipticity in general arises
when ellipticity degenerates to infinite order in some directions, and we
briefly review what is known in this infinite regime here. The theory has
only had its surface scratched so far, as evidenced by the results of Fedii 
\cite{Fe}, Kusuoka and Strook \cite{KuStr}, Kohn \cite{Koh}, Koike \cite{Koi}%
, Korobenko and Rios \cite{KoRi}, Morimoto \cite{Mor}, Akhunov, Korobenko
and Rios \cite{AkKoRi}, and the aforementioned paper of Christ \cite{Chr},
to name just a few. In the \emph{rough} infinitely differentiable regime,
Rios, Sawyer and Wheeden \cite{RiSaWh} had earlier obtained results
analogous to those in \cite{KoRi}, where $L$ is `rough' hypoelliptic if
every \emph{weak} solution $u$ of $Lu=\phi $ is continuous when $\phi \,$is
bounded.

In \cite{Fe}, Fedii proved that the two-dimensional operator $\frac{\partial 
}{\partial x^{2}}+f\left( x\right) ^{2}\frac{\partial }{\partial y^{2}}$ is
hypoelliptic merely under the assumption that $f$ is smooth and positive
away from $x=0$. In \cite{KuStr}, Kusuoka and Strook showed using
probabilistic methods that under the same conditions on $f\left( x\right) $,
the three-dimensional analogue $\frac{\partial ^{2}}{\partial x^{2}}+\frac{%
\partial ^{2}}{\partial y^{2}}+f\left( x\right) ^{2}\frac{\partial ^{2}}{%
\partial z^{2}}$ of Fedii's operator is hypoelliptic \emph{if and only if} 
\begin{equation*}
\lim_{x\rightarrow 0}x\ln f\left( x\right) =0.
\end{equation*}%
Morimoto \cite{Mor} and Koike \cite{Koi} introduced the use of
nonprobabilistic methods, and further refinements of this approach were
obtained in Christ \cite{Chr}, using a general theorem on hypoellipticity of
sums of squares of smooth vector fields in the infinite regime, i.e. where\
the Lie algebra does \emph{not} span at all points. In particular, for the
operator $L_{3}=\frac{\partial ^{2}}{\partial x^{2}}+a^{2}(x)\frac{\partial
^{2}}{\partial y^{2}}+b^{2}(x)\frac{\partial ^{2}}{\partial z^{2}}\ $in$\ 
\mathbb{R}^{3}$, Christ proved that if $a,b\in C^{\infty }$ are even,
elliptic, nondecreasing on $[0,\infty )$, and $a(x)\geq b(x)$ for all $x$,
and if in addition $\limsup_{x\rightarrow 0}|x\ln a(x)|\neq 0$, and the
coefficient $b$ satisfies 
\begin{equation*}
\lim_{x\rightarrow 0}b(x)x|\ln a(x)|=0,
\end{equation*}%
then $L_{3}$ is hypoelliptic. Moreover, he showed that if some partial
derivative of $b$ is nonzero at $x=0$, then $L_{3}$ is hypoelliptic \emph{if
and only if} the above condition holds.

On the other hand, the novelty in Kohn \cite{Koh}, which was generalized in 
\cite{KoRi}, was the absence of any assumption regarding sums of squares of
vector fields. This is relevant since it is an open problem whether or not
there are smooth nonnegative functions $\lambda $ on the real line vanishing
only at the origin, and to infinite order there, such that they \textbf{%
cannot} be written as a finite sum $\lambda =\sum_{n=1}^{N}f_{n}^{2}$ of
squares of smooth functions $f_{n}$. The existence of such examples are
attributed to\ Paul Cohen in both \cite{Bru} and \cite{BoCoRo}, but
apparently no example has ever appeared in the literature, and the existence
of such an example is an open problem, see \cite[Remark 5.1]{Pie}\footnote{%
See also https:/mathoverflow.net/a/106072}. This extends moreover to
matrices since if a matrix is a sum of squares (equivalently a sum of\
positive rank one matrices), then each of its diagonal elements is as well.
On the other hand, Kohn makes the additional assumption that $\lambda \left(
x\right) $ vanishes only at the origin in $\mathbb{R}^{m}$, something not
necessarily assumed in the other aforementioned works. More importantly,
Kohn's theorem applies only to operators of Grushin type $L\left( x,D\right)
+\lambda \left( x\right) L\left( y,D\right) $, where the degeneracy $\lambda
\left( x\right) $ factors out of the operator $\lambda \left( x\right)
L\left( y,D\right) $, a restriction that this paper will in part work to
remove.

Missing then is a treatment of more general smooth operators $L=\nabla
A\left( x\right) \nabla +\func{lower}\func{order}\func{terms}$, whose matrix 
$A\left( x\right) $ is \emph{comparable} to an operator in diagonal form of
the types considered above - see Definition \ref{comp mat} below. Our
purpose in this paper is to address this more general case in the following
setting of real-valued differential operators. Suppose $1\leq m<p\leq n$.
Let $L=\nabla A\left( x\right) \nabla $ where $A\left( x\right) \sim D_{%
\mathbf{\lambda }}\left( \tilde{x}\right) $ with $\tilde{x}=\left(
x_{1},...,x_{m}\right) $, $x=\left( x_{1},...,x_{n}\right) $ and where $D_{%
\mathbf{\lambda }}\left( \tilde{x}\right) $ has $C^{2}$ nonnegative diagonal
entries $\lambda _{1}\left( \tilde{x}\right) ,...,\lambda _{n}\left( \tilde{x%
}\right) $ depending only on $\tilde{x}$ and positive away from the origin
in $\mathbb{R}^{m}$:%
\begin{equation*}
A\left( x\right) \sim D_{\mathbf{\lambda }}\left( \tilde{x}\right) =\left[ 
\begin{array}{ccc}
\mathbb{I}_{m} & \mathbf{0}_{m\times \left( p-m-1\right) } & \mathbf{0}%
_{m\times \left( n-p+1\right) } \\ 
\mathbf{0}_{\left( p-m-1\right) \times m} & D_{\left\{ \lambda _{m+1}\left( 
\tilde{x}\right) ,...,\lambda _{p-1}\left( \tilde{x}\right) \right\} } & 
\mathbf{0}_{\left( p-m-1\right) \times \left( n-p+1\right) } \\ 
\mathbf{0}_{\left( n-p+1\right) \times m} & \mathbf{0}_{\left( n-p+1\right)
\times \left( p-m-1\right) } & \lambda _{p}\left( \tilde{x}\right) \mathbb{I}%
_{n-p+1}%
\end{array}%
\right] .
\end{equation*}%
We will refer to a diagonal matrix having this form for any $m<p\leq n$ as a 
\emph{Grushin matrix function of type }$m$. Note that the comparability $%
A\left( x\right) \sim D_{\mathbf{\lambda }}\left( \tilde{x}\right) $ impies
that $a_{k,k}\left( x\right) \approx \lambda _{k}\left( \tilde{x}\right) $
for all the diagonal entries, so that $\lambda _{k}\left( \tilde{x}\right)
\approx a_{k,k}\left( \tilde{x},0\right) $ may be assumed smooth without
loss of generality. Moreover $A\left( x\right) \sim A_{\func{diag}}\left( 
\tilde{x},0\right) $ (see \cite[after Definition 10]{KoSa2}).

All of our theorems will apply to operators $L$ having a Grushin matrix
function $A\left( x\right) $ of type $m$ that is also \emph{elliptical} in
the sense that $A\left( x\right) $ is positive definite for $x\neq 0$.
Moreover, we will require in addition that the intermediate diagonal entries 
$\left\{ a_{k,k}\left( \tilde{x}\right) \right\} _{k=m+1}^{p-1}$ (there
won't be any such entries in the case $p=m+1$) are smooth and \emph{strongly}
$C^{4,2\delta }$ (see \cite{KoSa1}) for some $\delta >0$ (we show in \cite%
{KoSa2} that such functions can be written as a sum of squares of $%
C^{2,\delta }$ functions, and moreover give a sharp $\omega $-monotonicity
criterion for strongly $C^{4,2\delta }$), and that the off diagonal entries
of $A\left( x\right) $ satisfy certain strongly subordinate inequalities
(which are shown to be sharp in a certain case, see \cite[Theorem 42]{KoSa2}%
). We emphasize that no additional assumptions are made on the last $n-p+1$
entries of $D\left( \tilde{x}\right) $, which are all equal to $\lambda
_{p}\left( \tilde{x}\right) $.

Our approach is broadly divided into four separate steps, the first and
second of which are the subject of the first two papers in this series:

\begin{enumerate}
\item First, a proof that a $C^{3,1}$ function can be written as a finite
sum of squares of $C^{1,1}$ functions first appeared in Guan \cite{Gua}, who
attributed the result to Fefferman. In \cite{KoSa1} we adapted treatments of
this result from Tataru \cite{Tat} and Bony \cite{Bon} to establish
conditions under which a $C^{4,2\delta }$ nonnegative function can be
written as a finite sum of squares of $C^{2,\delta }$ functions for some $%
\delta >0$. The methods of Tataru and Bony were in turn modelled on a
localized splitting of a nonnegative symbol $a$, due to Fefferman and Phong 
\cite{FePh}, who used it to establish a strong form of G\aa rding's
inequality, and is the main idea behind the result of Fefferman appearing in 
\cite{Gua}. That splitting used the implicit function theorem to write a
nonnegative symbol $a$ as a sum of squares plus a symbol depending on fewer
variables, so that induction could be applied. This same scheme was used in 
\cite{KoSa1} to obtain a sum of squares of $C^{2,\delta }$ functions, but
taking care to arrange assumptions so that the implicit function theorem
applied.

\item Second, in \cite{KoSa2}, we showed that under analogous conditions on
the diagonal entries of a matrix-valued function $M$, and strong
subordinate-type inequalities on the off diagonal entries, $M$ can then be
written as a finite sum of squares of $C^{2,\delta }$ vector fields for some 
$\delta >0$.

\item Third, we here extend a theorem of M. Christ on hypoellipticity of
sums of smooth squares of vector fields to the setting of $C^{2,\delta }$
vector fields, with the appropriate notion of gain in a range of Sobolev
spaces.

\item Fourth, we here adapt arguments of M.\ Christ together with the above
steps to obtain hypoellipticity of linear operators $L$ of the form%
\begin{equation}
L=\nabla ^{\func{tr}}A\left( x\right) \nabla +D\left( x\right) ,
\label{L form}
\end{equation}%
where the matrix $A$ and scalar $D$ are smooth functions of $x\in \mathbb{R}%
^{n}$, and with $\tilde{x}=\left( x_{1},...,x_{m}\right) $, we have%
\begin{equation}
A\left( x\right) \sim \left[ 
\begin{array}{cc}
\mathbb{I}_{m} & 0 \\ 
0 & D_{\mathbf{\lambda }}\left( \tilde{x}\right)%
\end{array}%
\right] ,  \label{A comp}
\end{equation}%
where $\mathbb{I}_{m}$ is the $m\times m$ identity matrix, and $D_{\mathbf{\
\lambda }}\left( \tilde{x}\right) $ is the $\left( n-m\right) \times \left(
n-m\right) $ diagonal matrix with the components of $\mathbf{\lambda }\left( 
\tilde{x}\right) =\left( \lambda _{m+1}\left( \tilde{x}\right) ,...,\lambda
_{n}\left( \tilde{x}\right) \right) $ along the diagonal. The component
functions $\lambda _{\ell }\left( \tilde{x}\right) $ satisfy certain natural
conditions described explicitly below.
\end{enumerate}

We will end this section by stating our main results on hypoellipticity.
Then in the next section, we use a result on calculus of rough symbols from
the 1980's \cite{Saw} to derive a rough version of M. Christ's
hypoellipticity theorem for sums of smooth vector fields in the infinitely
degenerate regime, where symbol splitting is inadequate. Finally in the last
sections, we use a bootstrap argument that exploits the $C^{2,\delta }$
regularity of the vector fields, to bring all of these results to bear on
proving hypoellipticity for linear partial differential operators $L$ of the
form (\ref{L form}).

But first we recall the main results from the second paper in this series 
\cite{KoSa2} on sums of squares of matrix functions that we will use here.

\begin{definition}
\label{comp mat}Let $A$ and $B$ be real symmetric positive semidefinite $%
n\times n$ matrices. We define $A\preccurlyeq B$ if $B-A$ is positive
semidefinite. Let $\beta <\alpha $ be positive constants. A real symmetric
positive semidefinite $n\times n$ matrix $A$ is said to be $\left( \beta
,\alpha \right) $-\emph{comparable} to a symmetric $n\times n$ matrix $B$,
written $A\sim _{\beta ,\alpha }B$, if $\beta B\preccurlyeq A\preccurlyeq
\alpha \ B$, i.e.%
\begin{equation}
\beta \ \xi ^{\limfunc{tr}}B\xi \leq \xi ^{\limfunc{tr}}A\xi \leq \alpha \
\xi ^{\limfunc{tr}}B\xi ,\ \ \ \ \ \text{for all }\xi \in \mathbb{R}^{n}.
\label{pos def}
\end{equation}%
We say $A$ is \emph{comparable} to $B$, written $A\sim B$, \ if $A\sim
_{\beta ,\alpha }B$ for some $0<\beta <\alpha <\infty $.
\end{definition}

Note that if $A$ is \emph{comparable} to $B$, then both $A$ and $B$ are
positive semidefinite. Indeed, both $0\leq \left( \alpha -\beta \right) \xi
^{\limfunc{tr}}B\xi $ and $0\leq \left( \frac{1}{\beta }-\frac{1}{\alpha }%
\right) \xi ^{\limfunc{tr}}A\xi $ hold for all $\xi \in \mathbb{R}^{n}$.

\begin{definition}
A matrix function $\mathbf{A}\left( x\right) $ is \emph{subordinate} if $%
\left\vert \frac{\partial \mathbf{A}}{\partial x_{k}}\left( x\right) \cdot
\xi \right\vert ^{2}\leq C\xi ^{\func{tr}}A\left( x\right) \xi $ for all $%
\xi \in \mathbb{R}^{n}$, equivalently $\frac{\partial \mathbf{A}}{\partial
x_{k}}\left( x\right) ^{\func{tr}}\frac{\partial \mathbf{A}}{\partial x_{k}}%
\left( x\right) \preccurlyeq C\mathbf{A}\left( x\right) $.
\end{definition}

Finally recall the following seminorm from \cite{Bon}, 
\begin{equation}
\left[ h\right] _{\alpha ,\delta }\left( x\right) \equiv
\limsup_{y,z\rightarrow x}\frac{\left\vert D^{\alpha }h\left( y\right)
-D^{\alpha }h\left( z\right) \right\vert }{\left\vert y-z\right\vert
^{\delta }}.  \label{def mod D}
\end{equation}%
Here is the sum of squares decomposition with a quasiformal block of order $%
\left( n-p+1\right) \times \left( n-p+1\right) $, where $1<p\leq n$. We say
that a symmetric matrix function $\mathbf{Q}_{p}\left( x\right) $ is
quasiconformal if the eigenvalues $\lambda _{i}\left( x\right) $ of $\mathbf{%
Q}_{p}\left( x\right) $ are nonnegative and comparable.

\begin{theorem}
\label{final n Grushin}Let 
\begin{equation*}
1<p\leq n,\ \ \ \frac{1}{4}\leq \varepsilon <1,\ \ \ 0<\delta <\delta
^{\prime \prime }<1,\text{\ \ \ }M\geq 1,
\end{equation*}%
and%
\begin{equation*}
\delta ^{\prime }=\frac{2\delta \left( 1+\delta \right) }{2+\delta }.
\end{equation*}%
Suppose that $\mathbf{A}\left( x\right) $ is a $C^{4,2\delta }$ symmetric $%
n\times n$ matrix function of a variable $x\in \mathbb{R}^{M}$, which is
comparable to a diagonal matrix function $\mathbf{D}\left( x\right) $, hence
comparable to its associated diagonal matrix function $\mathbf{A}_{\func{diag%
}}\left( x\right) $.

\begin{enumerate}
\item Moreover, assume $a_{p,p}\left( x\right) \approx a_{p+1,p+1}\left(
x\right) \approx ...\approx a_{n,n}\left( x\right) $ and that the diagonal
entries $a_{1,1}\left( x\right) ,...,a_{p-1,p-1}\left( x\right) $ satisfy
the following differential estimates up to fourth order,%
\begin{eqnarray}
\left\vert D^{\mu }a_{k,k}\left( x\right) \right\vert &\lesssim
&a_{k,k}\left( x\right) ^{\left[ 1-\left\vert \mu \right\vert \varepsilon %
\right] _{+}+\delta ^{\prime }},\ \ \ \ \ \text{ }1\leq \left\vert \mu
\right\vert \leq 4\text{ and }1\leq k\leq p-1,  \label{diag hyp} \\
\left[ a_{k,k}\right] _{\mu ,2\delta }\left( x\right) &\lesssim &1,\ \ \ \ \ 
\text{ }\left\vert \mu \right\vert =4\text{ and }1\leq k\leq p-1.  \notag
\end{eqnarray}

\item Furthermore, assume the off diagonal entries $a_{k,j}\left( x\right) $
satisfy the following differential estimates up to fourth order, 
\begin{eqnarray}
\left\vert D^{\mu }a_{k,j}\right\vert &\lesssim &\left( \min_{1\leq s\leq
j}a_{s,s}\right) ^{\left[ \frac{1}{2}+\left( 2-\left\vert \mu \right\vert
\right) \varepsilon \right] _{+}+\delta ^{\prime \prime }},\ \ \ \ \ 0\leq
\left\vert \mu \right\vert \leq 4\text{ and }1\leq k<j\leq p-1,
\label{off diag hyp} \\
\left[ a_{k,j}\right] _{\mu ,2\delta } &\lesssim &1,\ \ \ \ \ \left\vert \mu
\right\vert =4\text{ and }1\leq k<j\leq p-1,  \notag \\
\left\vert D^{\mu }a_{k,j}\right\vert &\lesssim &\left( \min_{1\leq s\leq
k}a_{s,s}\right) ^{\left[ \frac{1}{2}+\left( 2-\left\vert \mu \right\vert
\right) \varepsilon \right] _{+}+\delta ^{\prime \prime }},\ \ \ \ \ 0\leq
\left\vert \mu \right\vert \leq 4\text{ and }1\leq k\leq p-1<j\leq n  \notag
\\
\left[ a_{k,j}\right] _{\mu ,2\delta } &\lesssim &1,\ \ \ \ \ \left\vert \mu
\right\vert =4\text{ and }1\leq k\leq p-1<j\leq n.  \notag
\end{eqnarray}

\item Then there is a positive integer $I\in \mathbb{N}$ such that the
matrix function $\mathbf{A}$ can be written as a finite sum of squares of $%
C^{2,\delta }$ vectors $X_{k,j}$, plus a matrix function $\mathbf{A}_{p}$, 
\begin{equation*}
\mathbf{A}\left( x\right) =\sum_{k=1}^{p-1}\sum_{i=1}^{I}X_{k,j}\left(
x\right) X_{k,j}\left( x\right) ^{\func{tr}}+\mathbf{A}_{p}\left( x\right)
,\ \ \ \ \ x\in \mathbb{R}^{M},
\end{equation*}%
where the vectors $X_{k,i}\left( x\right) ,\ 1\leq k\leq p-1,\ 1\leq i\leq I$
are $C^{2,\delta }\left( \mathbb{R}^{M}\right) $, $\mathbf{A}_{p}\left(
x\right) =\left[ 
\begin{array}{cc}
\mathbf{0} & \mathbf{0} \\ 
\mathbf{0} & \mathbf{Q}_{p}\left( x\right)%
\end{array}%
\right] $, and $\mathbf{Q}_{p}\left( x\right) \in C^{4,2\delta }\left( 
\mathbb{R}^{M}\right) $ is \emph{quasiconformal}. Moreover, $Z_{k}\equiv
\sum_{i=1}^{I}X_{k,i}X_{k,i}^{\func{tr}}\in C^{4,2\delta }\left( \mathbb{R}%
^{M}\right) $ and 
\begin{eqnarray}
ca_{k,k}\mathbf{e}_{k}\otimes \mathbf{e}_{k} &\prec &Z_{k}Z_{k}^{\func{tr}%
}+\sum_{m=k+1}^{n}a_{m,m}\mathbf{e}_{m}\otimes \mathbf{e}_{m}\prec
C\sum_{m=k}^{n}a_{m,m}\mathbf{e}_{m}\otimes \mathbf{e}_{m},\ \ \ \ \ 1\leq
k\leq p-1,  \label{more} \\
\mathbf{Q}_{p}\left( x\right) &\sim &a_{p,p}\left( x\right) \mathbb{I}%
_{n-p+1}\ .  \notag
\end{eqnarray}%
Finally, if in addition $\mathbf{A}\left( x\right) $ is subordinate, then $%
\mathbf{Q}_{p}\left( x\right) $ is also subordinate\footnote{%
A more general assumption is that of \emph{semisubordinaticity}, namely $%
\frac{\partial \mathbf{A}}{\partial x_{k}}=\mathbf{S}_{k}^{1}+\left( \mathbf{%
S}_{k}^{2}\right) ^{\func{tr}}$ where $\left[ \mathbf{S}_{k}^{j}\right] ^{%
\func{tr}}\mathbf{S}_{k}^{j}\preccurlyeq C\mathbf{A}$ for $j=1,2$ and $%
k=1,2,...,n$, whose importance arises from the fact that semisubordinaticity
of $\mathbf{Q}_{p}$ can be used in place of subordinaticity of $\mathbf{Q}%
_{p}$ in the proof of Theorem \ref{HYP G} below. However, the
semisubordinate condition is much harder to pass through the $1$-SD in \cite%
{KoSa2} than is the subordinate condition, and it is ultimately as difficult
to deal with as the sum of squares decomposition itself.}.
\end{enumerate}
\end{theorem}

\begin{remark}
If in addition $a_{k,k}\left( x\right) \approx 1$ for $1\leq k\leq m<p$,
then the conditions (\ref{diag hyp}) and (\ref{off diag hyp}) in $\left(
1\right) $ and $\left( 2\right) $ are vacuous for $1\leq k\leq m$, and
moreover the proof shows that the vectors $X_{k,i}$ are actually in $%
C^{4,2\delta }\left( \mathbb{R}^{M}\right) $ for $1\leq k\leq m,1\leq i\leq
I $.
\end{remark}

These remarks yield the following corollary in which conditions (\ref{diag
hyp}) and (\ref{off diag hyp}) in $\left( 1\right) $ and $\left( 2\right) $
play no role.

\begin{corollary}
Suppose $\mathbf{A}\left( x\right) $ is a $C^{4,\delta }\left( \mathbb{R}%
^{M}\right) $ symmetric $n\times n$ matrix function that is comparable to a
diagonal matrix function. In addition suppose that $a_{k,k}\left( x\right)
\approx 1$ for $1\leq k\leq p-1$ and $a_{k,k}\left( x\right) \approx
a_{p,p}\left( x\right) $ for $p\leq k\leq n$. Then%
\begin{equation*}
\mathbf{A}\left( x\right) =\sum_{k=1}^{p-1}X_{k}\left( x\right) X_{k}\left(
x\right) ^{\func{tr}}+\mathbf{Q}_{p}\left( x\right) ,\ \ \ \ \ x\in \mathbb{R%
}^{M},
\end{equation*}%
where $X_{k},\mathbf{Q}_{p}\in C^{4,\delta }\left( \mathbb{R}^{M}\right) $
and (\ref{more}) holds for $1\leq k\leq p-1$.
\end{corollary}

\begin{remark}
If the diagonal entry $a_{k,k}\left( x\right) $ is smooth and $\omega _{s}$%
-montone on $\mathbb{R}^{n}$ for some $s>1-\varepsilon $, then the diagonal
differential estimates (\ref{diag hyp}) above hold for $a_{k,k}\left(
x\right) $ since $\left\vert D^{\mu }a_{k,k}\left( x\right) \right\vert \leq
C_{s,s^{\prime }}a_{k,k}\left( x\right) ^{s^{\prime }}$ for any $s^{\prime
}<s$ (\cite[Theorem 18]{KoSa2}).
\end{remark}

\begin{remark}
If in Theorem \ref{final n Grushin}, we drop the hypothesis (\ref{diag hyp})
that the diagonal entries satisfy the differential estimates, and even
slightly weaken the off diagonal hypotheses (\ref{off diag hyp}), then using
the Fefferman-Phong theorem for sums of squares of scalar functions, the
proof of Theorem \ref{final n Grushin} shows that the operator $L=\nabla ^{%
\limfunc{tr}}\mathbf{A}\nabla $ can be written as $L=\sum_{j=1}^{N}X_{j}^{%
\func{tr}}X_{j}$ where the vector fields $X_{j}$ are $C^{1,1}$ for$\
j=1,2,...,N$. However, unlike the situation for scalar functions, the
example in Theorem 38 of \cite{KoSa2} shows that we \emph{cannot} dispense
entirely with the off diagonal hypotheses (\ref{off diag hyp}) in $\left(
2\right) $. Moreover, the space $C^{1,1}$ seems not to be sufficient for
gaining a positive degree $\delta $ of smoothness for solutions to a second
order operator, and so this result will neither be used nor proved here.
\end{remark}

In this paper we will apply the sums of squares representations for matrix
functions obtained in \cite{KoSa2} to a rough generalization of a theorem of
M. Christ, that then leads to our main hypoellipticity theorem via a
bootstrap argument.

\section{Statement of main hypoellipticity theorems}

We begin with the following general hypoellipticity theorem in the
infinitely degenerate regime as in Step (4) of the introduction. We
emphasize that we make no assumptions regarding the order of vanishing of
the matrix function $A\left( x\right) $ at the origin. Since we only
consider degeneracies at the origin, it is useful to make the following
definition.

\begin{definition}
We say that a $q\times q$ matrix function $f:\mathbb{R}^{n}\rightarrow 
\mathbb{R}^{q^{2}}$ on $\mathbb{R}^{n}$ is \emph{elliptical} if $f\left(
x\right) $ is positive definite for $x\neq 0$. A scalar function $f$
corresponds to the case $q=1$.
\end{definition}

\begin{theorem}
\label{HYP G}Suppose $1\leq m<p\leq n$. Let $L$ be a second order real
self-adjoint divergence form partial differential operator in $\mathbb{R}%
^{n} $ given by 
\begin{equation}
L=\nabla ^{\func{tr}}A\left( x\right) \nabla +D\left( x\right) ,
\label{L form' G}
\end{equation}%
where the matrix $A$ and scalar $D$ are smooth real functions of $x\in 
\mathbb{R}^{n}$, and $A\left( x\right) $ is subordinate, i.e. $\frac{%
\partial A}{\partial x_{k}}\nabla $ is subunit with respect to $\nabla ^{%
\func{tr}}A\left( x\right) \nabla $.

\begin{enumerate}
\item Suppose further that with $\tilde{x}=\left( x_{1},...,x_{m}\right) $
we have the following \emph{Grushin} assumption, 
\begin{equation}
A\left( x\right) \sim \left[ 
\begin{array}{cc}
\mathbb{I}_{m} & 0 \\ 
0 & D_{\mathbf{\lambda }}\left( \tilde{x}\right)%
\end{array}%
\right] ,  \label{A comp' G}
\end{equation}%
where $\mathbb{I}_{m}$ is the $m\times m$ identity matrix, and $D_{\mathbf{\
\lambda }}\left( \tilde{x}\right) $ is the $\left( n-m\right) \times \left(
n-m\right) $ diagonal matrix with the components of $\mathbf{\lambda }\left( 
\tilde{x}\right) =\left( \lambda _{m+1}\left( \tilde{x}\right) ,...,\lambda
_{n}\left( \tilde{x}\right) \right) $ along the diagonal, i.e. 
\begin{equation}
D_{\lambda }\left( \xi \right) =\left[ 
\begin{array}{cccc}
\lambda _{m+1}\left( \tilde{x}\right) & 0 & \cdots & 0 \\ 
0 & \lambda _{m+2}\left( \tilde{x}\right) & \ddots & \vdots \\ 
\vdots & \ddots & \ddots & 0 \\ 
0 & \cdots & 0 & \lambda _{n}\left( \tilde{x}\right)%
\end{array}%
\right] .  \label{def diag G}
\end{equation}

\begin{enumerate}
\item Moreover, we suppose that the component functions $\lambda _{\ell }$
are elliptical in $\mathbb{R}^{m}$, and $\lambda_{p}(\tilde{x})\approx
\lambda_{p+1}(\tilde{x})\approx \dots\approx \lambda_{n}(\tilde{x})$.

\item We also suppose that there are positive numbers $0<\delta <\delta
^{\prime \prime }<\frac{1}{2}$, $\frac{1}{4}\leq \varepsilon <1$, such that
for $\delta ^{\prime }=\frac{2\delta \left( 1+\delta \right) }{2+\delta }$
and for $k<j\leq n$ and $1\leq k\leq p-1$, the entries $a_{k,j}\left(
x\right) $ of $A\left( x\right) $ satisfy the differential size inequalities%
\footnote{%
The diagonal inequalities become more demanding the smaller $\varepsilon $
is, while the off diagonal inequalities become less demanding.} in (\ref%
{diag hyp}) and (\ref{off diag hyp}) for all $x\in \mathbb{R}^{n}$.
\end{enumerate}

\item Then $L$ is hypoelliptic if 
\begin{equation}
\lim_{\hat{x}\rightarrow 0}\mu (|\tilde{x}|,\sqrt{\max \left\{ \lambda
_{m+1},...,\lambda _{p}\right\} \left( \tilde{x}\right) })\ln \min \left\{
\lambda _{m+1},...,\lambda _{p}\right\} \left( \tilde{x}\right) =0,
\label{star}
\end{equation}%
where 
\begin{equation*}
\mu (t,g)\equiv \max \{g(z)(t-|z|):0\leq |z|\leq t\}.
\end{equation*}%
Moreover, condition (\ref{star}) is necessary for hypoellipticity if in
addition $A\left( x\right) $ is a diagonal matrix with monotone entries.
\end{enumerate}
\end{theorem}

\begin{remark}
Note that when $m=1$, it suffices to assume only \emph{smoothness} of the
diagonal entries $\lambda _{\ell }\left( \tilde{x}\right) $ in place of (\ref%
{diag hyp}), in view of Bony's sum of squares theorem \cite[Th\'{e}or\`{e}me
1]{Bon}.
\end{remark}

Here is a variation, without any special hypotheses on the diagonal entries,
that will be used to prove Theorem \ref{HYP G} in conjunction with the sum
of squares decomposition in Theorem \ref{final n Grushin}. However, the
proof of this next result will require a generalization of M. Christ's sum
of squares theorem to include $C^{2,\delta }$ vector fields.

\begin{theorem}
\label{HYPSOS G} Let $L$ be a real second order divergence form partial
differential operator in $\mathbb{R}^{n}$ satisfying (\ref{L form' G}). Let $%
1\leq m<p\leq n+1$, and write%
\begin{equation*}
x=\left( x_{1},...,x_{m},x_{m+1},...,x_{p-1},x_{p},...,x_{n}\right) =\left( 
\tilde{x},\check{x},\hat{x}\right) \in \mathbb{R}^{m}\times \mathbb{R}%
^{p-m-1}\times \mathbb{R}^{n-p+1},
\end{equation*}%
where the middle factor $\mathbb{R}^{p-m-1}$ vanishes if $p=m+1$, and the
final factor vanishes if $p=n+1$.

\begin{enumerate}
\item Suppose that there exist $C^{2,\delta }$ vector fields $X_{j}\left(
x\right) \in \limfunc{Op}\left( \mathcal{C}^{2,\delta }S_{1,0}^{1}\right) $
for $1\leq j\leq N$, and an $\left( n-p+1\right) \times \left( n-p+1\right) $
matrix function $\mathbf{Q}_{p}\left( x\right) \in C^{4,2\delta }$ that is
elliptical, quasiconformal and subordinate, such that 
\begin{equation*}
L=\left( \sum_{j=1}^{N}X_{j}^{\func{tr}}X_{j}+\hat{\nabla}^{\func{tr}}%
\mathbf{Q}_{p}\left( x\right) \hat{\nabla}\right)
+\sum_{j=1}^{N}A_{j}X_{j}+\sum_{j=1}^{N}X_{j}^{\func{tr}}\tilde{A}_{j}+A_{0},
\end{equation*}%
where $\hat{\nabla}=(\partial _{x_{p}},\dots ,\partial _{x_{n}})$ and $A_{j},%
\tilde{A}_{j}\in \limfunc{Op}(\mathcal{C}^{1,\delta }S_{1,0}^{0})$, $%
A_{0}\in \mathcal{O}_{\left( -\delta /2,\delta /2\right) }^{-\delta
/2+\varepsilon }$ for all $\varepsilon >0$.

\item Suppose further that there are elliptical scalar functions $\lambda
_{m+1}\left( \tilde{x}\right) ,...\lambda _{p}\left( \tilde{x}\right) \in
C^{2}\left( \mathbb{R}^{n}\right) $ with $0\leq \lambda _{j}\leq 1$ for all $%
j$, such that $\mathbf{Q}_{p}\left( x\right) \sim \lambda _{p}\left( \tilde{x%
}\right) \mathbb{I}_{n-p+1}$ and such that the following inequalities hold
for all Lipschitz functions $v$:%
\begin{eqnarray}
&&\sum_{k=1}^{m}\left\vert \partial _{x_{k}}v\right\vert
^{2}+\sum_{k=m+1}^{p-1}\lambda _{k}\left( \tilde{x}\right) \left\vert
\partial _{x_{k}}v\right\vert ^{2}\lesssim \sum_{j=1}^{N}\left\vert
X_{j}v\right\vert ^{2}+\lambda _{p}\left( \tilde{x}\right)
\sum_{k=p}^{n}\left\vert \partial _{x_{k}}v\right\vert ^{2},
\label{add_cond_gen G} \\
&&\sum_{j=1}^{N}\left\vert X_{j}v\right\vert ^{2}\lesssim
\sum_{k=1}^{m}\left\vert \partial _{x_{k}}v\right\vert
^{2}+\sum_{k=m+1}^{p-1}\lambda _{k}\left( \tilde{x}\right) \left\vert
\partial _{x_{k}}v\right\vert ^{2}+\lambda _{p}\left( \tilde{x}\right)
\sum_{k=p}^{n}\left\vert \partial _{x_{k}}v\right\vert ^{2}  \notag
\end{eqnarray}

\item Finally set%
\begin{equation*}
\Lambda _{\func{sum}}\left( \tilde{x}\right) \equiv \sum_{k=m+1}^{p}\lambda
_{k}\left( \tilde{x}\right) \text{ and }\Lambda _{\func{product}}\left( 
\tilde{x}\right) \equiv \dprod\limits_{k=m+1}^{p}\lambda _{k}\left( \tilde{x}%
\right) ,
\end{equation*}%
and define the Koike functional $\mu \left( t,g\right) $ for any function $%
g\left( \tilde{x}\right) $ by 
\begin{equation}
\mu \left( t,g\right) \equiv \max \{g(\tilde{x})(t-|\tilde{x}|):0\leq |%
\tilde{x}|\leq t\}.  \label{mu_def_nd G}
\end{equation}

\item Then the operator $L$ is hypoelliptic if 
\begin{equation}
\lim_{x\rightarrow 0}\mu (|\tilde{x}|,\sqrt{\Lambda _{\func{sum}}})\ln
\Lambda _{\func{product}}\left( \tilde{x}\right) =0.  \label{log assump_n' G}
\end{equation}%
This is sharp in the sense that (\ref{log assump_n' G}) holds if $L$ is both
hypoelliptic and \emph{diagonal with monotone entries}.
\end{enumerate}
\end{theorem}

%
%
%
%
%

Here is our rough version, in the setting of sums of squares of real vector
fields, of M. Christ's hypoellipticity theorem as needed in Step (3) of the
introduction. Note in particular that the vector fields $X_{j}$ appearing
below are only assumed to be $C^{2,\delta }$, while the sum of their squares 
$\sum_{j}X_{j}^{\func{tr}}X_{j}$ is assumed to be smooth.

\begin{theorem}
\label{main_thm_Grushin}Suppose $1\leq p\leq n$ and $N\geq 1$. Let $R\subset
T^{\ast }V$, the cotangent bundle of an open set $V\subset R^{n}$, be any
ray, and assume that the operator $L$ has the form 
\begin{equation}
L=\sum_{j=1}^{N}X_{j}^{\func{tr}}X_{j}+\sum_{j=1}^{N}A_{j}X_{j}+%
\sum_{j=1}^{N}X_{j}^{\func{tr}}\tilde{A}_{j}+R_{1}+A_{0}+\widehat{\nabla }^{%
\func{tr}}\cdot \mathbf{Q}_{p}\left( x\right) \widehat{\nabla },
\label{operator-Grushin}
\end{equation}%
where the vector fields $X_{j},\ j=1,2,...,N$ are $C^{2,\delta }\left( 
\mathbb{R}^{n}\right) $ differential operators, and $\mathbf{Q}_{p}\left(
x\right) $ is a $C^{4,2\delta }\left( \mathbb{R}^{m}\right) $ $\left(
n-p+1\right) \times \left( n-p+1\right) $ matrix that is subordinate and
quasiconformal, and $\widehat{\nabla }=(\partial _{x_{p}},\dots \partial
_{x_{n}})$.

\begin{enumerate}
\item Assume further that $\mathbf{Q}_{p}=\mathbf{Q}_{p}(x)\approx a(x)%
\mathbb{I}_{n-p+1}$ with $a\in C^{4,2\delta }\left( \mathbb{R}^{n}\right) $
elliptical, $L\in \limfunc{Op}(S_{1,0}^{2})$, $X_{j}\in \limfunc{Op}(%
\mathcal{C}^{2,\delta }S_{1,0}^{1})$ and $A_{j},\tilde{A}_{j}\in \limfunc{Op}%
(\mathcal{C}^{1,\delta }S_{1,0}^{0})$, $A_{0}\in \mathcal{O}_{\left( -\delta
/2,\delta /2\right) }^{-\delta /2+\varepsilon }$ for all $\varepsilon >0$,
in some conic neighbourhood $V$ of $R$.

\begin{enumerate}
\item In addition, assume $R_{1}=\sum_{k=1}^{n}S_{k}\Theta _{k}\circ 
\widehat{\nabla }$, where each $S_{k}\in C^{1,\delta }(\mathbb{R}^{m\times
m})$ is subunit with respect to $\mathbf{Q}_{p}$, and $\Theta _{k}=(\Theta
_{kp},\dots ,\Theta _{kn})$ is a multiplier of order zero.

\item Suppose there exists $w\in C^{\infty }$ satisfying $w(\xi )\rightarrow
\infty $ as $|\xi |\rightarrow \infty $ such that%
\begin{equation}
\int_{\mathbb{R}^{d}}w^{2}(\xi )|\hat{u}(\xi )|^{2}d\xi \leq
C\sum_{j}||X_{j}u||^{2}+C||\sqrt{a}\widehat{\nabla }u||^{2}+C||u||^{2}\quad
\forall \ u\in C_{0}^{1}(V),  \label{Christ_cond1-G}
\end{equation}

\item Finally, suppose that for each small conic neighborhood $\Gamma $ of $%
R $ there exist scalar valued symbols $\psi ,p\in S_{1,0}^{0}$ such that $%
\psi $ is everywhere nonnegative, $\psi $ does not depend on $\xi $ in $%
\Gamma $, $\psi \equiv 0$ in some smaller conic neighborhood of $R$, $\psi
\geq 1$ on $T^{\ast }V\backslash \Gamma $, $p\equiv 0$ in a conic
neighborhood of the closure of $\Gamma $, and such that for each $\delta >0$
there exists $C_{\delta }<\infty $ such that for any relatively compact open
subset $U\Subset V$ and for all $u\in C_{0}^{2}(U)$ and each index $i$,%
\begin{eqnarray}
\left\Vert \limfunc{Op}\left[ \log \langle \xi \rangle \{\psi ,\sigma
(X_{i})\}\right] u\right\Vert ^{2} &\leq &\delta \sum_{j}\left\Vert
X_{j}u\right\Vert ^{2}+\delta \left\Vert \sqrt{a}\widehat{\nabla }%
u\right\Vert ^{2}+C_{\delta }\left\Vert u\right\Vert ^{2}+C_{\delta
}\left\Vert \func{Op}(p)u\right\Vert _{H^{1}}^{2},  \label{superlog-G} \\
\left\Vert \sqrt{\mathbf{Q}_{p}}\,\limfunc{Op}\left[ \log \langle \xi
\rangle \{\psi ,\widehat{\xi }\}\right] u\right\Vert ^{2} &\leq &\delta
\sum_{j}\left\Vert X_{j}u\right\Vert ^{2}+\delta \left\Vert \sqrt{a}\widehat{%
\nabla }u\right\Vert ^{2}+C_{\delta }\left\Vert u\right\Vert ^{2}+C_{\delta
}\left\Vert \func{Op}(p)u\right\Vert _{H^{1}}^{2},  \notag
\end{eqnarray}%
where $\widehat{\xi }=\left( \xi _{p},\dots ,\xi _{n}\right) $.
\end{enumerate}

\item Then there exists $\gamma >0$ such that for any $u\in L_{loc}^{2}$ we
have $Lu\in H^{\gamma }(R)\implies u\in H^{\gamma }(R)$.
\end{enumerate}
\end{theorem}

\begin{remark}
\label{Grush-conj} The term $R_{1}$ arises from the conjugation of $\widehat{%
\nabla}\cdot \mathbf{Q}_{p}\left( x\right) \widehat{\nabla}$ by $\Lambda
_{s}=(1+|\xi|^{2})^{s/2}$, needed in the bootstrap procedure. Indeed, we
have denoting $q_{ij}=(\mathbf{Q}_{p})_{ij}$ 
\begin{equation*}
\Lambda _{s}\widehat{\nabla }\cdot \mathbf{Q}_{p}\left( x\right) \widehat{%
\nabla }\Lambda _{-s}-\widehat{\nabla }\cdot \mathbf{Q}_{p}\left( x\right) 
\widehat{\nabla }=\sum_{i,j=p}^{n}[\Lambda _{s},q_{ij}]\Lambda _{-s}\partial
_{x_{i}}\partial _{x_{j}}.
\end{equation*}%
Using rough pseudodifferential calculus we have 
\begin{equation*}
\sigma ([\Lambda _{s},q_{ij}]\Lambda _{-s}\partial _{x_{i}})=-i\sum_{|\alpha
|=1}D^{\alpha }q_{ij}\frac{\xi ^{\alpha }\xi _{i}}{\left\langle \xi
\right\rangle ^{2}}=-i\sum_{k=1}^{n}\partial _{x_{k}}q_{ij}\frac{\xi _{k}\xi
_{i}}{\left\langle \xi \right\rangle ^{2}}\ \ \ \ \ \func{mod}\mathcal{O}%
_{\left( -\delta ,\delta \right) }^{-\varepsilon }.
\end{equation*}%
Denoting 
\begin{equation*}
S_{k}=\partial _{x_{k}}\mathbf{Q}_{p},\ \ (\theta _{k}(\xi ))_{i}=-i\frac{%
\xi _{k}\xi _{i}}{\left\langle \xi \right\rangle ^{2}},
\end{equation*}%
we have that $R_{1}=\sum_{k=1}^{n}S_{k}\Theta _{k}\circ \widehat{\nabla }$
has the desired properties since $\mathbf{Q}_{p}$ is subordinate, and 
\begin{equation*}
\Lambda _{s}\widehat{\nabla }\cdot \mathbf{Q}_{p}\left( x\right) \widehat{%
\nabla }\Lambda _{-s}=\widehat{\nabla }\cdot \mathbf{Q}_{p}\left( x\right) 
\widehat{\nabla }+R_{1}\ \ \ \ \ \func{mod}\mathcal{O}_{\left( -\delta
,\delta \right) }^{-\varepsilon }.
\end{equation*}
\end{remark}

We end this section on statements of the main hypoellipticity theorems, by
outlining the four steps taken in order to get to the point where we can
apply Theorems \ref{final n Grushin}, \ref{HYPSOS G} and \ref%
{main_thm_Grushin}\ to obtain our hypoellipticity Theorem \ref{HYP G}.

\subsection{Summary of the steps\label{Sub six steps}}

Consider the operator $L=\nabla A\left( \hat{x}\right) \nabla +D\left(
x\right) $ with smooth coefficients.

\begin{enumerate}
\item We first apply Theorem \ref{final n Grushin} to write $\nabla A\left( 
\hat{x}\right) \nabla =\mathbf{X}^{\func{tr}}\mathbf{X}$ plus a
quasiconformal subordinate term $\widehat{\nabla }\cdot \mathbf{Q}_{p}\left(
x\right) \widehat{\nabla }$, where the vector fields $\mathbf{X}$ belong to $%
\mathcal{C}^{2,\delta }S_{1,0}^{1}$ for some $\delta >0$, and $\mathbf{Q}%
_{p}\in C^{4,\delta }$.

\item We then use the smooth pseudodifferential calculus to write%
\begin{equation*}
\Lambda _{s}L\Lambda _{-s}=L+\widehat{\nabla }\cdot \mathbf{Q}_{p}\left(
x\right) \widehat{\nabla }+V\mathbf{X}+\mathbf{X}^{\func{tr}}U+A_{0}\left(
x,\xi \right) +R_{1}
\end{equation*}%
where the pseudodifferential operators $V\mathbf{X},U\mathbf{X}\in \func{Op}%
\mathcal{C}^{1,\delta }S_{1,0}^{1}$ , and $R_{1}\in \func{Op}\mathcal{C}%
^{1,\delta }S_{1,0}^{1}$ is subunit with respect to the quasiconformal term,
and where $A_{0}\in \mathcal{C}^{0,\delta }S_{1,0}^{0}$.

\item We\ next show that the operator $L=\nabla A\left( \hat{x}\right)
\nabla +D\left( x\right) $ is hypoelliptic if and only if for every integer $%
s\in \mathbb{Z}$, there is $\gamma =\gamma \left( \left[ s\right] \right) >0$
depending only on the integer part $\left[ s\right] $ of $s$, such that 
\begin{equation*}
u\in H^{0}\text{ and }\Lambda _{s}L\Lambda _{-s}u\in H^{\gamma }\text{
implies }u\in H^{\gamma }\text{, for }0\leq \gamma \leq 1.
\end{equation*}

\item Finally, we apply Theorem \ref{main_thm_Grushin} and Theorem \ref%
{HYPSOS G} to obtain hypoellipticity of $L$.
\end{enumerate}

\begin{remark}
Note that if we apply \emph{symbol splitting} as in \cite{Tay} to the vector
fields $\mathbf{X}$ to obtain $\mathbf{X}=\mathbf{X}^{\natural }+\mathbf{X}%
^{\flat }$ where $\mathbf{X}^{\natural }\in \func{Op}S_{1,\eta }^{1}$ and $%
\mathbf{X}^{\flat }\in \func{Op}\mathcal{C}^{2,\delta }S_{1,\eta }^{1-\eta
\left( 2+\delta \right) }$, then the subunit property of the vector field $%
\mathbf{X}$ is \emph{not} inherited by the smooth vector field $\mathbf{X}%
^{\natural }$. Indeed, the definition of $\mathbf{X}^{\natural }$ shows that
it is obtained by applying a mollification of size $2^{-j\eta }$ to a
Littlewood-Paley projection onto frequencies of size $2^{j}$, and such
mollifications are not comparable when applied to infinitely degenerate
fields, even suitably away from the degeneracies.
\end{remark}

\section{A rough variant of M. Christ's theorem}

We now prove our extension of M. Christ's hypoellipticity theorem, namely
Theorem \ref{main_thm_Grushin}, to the case of a sum of squares of \emph{%
rough} vector fields, whose sum of squares is nevertheless smooth. We will
assume the rough symbols are in the classes $\mathcal{C}^{2,\delta
}S_{1,0}^{\alpha }$, but we could just as well formulate and prove a variant
for the symbol classes $\mathcal{C}^{2,\delta }S_{\rho ,\eta }^{\alpha }$,
which we leave for the interested reader, as we will not use such a variant
in our applications. The proof of this rough theorem is accomplished by
adapting the sum of squares argument of Christ \cite{Chr} in the smooth
case. For this we begin with some preliminaries.

\subsection{Preliminaries}

Here we recall definitions and properties of symbols, G\aa rding's
inequality, parametrices, rough symbols, and wave front sets.

\subsubsection{Symbols}

We begin by recalling in $\mathbb{R}^{n}$, the definition of symbols $%
S_{\rho ,\eta }^{m}$ from Stein \cite[Chapter VI]{Ste}, the definition of
symbols $S_{\rho ,\eta }^{m,k}$ and $S_{\rho ,\eta }^{m+}$ from Christ \cite%
{Chr}, and then some results on rough versions of the symbol classes $%
S_{\rho ,\eta }^{m}$ from \cite{Saw} and \cite{Tay}. See also Treves \cite%
{Tre} for symbols defined in open sets $\Omega \subset \mathbb{R}^{n}$.

\begin{definition}
Let $a\left( x,\xi \right) $ be a smooth function on $\mathbb{R}^{n}\times 
\mathbb{R}^{n}$, $0\leq \eta <\rho \leq 1$, and $-\infty <m<\infty $.

\begin{enumerate}
\item Define $a\in S_{\rho ,\eta }^{m}$ , referred to as a \emph{symbol} of
type $\left( \rho ,\eta \right) $ and order $m$, if%
\begin{equation}
\left\vert \partial _{x}^{\alpha }\partial _{\xi }^{\beta }a\left( x,\xi
\right) \right\vert \leq C_{\alpha ,\beta }\left\langle \xi \right\rangle
^{m-\rho \left\vert \beta \right\vert +\eta \left\vert \alpha \right\vert }\
\ \ \ \ x\in \mathbb{R}^{n},\xi \in \mathbb{R}^{n},\left( \alpha ,\beta
\right) \in \mathbb{Z}_{+}^{n}\times \mathbb{Z}_{+}^{n}.  \label{def symbol}
\end{equation}

\item Define $a\in S_{\rho ,\eta }^{m,k}$ if%
\begin{equation*}
\left\vert \partial _{x}^{\alpha }\partial _{\xi }^{\beta }a\left( x,\xi
\right) \right\vert \leq C_{\alpha ,\beta }\left\langle \xi \right\rangle
^{m-\rho \left\vert \beta \right\vert +\eta \left\vert \alpha \right\vert
}\left( \log \left\langle \xi \right\rangle \right) ^{k+\left\vert \alpha
\right\vert +\left\vert \beta \right\vert }.
\end{equation*}

\item Define 
\begin{equation*}
S_{\rho ,\eta }^{m+}\equiv \dbigcap\limits_{\varepsilon >0}S_{\rho
-\varepsilon ,\eta +\varepsilon }^{m,\varepsilon }\ ,\ \ \ \ \ m\in \mathbb{R%
}.
\end{equation*}
\end{enumerate}
\end{definition}

For a symbol $a\in S_{\rho ,\eta }^{m}$, the associated pseudodifferential
operator $A:\mathcal{S}\left( \mathbb{R}^{n}\right) \rightarrow \mathcal{S}%
\left( \mathbb{R}^{n}\right) $, also denoted by $A=\limfunc{Op}a$, is
defined on the space of rapidly decreasing functions $\mathcal{S}\left( 
\mathbb{R}^{n}\right) $ on $\mathbb{R}^{n}$ by%
\begin{equation}
Au\left( x\right) =\frac{1}{\left( 2\pi \right) ^{n}}\int_{\mathbb{R}%
^{n}}e^{ix\cdot \xi }a\left( x,\xi \right) \widehat{u}\left( \xi \right)
d\xi ,\ \ \ \ \ x\in \mathbb{R}^{n}.  \label{def op symbol}
\end{equation}%
It follows with some work (see e.g. \cite{Ste}) that $\limfunc{Op}a:\mathcal{%
S}\left( \mathbb{R}^{n}\right) \rightarrow \mathcal{S}\left( \mathbb{R}%
^{n}\right) $ is continuous, and moreover, if $a_{k}$ converges pointwise to 
$a$ on $\mathbb{R}^{n}$, and (\ref{def symbol}) holds for $a=a_{k}$
uniformly in $k$, then $a\in S_{\rho ,\eta }^{m}$ as well. By duality $%
\limfunc{Op}a:\mathcal{S}^{\prime }\left( \mathbb{R}^{n}\right) \rightarrow 
\mathcal{S}^{\prime }\left( \mathbb{R}^{n}\right) $ is a continuous map from
the space of tempered distributions $\mathcal{S}^{\prime }\left( \mathbb{R}%
^{n}\right) $ to itself, and the asymptotic formulas for adjoints and
compositions holds without restriction, e.g. if $a\in S_{\rho ,\eta
}^{m_{1}} $ and $b\in S_{\rho ,\eta }^{m_{2}}$, then $\limfunc{Op}a\circ 
\limfunc{Op}b=\limfunc{Op}\left( a\circ b\right) $ where for all $M\in 
\mathbb{N}$,%
\begin{eqnarray*}
a\circ b &=&\sum_{\ell =0}^{M}\frac{1}{i^{\ell }\ell !}\widehat{\nabla }%
_{\xi }^{\ell }a\cdot \nabla _{x}^{\ell }b+E_{M}; \\
\text{with }E_{M} &\in &S_{\rho ,\eta }^{m_{1}+m_{2}-M-1}.
\end{eqnarray*}%
It follows immediately from the definitions that the asymptotic formulas for
adjoints and compositions extend to the symbol classes $S_{\rho ,\eta }^{m+}$%
. For example, by uniqueness of the expansions, we have 
\begin{equation*}
E_{M}\in S_{\rho -\varepsilon ,\eta +\varepsilon }^{m_{1}+\varepsilon
+m_{2}+\varepsilon -M-1}\subset S_{\rho -2\varepsilon ,\eta +2\varepsilon
}^{m_{1}+m_{2}-M-1,2\varepsilon }
\end{equation*}%
for each $\varepsilon >0$, and so%
\begin{equation*}
E_{M}\in \dbigcap\limits_{\varepsilon >0}S_{\rho -2\varepsilon ,\eta
+2\varepsilon }^{m_{1}+m_{2}-M-1,2\varepsilon }=S_{\rho ,\eta
}^{m_{1}+m_{2}-M-1+}.
\end{equation*}%
Now $S_{\rho ,\eta }^{m+}\subset S_{\rho ,\eta }^{m,k}$, and it turns out
that for our purposes, we apply the pseudodifferential calculus to the
symbol classes $S_{\rho ,\eta }^{m+}$, as well as to the classes $S_{\rho
,\eta }^{m,k}$ that arise naturally from the hypotheses of the theorems. We
will not necessarily make explicit mention of this distinction in the sequel
however.

\subsubsection{Parametrices}

Let $a\left( x,\xi \right) \in S_{1,\eta }^{m}$ be elliptic of order $m$,
i.e. there are strictly positive continuous functions $\rho \left( x\right) $
and $c\left( x\right) $ in $\Omega $ such that the symbol $a\left( x,\xi
\right) $ satisfies%
\begin{equation*}
c\left( x\right) \left\vert \xi \right\vert ^{m}\leq \left\vert a\left(
x,\xi \right) \right\vert ,\ \ \ \ \ \xi \in \mathbb{R}^{n}\text{ with }%
\left\vert \xi \right\vert \geq \rho \left( x\right) ,x\in \Omega .
\end{equation*}

\begin{proposition}
\label{para}Let $a\left( x,\xi \right) \in S_{1,\eta }^{m}\left( \Omega
\right) $. If $a\left( x,\xi \right) $ is elliptic of order $m$, then there
is $b\left( x,\xi \right) \in S_{1,\eta }^{-m}$ such that $a\circ b=1$.
Conversely, if there is $b\left( x,\xi \right) \in S_{1,\eta }^{-m}$ such
that $a\circ b=1$, then $a\left( x,\xi \right) $ is elliptic of order $m$.
\end{proposition}

\begin{proof}
Determine recursively symbols $b_{j}$ from the relations%
\begin{eqnarray}
b_{0}\left( x,\xi \right) a\left( x,\xi \right) &=&1,  \label{recurse} \\
b_{j}\left( x,\xi \right) a\left( x,\xi \right) &=&-\sum_{1\leq \left\vert
\alpha \right\vert \leq j}\frac{1}{\alpha !}\partial _{\xi }^{\alpha
}a\left( x,\xi \right) D_{x}^{\alpha }b_{j-\left\vert \alpha \right\vert
}\left( x,\xi \right) ,\ \ \ \ \ j\geq 1,  \notag
\end{eqnarray}%
which make sense only for $\left\vert \xi \right\vert \geq \rho \left(
x\right) $. The first three such symbols are given by%
\begin{eqnarray*}
b_{0}\left( x,\xi \right) &=&\frac{1}{a\left( x,\xi \right) }, \\
b_{1}\left( x,\xi \right) &=&-b_{0}\left( x,\xi \right) \sum_{i=1}^{n}\frac{%
\partial }{\partial \xi _{i}}a\left( x,\xi \right) \frac{1}{i}\frac{\partial 
}{\partial x_{i}}b_{0}\left( x,\xi \right) =-\frac{1}{i}b_{0}\left( x,\xi
\right) \nabla _{\xi }a\left( x,\xi \right) \cdot \nabla _{x}b_{0}\left(
x,\xi \right) , \\
b_{2}\left( x,\xi \right) &=&-b_{0}\left( x,\xi \right) \sum_{i=1}^{n}\frac{%
\partial }{\partial \xi _{i}}a\left( x,\xi \right) \frac{1}{i}\frac{\partial 
}{\partial x_{i}}b_{1}\left( x,\xi \right) -b_{0}\left( x,\xi \right)
\sum_{\left\vert \alpha \right\vert =2}\frac{1}{\alpha !}\partial _{\xi
}^{\alpha }a\left( x,\xi \right) D_{x}^{\alpha }b_{0}\left( x,\xi \right) \\
&=&-\frac{1}{i}b_{0}\left( x,\xi \right) \nabla _{\xi }a\left( x,\xi \right)
\cdot \nabla _{x}b_{1}\left( x,\xi \right) -b_{0}\left( x,\xi \right) \frac{1%
}{2!}\nabla _{\xi }^{2}a\left( x,\xi \right) \cdot \nabla
_{x}^{2}b_{0}\left( x,\xi \right) .
\end{eqnarray*}%
To deal with the requirement that $\left\vert \xi \right\vert \geq \rho
\left( x\right) $, we select a monotone increasing sequence of continuous
functions $\rho _{j+1}\left( x\right) >\rho _{j}\left( x\right) >\rho \left(
x\right) $ and a sequence of smooth cutoff functions $\chi _{j}\left( x,\xi
\right) \in C^{\infty }\left( \Omega \times \mathbb{R}^{n}\right) $
satisfying%
\begin{equation*}
\chi _{j}\left( x,\xi \right) =\left\{ 
\begin{array}{ccc}
0 & \text{ if } & \left\vert \xi \right\vert \leq \rho _{j}\left( x\right)
\\ 
1 & \text{ if } & \left\vert \xi \right\vert \leq 2\rho _{j}\left( x\right)%
\end{array}%
\right. .
\end{equation*}%
One can easily prove by induction on $j$ that $\chi _{j}b_{j}\in
S^{-m-j}\left( \Omega \right) $, and moreover that for carefully chosen such 
$\chi _{j}$ the series\ $\sum_{j=1}^{\infty }\chi _{j}b_{j}$ converges in $%
S^{-m}\left( \Omega \right) $ to a symbol $b$ satisfying $a\circ b=1$.
Indeed, if $\left\{ K_{j}\right\} _{j=1}^{\infty }$is a standard exhausting
sequence of compact sets for $\Omega $, and if the constants $C_{\alpha
,\beta }^{\left( j\right) }\left( K_{i}\right) $ satisfy%
\begin{equation*}
\left\vert \partial _{\xi }^{\alpha }\partial _{x}^{\beta }\left( \chi
_{j}b_{j}\right) \right\vert \leq C_{\alpha ,\beta }^{\left( j\right)
}\left( K_{i}\right) \left\vert \xi \right\vert ^{-m-j-\left\vert \alpha
\right\vert },\ \ \ \ \ \text{for }x\in K_{i},\xi \in \mathbb{R}%
^{n}\setminus \left\{ 0\right\} ,
\end{equation*}%
then we need only require in addition that $\rho _{j}\left( x\right) \geq
2\sup_{i\leq j,\left\vert \alpha +\beta \right\vert \leq j}C_{\alpha ,\beta
}^{\left( j\right) }\left( K_{i}\right) ^{\frac{1}{j}}$.

The converse is an easy exercise using only the consequence%
\begin{equation*}
a\left( x,\xi \right) b\left( x,\xi \right) -1\in S^{-1}\left( \Omega
\right) ,
\end{equation*}%
which implies that for every compact set $K\subset \Omega $, there is a
constant $C_{K}$ such that%
\begin{equation*}
\left\vert a\left( x,\xi \right) b\left( x,\xi \right) -1\right\vert \leq
C_{K}\frac{1}{1+\left\vert \xi \right\vert }.
\end{equation*}
\end{proof}

\begin{corollary}
Let $A$ belong to $S_{1,0}^{m}\left( \Omega \right) $. Then $A$ is elliptic
of order $m$ if and only if there is $B\in S_{1,0}^{-m}\left( \Omega \right) 
$ with 
\begin{equation*}
AB=BA=I\ \ \ \ \ \func{mod}S^{-\infty }\left( \Omega \right) ,
\end{equation*}%
where $S^{-\infty }\left( \Omega \right) =\dbigcap\limits_{m\in \mathbb{R}%
}S_{1,0}^{-m}\left( \Omega \right) $.
\end{corollary}

\subsubsection{Rough symbols}

The following definitions are taken from \cite{Tay} and \cite{Saw}.

\begin{definition}
A symbol $\sigma :\mathbb{R}^{n}\times \mathbb{R}^{n}\rightarrow \mathbb{R}$
belongs to the rough symbol class $\mathcal{C}^{M}S_{\rho ,\delta }^{m}$
(where $M\in \mathbb{Z}_{+}$ and $0\leq \rho ,\delta \leq 1$) if for all
multiindices $\alpha ,\beta $ with $\left\vert \alpha \right\vert \leq M$,
there are constants $C_{\alpha ,\beta }$ such that 
\begin{equation*}
\left\vert D_{x}^{\alpha }D_{\xi }^{\beta }\sigma \left( x,\xi \right)
\right\vert \leq C_{\alpha ,\beta }\left( 1+\left\vert \xi \right\vert
\right) ^{m+\delta \left\vert \alpha \right\vert -\rho \left\vert \beta
\right\vert },\ \ \ \ \ x\in \mathbb{R}^{n},\xi \in \mathbb{R}^{n}.
\end{equation*}%
If $0<\mu <1$, then $\sigma \in \mathcal{C}^{M+\mu }S_{\rho ,\delta }^{m}$
if in addition we have 
\begin{equation*}
\left\vert D_{\xi }^{\beta }\sigma \left( x+h,\xi \right) -\sum_{\ell =0}^{M}%
\frac{\left( h\cdot \nabla _{x}\right) ^{\ell }}{\ell !}D_{\xi }^{\beta
}\sigma \left( x,\xi \right) \right\vert \leq C_{M,\beta }\left\vert
h\right\vert ^{M+\mu }\left( 1+\left\vert \xi \right\vert \right) ^{m+\delta
\left( M+\mu \right) -\rho \left\vert \beta \right\vert },\ \ \ \ \ x\in 
\mathbb{R}^{n},\xi \in \mathbb{R}^{n}.
\end{equation*}
\end{definition}

\begin{definition}
A symbol $\sigma :\mathbb{R}^{n}\times \mathbb{R}^{n}\rightarrow \mathbb{R}$
belongs to the operator class $\mathcal{O}_{I}^{m}$ if its associated
operator 
\begin{equation*}
\left( \limfunc{Op}\sigma \right) u\left( x\right) =\frac{1}{\left( 2\pi
\right) ^{n}}\int_{\mathbb{R}^{n}}e^{ix\cdot \xi }\sigma \left( x,\xi
\right) \widehat{u}\left( \xi \right) d\xi ,\ \ \ \ \ x\in \mathbb{R}^{n},
\end{equation*}%
admits a bounded extension from $H_{p,\limfunc{comp}}^{s+m}$ to $H_{p,%
\limfunc{loc}}^{s}$ (resp. $\Lambda _{p,\limfunc{comp}}^{s+m}$ to $\Lambda
_{p,\limfunc{loc}}^{s}$) for $s\in I$ (resp. $s\in I\cap \left[ 0,\infty
\right) $ and all $1<p<\infty $.\newline
The symbol $\sigma $ belongs to the operator class $\overline{\mathcal{O}}%
_{I}^{m}$ if in addition $\limfunc{Op}\sigma $ is bounded from $\Lambda _{p,%
\limfunc{comp}}^{t+m}$ to $\Lambda _{p,\limfunc{loc}}^{t}$ where $t$ is the
right endpoint of the interval $I$.\newline
Here the subscript $\limfunc{comp}$ means compactly supported distributions
in the space, while the subscript $\limfunc{loc}$ means distributions
locally in the space.
\end{definition}

The following result of Bourdaud is well known, see also \cite[Section 2.1]%
{Tay} and \cite[Theorem 3]{Saw}.

\begin{theorem}[{\protect\cite[Bou]{Bou}}]
\label{rough bound}For all real $m$, and all $\nu >0$ and $0\leq \delta <1$
we have%
\begin{equation*}
\mathcal{C}^{\nu }S_{1,\delta }^{m}\subset \overline{\mathcal{O}}_{\left(
-\left( 1-\delta \right) \nu ,\nu \right) }^{m}.
\end{equation*}
\end{theorem}

\subsubsection{Rough pseudodifferential calculus}

While symbol smoothing is a very effective and relatively simple tool for
use in elliptic and finite type situations, it fails to sufficiently
preserve the subunit property of vector fields in the infinitely degenerate
regime. For this reason we will instead use the pseudodifferential calculus
from \cite{Saw}, to which we now turn.

If $\sigma \in \mathcal{C}^{\nu }S_{1,\delta _{1}}^{m_{1}}$ and $\tau \in 
\mathcal{C}^{M+\mu +\nu }S_{1,\delta _{2}}^{m_{2}}$ have compact support in $%
\mathbb{R}^{n}\times \mathbb{R}^{n}$, then the composition $\limfunc{Op}%
\sigma \circ \limfunc{Op}\tau $ of the operators $\limfunc{Op}\sigma $ and $%
\limfunc{Op}\tau $ equals the operator $\limfunc{Op}\left( \sigma \circ \tau
\right) $ where%
\begin{equation*}
\left( \sigma \circ \tau \right) \left( x,\eta \right) \equiv \int_{\mathbb{R%
}^{n}}\int_{\mathbb{R}^{n}}e^{i\left( x-y\right) \cdot \left( \xi -\eta
\right) }\sigma \left( x,\xi \right) \tau \left( y,\eta \right) dyd\xi ,
\end{equation*}%
and the double integral on the right hand side is absolutely convergent
under the compact support assumption, thus justifying the claim. Given such
symbols without the assumption of compact support, we may then consider
instead the symbols $\sigma _{\varepsilon }$ and $\tau _{\varepsilon }$
where $a_{\varepsilon }\left( x,\xi \right) \equiv \psi \left( \varepsilon
x,\varepsilon \xi \right) a\left( x,\xi \right) $. Provided $\psi \in
C_{c}^{\infty }\left( \mathbb{R}^{n}\times \mathbb{R}^{n}\right) $ is $1$ on
the unit ball, the symbols $a_{\varepsilon }$ are uniformly in the same
symbol class as $a$, and hence the above formula persists in the limit when
the operators are restricted to acting on the space $\mathcal{S}$ of rapidly
decreasing functions. Of course it may happen that the resulting symbol $%
\sigma \circ \tau $ fails to belong to any reasonable rough symbol class $%
\mathcal{C}^{M+\mu }S_{\rho ,\delta }^{m}$ - see \cite[Subsection 5.3]{Saw}.
Nevertheless, we have the following useful symbol expansion of $\sigma \circ
\tau $ valid up to an error operator in an appropriate class $\overline{%
\mathcal{O}}_{I}^{m}$.

\begin{theorem}
\label{Saw calc}(\cite[Theorem 4]{Saw}) Suppose $\sigma \in \mathcal{C}^{\nu
}S_{1,\delta _{1}}^{m_{1}}$ and $\tau \in \mathcal{C}^{M+\mu +\nu
}S_{1,\delta _{2}}^{m_{2}}$ where $M$ is a nonnegative integer, $0<\mu
,\delta _{1},\delta _{2}<1$, $\nu >0$ and $M+\mu \geq m_{1}\geq 0$. Let $%
\delta \equiv \max \left\{ \delta _{1},\delta _{2}\right\} $. Then%
\begin{eqnarray*}
\sigma \circ \tau &=&\sum_{\ell =0}^{M}\frac{1}{i^{\ell }\ell !}\nabla _{\xi
}^{\ell }\sigma \cdot \nabla _{x}^{\ell }\tau +E; \\
E &\in &\mathcal{O}_{\left( -\left( 1-\delta \right) \nu ,\nu \right)
}^{m_{1}+m_{2}+\left( M+\mu \right) \left( \delta _{2}-1\right) +\varepsilon
},\ \ \ \ \ \text{for every }\varepsilon >0.
\end{eqnarray*}
\end{theorem}

There is an analogous expansion for the symbol of the adjoint operator $%
\left( \limfunc{Op}\sigma \right) ^{\func{tr}}$.

\subsubsection{Smooth distributions and wave front sets}

The following definitions are taken from\ Treves \cite{Tre}.

\begin{definition}
A distribution $u$ in an open\ set $\Omega \subset \mathbb{R}^{n}$ is said
to be $C^{\infty }$ in some neighbourhood of a point $\left( x_{0},\xi
^{0}\right) \in \Omega \times \left( \mathbb{R}^{n}\setminus \left\{
0\right\} \right) $ if there is a function $g\in C_{c}^{\infty }\left( 
\mathbb{R}^{n}\right) $ equal to $1$ in a neighbourhood of $x_{0}$, and an
open cone $\Gamma ^{0}\subset \mathbb{R}^{n}$ containing $\xi ^{0}$ such
that for every $M>0$ there is a positive constant $C_{M}$ satisfying%
\begin{equation*}
\left\vert \widehat{gu}\left( \xi \right) \right\vert \leq C_{M}\left(
1+\left\vert \xi \right\vert \right) ^{-M},\ \ \ \ \ \xi \in \Gamma ^{0}.
\end{equation*}
\end{definition}

\begin{definition}
A distribution $u$ in an open\ set $\Omega \subset \mathbb{R}^{n}$ is said
to be $C^{\infty }$ in a conic open subset $\Gamma \subset \Omega \times
\left( \mathbb{R}^{n}\setminus \left\{ 0\right\} \right) $ if it is $%
C^{\infty }$ in some neighbourhood of every point of $\Gamma $. The wave
front set $WF\left( u\right) $ of $u$ is the complement in $\Omega \times
\left( \mathbb{R}^{n}\setminus \left\{ 0\right\} \right) $ of the union of
all conic open sets in which $u$ is $C^{\infty }$:%
\begin{equation*}
WF\left( u\right) \equiv \Omega \times \left( \mathbb{R}^{n}\setminus
\left\{ 0\right\} \right) \setminus \dbigcup \left\{ \Gamma \text{ conic open%
}\subset \Omega \times \left( \mathbb{R}^{n}\setminus \left\{ 0\right\}
\right) :u\text{ is }C^{\infty }\text{ in }\Gamma \right\} .
\end{equation*}%
For $\gamma \in \mathbb{R}$, the $H^{\gamma }$ wave front set of $u$ is
defined analogously, where $H^{\gamma }$ is the Sobolev space of order $%
\gamma $.
\end{definition}

\subsection{Proof of Theorem \protect\ref{main_thm_Grushin}, the limited
smoothness variant of Christ's theorem}

Now we can begin our proof of the limited smoothness variant Theorem \ref%
{main_thm_Grushin}, in the setting of real vector fields, of M. Christ's
theorem. Let $u\in \mathcal{D}^{\prime }\left( V\right) $ and $0<\gamma
<\delta $ be given. Suppose that the $H^{\gamma }$ wave front set of $Lu$\
is disjoint from some open conic neighbourhood $\Gamma _{0}$ of a point $%
\left( x_{0},\xi _{0}\right) \in T^{\ast }V$. Without loss of generality we
may assume that $u\in \mathcal{E}^{\prime }\left( V\right) $. Fix an integer 
$K\in \mathbb{Z}$ (possibly quite large) such that $u\in H^{-K}$. We will
show that $\left( x_{0},\xi _{0}\right) \notin WF_{H^{\gamma }}\left(
u\right) $ by first constructing a pseudodifferential operator $\Lambda $,
that is elliptic of order $\gamma $ in a conic neighbourhood of $\left(
x_{0},\xi _{0}\right) $, and then showing that $\Lambda u\in H^{0}\left( 
\mathbb{R}^{d}\right) $.

To do this, let $\psi $ be as in part (1) (c) of Theorem \ref%
{main_thm_Grushin}. Recall the definitions of the symbol classes $S_{\rho
,\eta }^{m}$, $S_{\rho ,\eta }^{m,k}$ and $S_{\rho ,\eta }^{m+}$: 
\begin{eqnarray*}
a &\in &S_{\rho ,\eta }^{m}\text{ if }\left\vert \partial _{x}^{\alpha
}\partial _{\xi }^{\beta }a\left( x,\xi \right) \right\vert \leq C_{\alpha
,\beta }\left\langle \xi \right\rangle ^{m-\rho \left\vert \beta \right\vert
+\eta \left\vert \alpha \right\vert }, \\
a &\in &S_{\rho ,\eta }^{m,k}\text{ if }\left\vert \partial _{x}^{\alpha
}\partial _{\xi }^{\beta }a\left( x,\xi \right) \right\vert \leq C_{\alpha
,\beta }\left\langle \xi \right\rangle ^{m-\rho \left\vert \beta \right\vert
+\eta \left\vert \alpha \right\vert }\left( \log \left\langle \xi
\right\rangle \right) ^{k+\left\vert \alpha \right\vert +\left\vert \beta
\right\vert }, \\
S_{\rho ,\eta }^{m+} &\equiv &\dbigcap\limits_{\varepsilon >0}S_{\rho
-\varepsilon ,\eta +\varepsilon }^{m,\varepsilon }\ .
\end{eqnarray*}%
Then following Christ we define a symbol of nonconstant order, depending on
parameters $\gamma $ and $N_{0}$ by 
\begin{equation}
\lambda \left( x,\xi \right) =\left\{ 
\begin{array}{ccc}
\left\vert \xi \right\vert ^{\gamma }e^{-N_{0}\left( \log \left\vert \xi
\right\vert \right) \psi \left( x,\xi \right) } & \text{ if } & \left\vert
\xi \right\vert \geq e \\ 
C^{\infty }\text{ and nonvanishing} & \text{ if } & \left\vert \xi
\right\vert <e%
\end{array}%
\right. \ .  \label{def lambda}
\end{equation}%
The nonnegativity of $\psi $ implies that $\lambda \in S_{1,0}^{\gamma +}$.
Moreover, $\lambda \in S_{1,0}^{\gamma ,0}$. With $\gamma $ fixed, there
exists $\theta >0$ such that for each $N_{0}$, we have $\lambda \in
S_{1,0}^{-\theta N_{0}+}$ on the closure of the complement of $\Gamma _{1}$.
Now choose $N_{0}$ so that $-\theta N_{0}<-K$. Then with 
\begin{equation*}
\Lambda =\limfunc{Op}\left( \lambda \right) ,
\end{equation*}%
we have $\Lambda u\in H^{-K+\theta N_{0}}\subset H^{0}$ microlocally on the
complement of $\Gamma _{1}$.

Define cutoff functions $\eta _{1},\eta _{2}\in C_{c}\left( \mathbb{R}
^{d}\right) $ such that $\eta _{2}\equiv 1$ in a neighbourhood of the
support of $u$, $\eta _{1}\equiv 1$ in a neighbourhood of the support of $%
\eta _{2}$, and $\func{Supp}\eta _{1}\subset V$.

Recall that if $a\in S_{\rho ,\eta }^{m}$ and $b\in S_{\rho ,\eta }^{n}$,
and $\rho >\eta $, then $\limfunc{Op}\left( a\right) \circ \limfunc{Op}
\left( b\right) $ has a symbol $a\odot b$ with an asymptotic expansion 
\begin{equation}
a\odot b\left( x,\xi \right) \sim \sum_{\alpha }c_{\alpha }\ \partial _{\xi
}^{\alpha }a\left( x,\xi \right) \ \partial _{x}^{\alpha }b\left( x,\xi
\right) ,\ \ \ \ \ c_{\alpha }=\frac{\left( -i\right) ^{\alpha }}{\alpha !}.
\label{asymp expan}
\end{equation}
The notation $\sim $ means that for every $N$, the operator 
\begin{equation*}
\limfunc{Op}\left( a\right) \circ \limfunc{Op}\left( b\right) -\limfunc{Op}
\left( \sum_{\alpha <N}c_{\alpha }\ \partial _{\xi }^{\alpha }a\left( x,\xi
\right) \ \partial _{x}^{\alpha }b\left( x,\xi \right) \right)
\end{equation*}
is smoothing of order $m+n-N\left( \rho -\eta \right) $ in the scale of
Sobolev spaces. The next lemma is taken verbatim from \cite{Chr}, as it
involves only symbols of type $\left( 1,0\right) $.

\begin{lemma}[Lemma 4.1 in \protect\cite{Chr}]
There exists an operator $\Lambda ^{-1}\in S_{1,0}^{m+}$ for some $m=m\left(
\gamma \right) $ depending on $\gamma $, such that $\Lambda \circ \Lambda
^{-1}-$ is smoothing of infinite order. Moreover, such an operator may be
constructed with a symbol of the form 
\begin{equation*}
\left( 1+f\right) \lambda ^{-1},\ \ \ \ \ f\in S_{1,0}^{-1,2}.
\end{equation*}
\end{lemma}

\begin{proof}
Write $f\sim \sum_{k=1}^{\infty }f_{k}$. Solve the equation 
\begin{equation*}
\lambda \odot \left[ \left( 1+f\right) \lambda ^{-1}\right] \sim 1
\end{equation*}
using the asymptotic expansion (\ref{asymp expan}) and the usual iterative
procedure as given in (\ref{recurse}). One obtains $f_{1}\in S^{-1,2}$, and
by induction, each $f_{k}\in S_{1,0}^{-k+}$. Choose $\Lambda $ to be an
operator whose full symbol has expansion $\sum_{k=1}^{\infty }f_{k}$, so
that the error is smoothing of all orders in the scale of Sobolev spaces.
\end{proof}

To prove an analogue of Lemma 4.2 in \cite{Chr} we will need an auxiliary
lemma.

\begin{lemma}
\label{lem:aux} Let $P\in Op(C^{\nu }S_{1,0}^{m,l})$ where $m,l\in \mathbb{N}
$, and let $\Lambda $ be the operator in (\ref{def lambda}), where where we
recall that $\psi $ is everywhere nonnegative, vanishes identically in a
small conic neighbourhood of $\left( x_{0},\xi _{0}\right) $, and is
strictly positive on the complement of $\Gamma _{1}$. Then 
\begin{equation*}
\Lambda P\Lambda ^{-1}=P+R_{1}+R_{2}+E,
\end{equation*}%
where $R_{1}\in Op(C^{\nu -1}S_{1,0}^{m-1,l+1})$, $R_{2}\in Op(C^{\nu
-2}S_{1,0}^{m-2,l+2})$, and $E\in \mathcal{O}_{\left( -\nu ,\nu \right)
}^{m-M-\varepsilon }$ for every $m\leq M<\nu $ and some $0<\varepsilon <1$.
Moreover, the operator $R_{1}$ has the form 
\begin{equation*}
R_{1}=\limfunc{Op}\left( \left\{ \log \lambda ,\sigma \left( P\right)
\right\} \right) .
\end{equation*}
\end{lemma}

\begin{proof}
Using Theorem \ref{Saw calc} we see that the symbol of $\Lambda P-P\Lambda $
divided by $\lambda $ equals 
\begin{eqnarray*}
&&\frac{1}{\lambda }\left\{ \lambda \odot \sigma \left( P\right) -\sigma
\left( P\right) \odot \lambda \right\} =\left\{ \sum_{\left\vert \alpha
\right\vert =1}+\sum_{2\leq \left\vert \alpha \right\vert \leq M}\right\}
c_{\alpha }\ \left[ \frac{\partial _{\xi }^{\alpha }\lambda }{\lambda }\
\partial _{x}^{\alpha }\sigma \left( P\right) -\partial _{\xi }^{\alpha
}\sigma \left( P\right) \ \frac{\partial _{x}^{\alpha }\lambda }{\lambda }%
\right] +E \\
&=&\sum_{\left\vert \alpha \right\vert =1}c_{\alpha }\ \left[ \partial _{\xi
}^{\alpha }\log \lambda \ \partial _{x}^{\alpha }\sigma \left( P\right)
-\partial _{\xi }^{\alpha }\sigma \left( P\right) \ \partial _{x}^{\alpha
}\log \lambda \right] +\text{symbol in }C^{\nu -M}S_{1,0}^{m-2,l+2}+E \\
&=&\left\{ \log \lambda ,\sigma \left( P\right) \right\} +\text{symbol in }%
C^{\nu -M}S_{1,0}^{m-2,l+2}+E\ ,
\end{eqnarray*}%
where $E\in \mathcal{O}_{\left( -\nu ,\nu \right) }^{m-M-\varepsilon }$ for
every $M<\nu $ and some $0<\varepsilon <1$; and $\left\{ \log \lambda
,\sigma \left( P\right) \right\} $ is the Poisson bracket of $\log \lambda $
and $\sigma \left( P\right) $, and is a symbol in $\mathcal{C}^{\nu
-1}S_{1,0}^{m-1,l+1}$.
\end{proof}

Define 
\begin{equation}
L_{1}=\sum_{j}X_{j}^{\func{tr}}X_{j}+\sum_{j}A_{j}X_{j}+\sum_{j}X_{j}^{\func{%
tr}}\tilde{A}_{j}+A_{0},  \label{L1-def}
\end{equation}%
so that $L=L_{1}+R_{1}+\widehat{\nabla }\cdot \mathbf{Q}_{p}\widehat{\nabla }
$. This next lemma is our first analogue of Lemma 4.5 in \cite{Chr}.

\begin{lemma}[Lemma 4.5 in \protect\cite{Chr}]
\label{lem form G} Let $\Lambda $ be the operator with symbol $\lambda $ in (%
\ref{def lambda}). Suppose that $L_{1}$ takes the form (\ref{L1-def}).
Define 
\begin{equation*}
b_{j}\equiv \limfunc{Op}\left\{ \log \lambda ,\sigma \left( X_{j}^{\func{tr}%
}\right) \right\} \text{ and }\widetilde{b_{j}}\equiv \limfunc{Op}\left\{
\log \lambda ,\sigma \left( X_{j}\right) \right\} .
\end{equation*}%
Then there exists a pseudodifferential operator $G$ of the form 
\begin{equation}
G=\sum_{j}B_{j}\circ X_{j}+\sum_{j}X_{j}^{\func{tr}}\circ \widetilde{B}%
_{j}+B_{0}\ ,  \label{form G}
\end{equation}%
such that 
\begin{equation*}
\left( L_{1}+G\right) \eta _{1}\Lambda \eta _{2}=\eta _{1}\Lambda L\eta
_{2}+R\ ,
\end{equation*}%
where 
\begin{eqnarray}
B_{j} &=&b_{j}+c_{j}\text{ and }\widetilde{B}_{j}=\widetilde{b}_{j}+%
\widetilde{c}_{j}\text{ for every }j\geq 1,  \label{B def} \\
B_{0} &=&\sum_{j}\left( b_{j}\circ \widetilde{b}_{j}+A_{j}\widetilde{b}_{j}+%
\widetilde{A}_{j}b_{j}\right) \ \func{mod}\limfunc{Op}\left( C^{0,\delta
}S_{1,0}^{-1,2}\right) \ ,  \notag
\end{eqnarray}%
where each $c_{j},\widetilde{c}_{j}\in \limfunc{Op}\left( C^{0,\delta
}S^{-1,1}\right) $, and where $A_{j},\widetilde{A}_{j}\in C^{1,\delta
}S_{1,0}^{0}$ are the coefficents of the differential operator $L$ in (\ref%
{L1-def}), and $R\in \mathcal{O}_{\left( -\delta ,\delta \right)
}^{-\varepsilon }$.
\end{lemma}

\begin{proof}
In constructing the symbol of $G$ we will work formally, ignoring the cutoff
functions $\eta _{1}$ and $\eta _{2}$. This is permissible by pseudolocality
since $\eta _{1}\eta _{2}=\eta _{2}$. The desired equation $\left(
L_{1}+G\right) \Lambda =\Lambda L+R$ is then equivalent to 
\begin{eqnarray*}
G &=&\Lambda L_{1}\Lambda ^{-1}-L_{1}+R\Lambda ^{-1} \\
&=&\sum_{j}\left[ \Lambda X_{j}^{\func{tr}}X_{j}\Lambda ^{-1}-X_{j}^{\func{tr%
}}X_{j}\right] +R\Lambda ^{-1} \\
&&+\sum_{j}\Lambda \left( A_{j}X_{j}+X_{j}^{\func{tr}}\widetilde{A_{j}}%
+A_{0}\right) \Lambda ^{-1}-\sum_{j}\left( A_{j}X_{j}+X_{j}^{\func{tr}}%
\widetilde{A_{j}}+A_{0}\right) \\
&\equiv &G_{\limfunc{top}}+\Lambda G_{\limfunc{lower}}\Lambda ^{-1}-G_{%
\limfunc{lower}}; \\
\text{where }G_{\limfunc{top}} &=&\sum_{j}\left[ \Lambda X_{j}^{\func{tr}%
}X_{j}\Lambda ^{-1}-X_{j}^{\func{tr}}X_{j}\right] +R\Lambda ^{-1} \\
&=&\sum_{j}\left[ \left( \Lambda X_{j}^{\func{tr}}\Lambda ^{-1}\right)
\left( \Lambda X_{j}\Lambda ^{-1}\right) -X_{j}^{\func{tr}}X_{j}\right]
+R\Lambda ^{-1}; \\
\text{and }G_{\limfunc{lower}} &=&\sum_{j}\left( A_{j}X_{j}+X_{j}^{\func{tr}}%
\widetilde{A_{j}}+A_{0}\right) .
\end{eqnarray*}

We first consider $G_{\limfunc{top}}$. Using Lemma \ref{lem:aux} with $%
P=X_{j}$, $m=1,l=0$ we have 
\begin{align*}
\Lambda X_{j}\Lambda ^{-1}& =X_{j}+\limfunc{Op}\left( \left\{ \log \lambda
,\sigma \left( X_{j}\right) \right\} \right) +\text{symbol in }C^{0,\delta
}S_{1,0}^{-1,2}\ \ \ \ \ \func{mod}\mathcal{O}_{\left( -\delta ,\delta
\right) }^{-1-\varepsilon } \\
& =X_{j}+b_{j}+c_{j}\ \ \ \ \ \func{mod}\mathcal{O}_{\left( -\delta ,\delta
\right) }^{-1-\varepsilon },
\end{align*}%
where $b_{j}=\limfunc{Op}\left( \left\{ \log \lambda ,\sigma \left(
X_{j}\right) \right\} \right) \in C^{1,\delta }S_{1,0}^{0,1}$ and $c_{j}$
has a symbol in $C^{0,\delta }S_{1,0}^{-1,2}$. Since both $\left\{ \log
\lambda ,\sigma \left( X_{j}\right) \right\} $ and $\left\{ \log \lambda
,\sigma \left( X_{j}^{\func{tr}}\right) \right\} $ belong to $C^{1,\delta
}S_{1,0}^{0,1}$, inserting these equations into the identity derived for $G_{%
\limfunc{top}}$ in the preceding paragraph shows that 
\begin{equation*}
G_{\limfunc{top}}=\sum_{j}B_{j}\circ X_{j}+\sum_{j}X_{j}^{\func{tr}}\circ 
\widetilde{B_{j}}+B_{0}
\end{equation*}%
where the operators $B_{j}$, $\widetilde{B_{j}}\in \limfunc{Op}\left(
C^{1,\delta }S_{1,0}^{0,1}\right) $ and $B_{0}\in \limfunc{Op}\left(
C^{0,\delta }S_{1,0}^{0,2}\right) $ satisfy (\ref{B def}).\newline
Now consider $G_{\limfunc{lower}}$. We can write 
\begin{equation*}
\Lambda G_{\limfunc{lower}}\Lambda ^{-1}=\sum_{j}\Lambda \left(
A_{j}X_{j}+X_{j}^{\func{tr}}\widetilde{A_{j}}+A_{0}\right) \Lambda
^{-1}=\sum_{j}\left( \Lambda A_{j}\Lambda ^{-1}\Lambda X_{j}\Lambda
^{-1}+\Lambda X_{j}^{\func{tr}}\Lambda ^{-1}\Lambda \widetilde{A_{j}}\Lambda
^{-1}+\Lambda A_{0}\Lambda ^{-1}\right) .
\end{equation*}%
Applying Lemma \ref{lem:aux} to $A_{j}$ and $X_{j}$ we have 
\begin{align*}
\Lambda A_{j}\Lambda ^{-1}& =A_{j}+\text{symbol in }\limfunc{Op}\left(
C^{0,\delta }S_{1,0}^{-1,1}\right) \ \ \ \ \ \func{mod}\mathcal{O}_{\left(
-\delta ,\delta \right) }^{-1-\varepsilon }. \\
\Lambda X_{j}\Lambda ^{-1}& =X_{j}+\limfunc{Op}\left( \left\{ \log \lambda
,\sigma \left( X_{j}\right) \right\} \right) +\text{symbol in }C^{0,\delta
}S_{1,0}^{-1,2}\ \ \ \ \ \func{mod}\mathcal{O}_{\left( -\delta ,\delta
\right) }^{-1-\varepsilon }.
\end{align*}%
Using Theorem \ref{Saw calc} this gives 
\begin{equation*}
\Lambda A_{j}\Lambda ^{-1}\Lambda X_{j}\Lambda ^{-1}=A_{j}X_{j}+c_{j}X_{j}+%
\text{ symbol in }C^{0,\delta }S_{1,0}^{0,1}\ \ \ \ \ \func{mod}\mathcal{O}%
_{\left( -\delta ,\delta \right) }^{-\varepsilon },
\end{equation*}%
where $c_{j}\in C^{0,\delta }S_{1,0}^{-1,1}$, and the symbol in $C^{0,\delta
}S_{1,0}^{0,1}$ has the form $A_{j}\tilde{b}_{j}+$symbol in $C^{0,\delta
}S_{1,0}^{0,1}$ with $\tilde{b}_{j}=\left\{ \log \lambda ,\sigma \left(
X_{j}\right) \right\} $. Analyzing the other terms in $\Lambda G_{\limfunc{%
lower}}\Lambda ^{-1}$ in the same way we obtain 
\begin{align*}
& \Lambda G_{\limfunc{lower}}\Lambda ^{-1}=B_{j}X_{j}+X_{j}^{\func{tr}}%
\tilde{B}_{j}+\tilde{B}_{0}\ \ \ \ \ \func{mod}\mathcal{O}_{\left( -\delta
,\delta \right) }^{-\varepsilon }, \\
\text{where }\ & B_{j},\ \tilde{B}_{j}\ \ \ \ \ \text{as in \ref{B def}},
\end{align*}%
and $\tilde{B}_{0}\in \limfunc{Op}\left( C^{0,\delta }S_{1,0}^{0,1}\right) $
and has the structure as in (\ref{B def}). Combining with the estimate for $%
G_{\limfunc{top}}$ we obtain the result.
\end{proof}

\begin{lemma}[Lemma 4.6 in \protect\cite{Chr}]
\label{4.6'}Suppose that $L,\psi ,p$ satisfy the hypotheses of Theorem \ref%
{main_thm_Grushin}. Then for any $N\geq 0$, and for any fixed relatively
compact subset $U\subset V$, any $\delta >0$ and any $f\in C^{\gamma +3}$
supported in $U$, the operator $G$ constructed in Lemma \ref{lem form G}
satisfies 
\begin{equation}
\left\vert \left\langle Gf,f\right\rangle \right\vert \leq \delta
\sum_{j}\left\Vert X_{j}f\right\Vert ^{2}+\delta ||\sqrt{a}\widehat{\nabla }%
f||^{2}+C_{\delta }\left\Vert f\right\Vert ^{2}+C_{\delta }\left\Vert 
\limfunc{Op}\left( p\right) f\right\Vert _{H^{1}}^{2}\ .  \label{G bound}
\end{equation}
\end{lemma}

\begin{proof}
We first note 
\begin{equation*}
\sigma (b_{j})=\{\log \lambda ,\sigma (X_{j}^{\func{tr}})\}=-N_{0}\log |\xi
|\{\psi ,\sigma (X_{j}^{\func{tr}})\}+\text{symbol in }C^{1,\delta
}S_{1,0}^{0},
\end{equation*}%
and similarly for $\widetilde{b}_{j}$. Using this together with (\ref{B def}
) and hypothesis (\ref{superlog-G}) with $\delta =\delta _{0}$ we therefore
obtain 
\begin{align*}
\left\vert \left\langle B_{j}\circ X_{j}f,f\right\rangle \right\vert &
=\left\vert \left\langle (b_{j}+c_{j})\circ X_{j}f,f\right\rangle \right\vert
\\
& \leq \varepsilon \left\Vert X_{j}f\right\Vert ^{2}+C_{\varepsilon
}\left\Vert \widetilde{b_{j}}f\right\Vert ^{2}+C_{\varepsilon }\left\Vert
f\right\Vert ^{2} \\
& \leq \varepsilon \left\Vert X_{j}f\right\Vert ^{2}+C_{\varepsilon
}\left\Vert \log |\xi |\{\psi ,\sigma (X_{j}^{\func{tr}})\}f\right\Vert
^{2}+C_{\varepsilon }\left\Vert f\right\Vert ^{2} \\
& \leq \varepsilon \left\Vert X_{j}f\right\Vert ^{2}+C_{\varepsilon }\left(
\delta _{0}\sum_{j}||X_{j}u||^{2}+\delta _{0}||\sqrt{a}\widehat{\nabla }%
f||^{2}+C_{\delta _{0}}||f||^{2}+C_{\delta
_{0}}||Op(p)f||_{H^{1}}^{2}\right) +C_{\varepsilon }\left\Vert f\right\Vert
^{2}.
\end{align*}%
Choosing $\delta _{0}=\varepsilon /C_{\varepsilon }$ this gives 
\begin{equation*}
\left\vert \left\langle B_{j}\circ X_{j}f,f\right\rangle \right\vert \leq
\varepsilon \sum_{j}||X_{j}u||^{2}+\varepsilon ||\sqrt{a}\widehat{\nabla }%
f||^{2}+C_{\varepsilon }||f||^{2}+C_{\varepsilon }||Op(p)f||_{H^{1}}^{2}.
\end{equation*}%
The rest of the terms in (\ref{form G}) are handled in the same way, giving
( \ref{G bound}).
\end{proof}

To handle the Grushin type term $\widehat{\nabla }\cdot \mathbf{Q}_{p}\left(
x\right) \widehat{\nabla }$ in (\ref{operator-Grushin}) we will need the
following two lemmas

\begin{lemma}
\label{lem: E} There holds 
\begin{equation*}
\left( \widehat{\nabla }\cdot \mathbf{Q}_{p}\widehat{\nabla }\eta _{1}+%
\mathbf{E}\right) \Lambda \eta _{2}=\eta _{1}\Lambda \widehat{\nabla }\cdot 
\mathbf{Q}_{p}\widehat{\nabla }\eta _{2}+\mathbf{R},
\end{equation*}%
where $\mathbf{R}\in \mathcal{O}_{\left( -\delta ,\delta \right)
}^{-\varepsilon }$, and with $\widehat{\xi }=\left( \xi _{p},\dots ,\xi
_{n}\right) $, the matrix operator $\mathbf{E}$ takes the form 
\begin{eqnarray}
\mathbf{E} &=&H\circ \mathbf{Q}_{p}\widehat{\nabla}+\widehat{\nabla}\left(
\sum_{\left\vert \alpha \right\vert =1}D^{\alpha }\mathbf{Q}_{p}\right)
\circ H_{0}+H\circ \left( \sum_{\left\vert \alpha \right\vert =1}D^{\alpha }%
\mathbf{Q}_{p}\right) \circ H_{0}  \label{E form p} \\
&&+H_{3}\circ \mathbf{Q}_{p}\widehat{\nabla}+H\circ \mathbf{Q}_{p}H+\tilde{H}%
_{0}\ \ \ \ \ \mathcal{O}_{\left( -\delta ,\delta \right) }^{-\varepsilon };
\notag \\
\text{where }H &=&\func{Op}\left\{ \log \lambda ,\widehat{\xi }\right\} \in 
\func{Op}\left( S_{1,0}^{0,1}\right) ,\ \ \ H_{0}\in \func{Op}\left(
S_{1,0}^{0}\right) ,\text{\ \ \ }\tilde{H}_{0}\in \func{Op}\left(
C^{0,\delta }S_{1,0}^{0}\right) ,\text{\ \ \ }H_{3}\in \func{Op}\left(
S_{1,0}^{-1,1}\right) .  \notag
\end{eqnarray}
\end{lemma}

\begin{proof}
In constructing the symbol of $\mathbf{E}$ we will work formally, ignoring
the cutoff functions $\eta _{1}$ and $\eta _{2}$. This is permissible by
pseudolocality since $\eta _{1}\eta _{2}=\eta _{2}$. Let $\mathbf{L}%
_{2}\equiv \widehat{\nabla }\cdot \mathbf{Q}_{p}\widehat{\nabla }$, the
desired equation $\left( \mathbf{L}_{2}+\mathbf{E}\right) \Lambda =\Lambda 
\mathbf{L}_{2}+\mathbf{R}$ is then equivalent to 
\begin{align}
\mathbf{E}& =\Lambda \mathbf{L}_{2}\Lambda ^{-1}-\mathbf{L}_{2}+\mathbf{R}%
\Lambda ^{-1}  \notag \\
& =\Lambda \widehat{\nabla }\cdot \mathbf{Q}_{p}\widehat{\nabla }\Lambda
^{-1}-\widehat{\nabla }\cdot \mathbf{Q}_{p}\widehat{\nabla }+\mathbf{R}%
\Lambda ^{-1}  \notag \\
& =\left( \Lambda \widehat{\nabla }\Lambda ^{-1}\right) \cdot \left( \Lambda 
\mathbf{Q}_{p}\widehat{\nabla }\Lambda ^{-1}\right) -\widehat{\nabla }\cdot 
\mathbf{Q}_{p}\widehat{\nabla }+\mathbf{R}\Lambda ^{-1}.
\end{align}%
Next using Lemma \ref{lem:aux} we have 
\begin{equation*}
\Lambda \widehat{\nabla }\Lambda ^{-1}=\widehat{\nabla }+\{\log \lambda ,%
\widehat{\xi }\}+\text{symbol in }S_{1,0}^{-1,1}\equiv \widehat{\nabla }%
+H+H_{3},
\end{equation*}%
where $H=\{\log \lambda ,\widehat{\xi }\}\in \limfunc{Op}\left(
S_{1,0}^{0,1}\right) $, and $H_{3}\in \limfunc{Op}\left(
S_{1,0}^{-1,1}\right) $. To estimate $\Lambda \mathbf{Q}_{p}\widehat{\nabla }%
\Lambda ^{-1}$ we will need a refinement of Lemma \ref{lem:aux}, namely, the
estimate obtained in the proof 
\begin{eqnarray*}
\frac{1}{\lambda }\left\{ \lambda \odot \sigma \left( P\right) -\sigma
\left( P\right) \odot \lambda \right\} &=&\left\{ \sum_{\left\vert \alpha
\right\vert =1}+\sum_{\left\vert \alpha \right\vert =2}\right\} c_{\alpha }\ %
\left[ \frac{\partial _{\xi }^{\alpha }\lambda }{\lambda }\ \partial
_{x}^{\alpha }\sigma \left( P\right) -\partial _{\xi }^{\alpha }\sigma
\left( P\right) \ \frac{\partial _{x}^{\alpha }\lambda }{\lambda }\right] +S
\\
&=&\left\{ \log \lambda ,\sigma \left( P\right) \right\} +\sum_{\left\vert
\alpha \right\vert =2}c_{\alpha }\ \left[ \partial _{\xi }^{\alpha }\log
\lambda \ \partial _{x}^{\alpha }\sigma \left( P\right) -\partial _{\xi
}^{\alpha }\sigma \left( P\right) \ \partial _{x}^{\alpha }\log \lambda %
\right] +S\ ,
\end{eqnarray*}%
where $S\in \mathcal{O}_{\left( -\nu ,\nu \right) }^{-1-\varepsilon }$ for
some $0<\varepsilon <1$ and $0<\nu <\delta $. Now $\sigma (P)=\sigma (%
\mathbf{Q}_{p}\widehat{\nabla })=\mathbf{Q}_{p}\widehat{\xi }$, so 
\begin{equation*}
\sum_{\left\vert \alpha \right\vert =2}c_{\alpha }\ \left[ \partial _{\xi
}^{\alpha }\log \lambda \ \partial _{x}^{\alpha }\sigma \left( P\right)
-\partial _{\xi }^{\alpha }\sigma \left( P\right) \ \partial _{x}^{\alpha
}\log \lambda \right] =\sum_{\left\vert \alpha \right\vert =2}c_{\alpha
}\partial _{\xi }^{\alpha }\log \lambda \ \partial _{x}^{\alpha }\mathbf{Q}%
_{p}\widehat{\xi }=\text{symbol in }\mathcal{C}^{0,\delta }S_{1,0}^{-1,0},
\end{equation*}%
where the last equality holds since $\psi $ does not depend on $\xi $ in $%
\Gamma $, and therefore no logarithmic terms arise from differentiation of $%
\log \lambda $ with respect to $\xi $. Altogether we thus have 
\begin{align*}
\Lambda \mathbf{Q}_{p}\widehat{\nabla }\Lambda ^{-1}& =\mathbf{Q}_{p}%
\widehat{\nabla }+\limfunc{Op}\left( \{\log \lambda ,\mathbf{Q}_{p}\widehat{%
\xi }\}\right) +\text{ symbol in }\mathcal{C}^{0,\delta }S_{1,0}^{-1,0}\ \ \
\ \ \func{mod}\mathcal{O}_{\left( -\delta ,\delta \right) }^{-1-\varepsilon }
\\
& =\mathbf{Q}_{p}\widehat{\nabla }+\sum_{|\alpha |=1}(D^{\alpha }\mathbf{Q}%
_{p})\widehat{\xi }D_{\xi }^{\alpha }\log \lambda +\mathbf{Q}_{p}\cdot
\{\log \lambda ,\widehat{\xi }\}+\text{symbol in }\mathcal{C}^{0,\delta
}S_{1,0}^{-1,0}\ \ \ \ \ \func{mod}\mathcal{O}_{\left( -\delta ,\delta
\right) }^{-1-\varepsilon } \\
& =\mathbf{Q}_{p}\widehat{\nabla }+\sum_{|\alpha |=1}(D^{\alpha }\mathbf{Q}%
_{p})\cdot \text{ symbol in }S_{1,0}^{0}+\mathbf{Q}_{p}\cdot H+\text{symbol
in }\mathcal{C}^{0,\delta }S_{1,0}^{-1,0}\ \ \ \ \ \func{mod}\mathcal{O}%
_{\left( -\delta ,\delta \right) }^{-1-\varepsilon }
\end{align*}%
where we note that $\widehat{\xi }D_{\xi }^{\alpha }\log \lambda \in
S_{1,0}^{0}$ for each $\alpha $ with $|\alpha |=1$ since $\psi $ does not
depend on $\xi $, and therefore no logarithmic terms arise from
differentiation of $\log \lambda $ with respect to $\xi $. This gives 
\begin{align*}
\left( \Lambda \widehat{\nabla }\Lambda ^{-1}\right) \cdot \left( \Lambda 
\mathbf{Q}_{p}\widehat{\nabla }\Lambda ^{-1}\right) & =\widehat{\nabla }%
\cdot \mathbf{Q}_{p}\widehat{\nabla }+H\circ \mathbf{Q}_{p}\widehat{\nabla }%
+H_{3}\circ \mathbf{Q}_{p}\widehat{\nabla }+\widehat{\nabla }\left(
\sum_{|\alpha |=1}D^{\alpha }\mathbf{Q}_{p}\right) \circ H_{0} \\
& \quad +H\circ \left( \sum_{|\alpha |=1}D^{\alpha }\mathbf{Q}_{p}\right)
\circ H_{0}+H\circ \mathbf{Q}_{p}H+\tilde{H}_{0}\ \ \ \ \ \func{mod}\mathcal{%
O}_{\left( -\delta ,\delta \right) }^{-\varepsilon }.
\end{align*}
\end{proof}

\begin{lemma}
Let $\mathbf{E}$ be a pseudodifferential operator of the form (\ref{E form p}
). Then for any fixed relatively compact subset $U\subset V$, any $\delta >0$
and any $f\in C_{c}^{\infty }$ supported in $U$, we have 
\begin{equation}
\left\vert \left\langle \mathbf{E}f,f\right\rangle \right\vert \leq \delta
\sum_{j}\left\Vert X_{j}f\right\Vert ^{2}+\delta ||\sqrt{a}\widehat{\nabla }%
f||^{2}+C_{\delta }\left\Vert f\right\Vert ^{2}+C_{\delta }\left\Vert 
\limfunc{Op}\left( p\right) f\right\Vert _{H^{1}}^{2}\ .  \label{E bound p}
\end{equation}
\end{lemma}

\begin{proof}
Here is where we will need to use that the matrix $\mathbf{Q}_{p}$ is
subordinate - in the case $p=n$, then $\mathbf{Q}_{n}$ is simply a scalar
and the subordinate inequality is that of Malgrange. We will use (\ref{E
form p}) and the notation $\mathbf{Q}_{p}^{\prime }=\sum_{|\alpha
|=1}D^{\alpha }\mathbf{Q}_{p}$. We have 
\begin{eqnarray*}
\left\langle \mathbf{E}f,f\right\rangle &=&-\left\langle \sqrt{\mathbf{Q}_{p}%
}H^{\func{tr}}f,\sqrt{\mathbf{Q}_{p}}\widehat{\nabla }f\right\rangle
-\left\langle H_{0}f,\mathbf{Q}_{p}^{\prime }\widehat{\nabla }f\right\rangle
+\left\langle H_{0}f,\mathbf{Q}_{p}^{\prime }H^{\func{tr}}f\right\rangle \\
&&+\left\langle H_{3}^{\func{tr}}f,\mathbf{Q}_{p}\nabla ^{\prime
}f\right\rangle +\left\langle \sqrt{\mathbf{Q}_{p}}Hf,\sqrt{\mathbf{Q}_{p}}%
H^{\func{tr}}f\right\rangle +\left\langle H_{0}f,f\right\rangle .
\end{eqnarray*}%
Now we use the crucial fact that $\mathbf{Q}_{p}$ is subordinate, i.e. $%
\left\vert \mathbf{Q}_{p}^{\prime }\right\vert ^{2}\leq C\mathbf{Q}_{p}$,
and together with Cauchy-Schwartz this gives 
\begin{equation*}
\left\vert \left\langle \mathbf{E}f,f\right\rangle \right\vert \leq \delta
\left\Vert \sqrt{\mathbf{Q}_{p}}\widehat{\nabla }f\right\Vert ^{2}+C_{\delta
}\left\Vert \sqrt{\mathbf{Q}_{p}}H^{\func{tr}}f\right\Vert ^{2}+C_{\delta
}\left\Vert \sqrt{\mathbf{Q}_{p}}Hf\right\Vert ^{2}+C_{\delta }\left\Vert
f\right\Vert ^{2}.
\end{equation*}%
Finally, using the definition of $\lambda $ we obtain 
\begin{equation*}
\sigma (H)=\{\log \lambda ,\widehat{\xi }\}=-N_{0}\log |\xi |\{\psi ,%
\widehat{\xi }\},
\end{equation*}%
which together with the fact that $\mathbf{Q}_{p}\approx a\mathbb{I}_{n-p+1}$
shows 
\begin{equation*}
\left\vert \left\langle \mathbf{E}f,f\right\rangle \right\vert \leq \delta
\left\Vert \sqrt{a}\widehat{\nabla }f\right\Vert ^{2}+C_{\delta }\left\Vert 
\sqrt{a}\limfunc{Op}\left( \log \left\langle \xi \right\rangle \{\psi ,%
\widehat{\xi }\}\right) f\right\Vert ^{2}+C_{\delta }\left\Vert f\right\Vert
^{2}\ .
\end{equation*}%
Combining with estimate (\ref{superlog-G}) as in the proof of Lemma \ref%
{4.6'} we conclude (\ref{E bound p}).
\end{proof}

Finally, we obtain an estimate on the subunit term $R_{1}$.

\begin{lemma}
\label{lem:R1-G} Let $R_{1}=\sum_{k=1}^{n}S_{k}\Theta _{k}\circ \widehat{%
\nabla }$, where each $S_{k}\in C^{1,\delta }(\mathbb{R}^{m\times m})$ is
subunit with respect to $\mathbf{Q}_{p}$, and $\Theta _{k}=(\Theta
_{kp},\dots ,\Theta _{kn})$ is a multiplier of order zero. 
Then 
\begin{equation}
(R_{1}\eta _{1}+J)\Lambda \eta _{2}=\eta _{1}\Lambda R_{1}\eta _{2}+R,
\label{R1-est-G}
\end{equation}%
where $J\in \limfunc{Op}(\mathcal{C}^{0,\delta }S_{1,0}^{0,1})$, $R\in 
\mathcal{O}_{\left( -\delta ,\delta \right) }^{-1-\varepsilon }$, and 
\begin{equation}
\left\vert \left\langle Jf,f\right\rangle \right\vert \leq \delta
\sum_{j}\left\Vert X_{j}f\right\Vert ^{2}+\delta ||\sqrt{a}\widehat{\nabla }%
f||^{2}+C_{\delta }\left\Vert f\right\Vert ^{2}+C_{\delta }\left\Vert 
\limfunc{Op}\left( p\right) f\right\Vert _{H^{1}}^{2}\ ,  \label{J-est-G}
\end{equation}%
any $\delta >0$ and any $f\in C_{c}^{\infty }$.
\end{lemma}

\begin{proof}
Proceeding as in the proof of Lemma \ref{lem: E} we have 
\begin{align*}
\Lambda S_{k}\Theta _{k}\circ \widehat{\nabla}\Lambda ^{-1}& =S_{k}\Theta
_{k}\circ \widehat{\nabla}+\sum_{|\alpha |=1}(D^{\alpha }S_{k})\widehat{\xi }%
\theta _{k}(\xi )D_{\xi }^{\alpha }\log \lambda +S_{k}\{\log \lambda ,%
\widehat{\xi }\theta _{k}(\xi )\}+\text{symbol in }C^{0,\delta
}S_{1,0}^{-1,0}\ \ \ \ \ \func{mod}\mathcal{O}_{\left( -\delta ,\delta
\right) }^{-1-\varepsilon } \\
& =S_{k}\Theta _{k}\circ \widehat{\nabla}+\text{symbol in }C^{0,\delta
}S_{1,0}^{0}+S_{k}H_{k}+\text{symbol in }C^{0,\delta }S_{1,0}^{-1,0}\ \ \ \
\ \func{mod}\mathcal{O}_{\left( -\delta ,\delta \right) }^{-1-\varepsilon }
\\
& \equiv S_{k}\Theta _{k}\circ \widehat{\nabla}+J_{k}\ \ \ \ \ \func{mod}%
\mathcal{O}_{\left( -\delta ,\delta \right) }^{-1-\varepsilon }.
\end{align*}%
where $H_{k}\in \limfunc{Op}\left( S_{1,0}^{0,1}\right) $. Defining $J\equiv
\sum_{k=1}^{n}J_{k}$ and using the fact that $S_{k}$ is subunit together
with $\mathbf{Q}_{p}\approx a\mathbb{I}_{n-p+1}$ and (\ref{superlog-G}) we
obtain (\ref{J-est-G}).
\end{proof}

We are now ready to prove a generalization of Lemma 4.4 in \cite{Chr}, which
is the main estimate we need.

\begin{lemma}[Lemma 4.4 in \protect\cite{Chr}]
\label{pre}Let $L$ take the form (\ref{operator-Grushin}) and satisfy (\ref%
{Christ_cond1-G}) and (\ref{superlog-G}). Let $0<\gamma <\delta $ be fixed.
If $N_{0}$ is chosen sufficiently large in the definition of $\Lambda $,
then for any fixed relatively compact $U\Subset V$ and any $u\in C^{2,\delta
}\left( U\right) $, 
\begin{equation}
\left\Vert \eta _{1}\Lambda u\right\Vert _{L^{2}\left( \mathbb{R}^{n}\right)
}\leq C\left\Vert \eta _{1}\Lambda Lu\right\Vert _{L^{2}\left( \mathbb{R}%
^{n}\right) }+C\left\Vert u\right\Vert _{H^{0}\left( \mathbb{R}^{n}\right) }.
\label{pre hypo}
\end{equation}
\end{lemma}

\begin{proof}
Recall that 
\begin{equation*}
L=\sum_{j}X_{j}^{\func{tr}}X_{j}+\sum_{j}A_{j}X_{j}+\sum_{j}X_{j}^{\func{tr}}%
\tilde{A}_{j}+A_{0}+R_{1}+\widehat{\nabla }\cdot \mathbf{Q}_{p}\widehat{%
\nabla }\ \equiv L_{1}+L_{2}+R_{1},
\end{equation*}%
where we used the notation $L_{2}=\widehat{\nabla }\cdot \mathbf{Q}_{p}%
\widehat{\nabla }$ If we set 
\begin{equation*}
v\equiv \eta _{1}\Lambda u\in C^{2}\left( \mathbb{R}^{n}\right) ,
\end{equation*}%
we have 
\begin{eqnarray*}
\left\langle \left( L_{1}+G\right) v,v\right\rangle &=&\left\langle
L_{1}v,v\right\rangle +\left\langle Gv,v\right\rangle \\
&=&\sum_{j}\left\Vert X_{j}v\right\Vert _{L^{2}}^{2}+\sum_{j}\left\langle
A_{j}\circ X_{j}v,v\right\rangle +\sum_{j}\left\langle X_{j}^{\func{tr}%
}\circ \widetilde{A}_{j}v,v\right\rangle +\left\langle A_{0}v,v\right\rangle
+\left\langle Gv,v\right\rangle \\
&=&\sum_{j}\left\Vert X_{j}v\right\Vert _{L^{2}}^{2}+\sum_{j}\left\langle
X_{j}v,A_{j}^{\func{tr}}v\right\rangle +\sum_{j}\left\langle \widetilde{A}%
_{j}v,X_{j}v\right\rangle +\left\langle A_{0}v,v\right\rangle +\left\langle
Gv,v\right\rangle \\
&=&\sum_{j}\left\Vert X_{j}v\right\Vert _{L^{2}}^{2}+O\left( \sqrt{%
\sum_{j}\left\Vert X_{j}v\right\Vert _{L^{2}}^{2}}\left\Vert v\right\Vert
_{L^{2}}+\left\Vert v\right\Vert _{L^{2}}^{2}\right) +\left\langle
Gv,v\right\rangle ,
\end{eqnarray*}%
since the operators $A_{j}$ and $\widetilde{A}_{j}$ have order $0$.
Similarly 
\begin{eqnarray*}
\left\langle \left( L_{2}+E\right) v,v\right\rangle &=&\left\langle \nabla
^{\prime }\cdot \mathbf{Q}_{p}\widehat{\nabla }v,v\right\rangle
+\left\langle Ev,v\right\rangle \\
&=&\int |\sqrt{\mathbf{Q}_{p}}\widehat{\nabla }v|^{2}+\left\langle
Ev,v\right\rangle , \\
\left\langle \left( R_{1}+J_{0}\right) v,v\right\rangle &=&\left\langle
\sum_{i=1}^{n}S_{i}\Theta _{i}\widehat{\nabla }v,v\right\rangle
+\left\langle Jv,v\right\rangle \\
&\leq &\delta \int a|\widehat{\nabla }v|^{2}+C_{\delta }\left\Vert
v\right\Vert _{L^{2}}^{2}+\left\langle J_{0}v,v\right\rangle .
\end{eqnarray*}%
We also have from Lemmas \ref{4.6'}, \ref{lem: E}, and \ref{lem:R1-G} that 
\begin{align*}
\left( L_{1}+G\right) v=\left( L_{1}+G\right) \eta _{1}\Lambda \eta _{2}u&
=\eta _{1}\Lambda L_{1}\eta _{2}u+Ru=\eta _{1}\Lambda L_{1}u+Ru\ , \\
\left( L_{2}+E\right) v=\left( L_{2}+E\right) \eta _{1}\Lambda \eta _{2}u&
=\eta _{1}\Lambda L_{2}\eta _{2}u+Ru=\eta _{1}\Lambda L_{2}u+Ru\ , \\
\left( R_{1}+J_{0}\right) v=\left( R_{1}+J\right) \eta _{1}\Lambda \eta
_{2}u& =\eta _{1}\Lambda R_{1}\eta _{2}u+Ru=\eta _{1}\Lambda R_{1}u+Ru
\end{align*}%
since $\eta _{2}u=u$, and hence adding together 
\begin{equation*}
\left\vert \left\langle \left( L+G+E+J\right) v,v\right\rangle \right\vert
\leq \left\vert \left\langle \eta _{1}\Lambda L\eta _{2}u,v\right\rangle
\right\vert +\left\vert \left\langle Ru,v\right\rangle \right\vert \leq 
\frac{1}{2}\left\Vert \eta _{1}\Lambda Lu\right\Vert _{L^{2}\left( \mathbb{R}%
^{n}\right) }^{2}+\frac{1}{2}\left\Vert Ru\right\Vert _{L^{2}\left( \mathbb{R%
}^{n}\right) }^{2}+\left\Vert v\right\Vert _{L^{2}\left( \mathbb{R}%
^{n}\right) }^{2}\newline
\end{equation*}

Thus from (\ref{G bound}), (\ref{E bound p}), (\ref{J-est-G}), and the above
we conclude that 
\begin{eqnarray*}
\sum_{j}\left\Vert X_{j}v\right\Vert _{L^{2}}^{2}+||\sqrt{\mathbf{Q}_{p}}%
\widehat{\nabla }v||^{2} &=&\left\langle \left( L_{1}+G\right)
v,v\right\rangle -\left\langle Gv,v\right\rangle +C\left( \sqrt{%
\sum_{j}\left\Vert X_{j}v\right\Vert _{L^{2}}^{2}}\left\Vert v\right\Vert
_{L^{2}}+\left\Vert v\right\Vert _{L^{2}}^{2}\right) \\
&&+\left\langle \left( L_{2}+E\right) v,v\right\rangle -\left\langle
Ev,v\right\rangle \\
&&+\left\langle \left( R_{1}+J\right) v,v\right\rangle -\left\langle
J_{0}v,v\right\rangle -\left\langle \sum_{i=1}^{n}S_{i}\Theta _{i}\widehat{%
\nabla }v,v\right\rangle \\
&\leq &\frac{1}{2}\left\Vert \eta _{1}\Lambda Lu\right\Vert _{L^{2}}^{2}+%
\frac{1}{2}\left\Vert Ru\right\Vert _{L^{2}}^{2}+C_{\delta }\left\Vert
v\right\Vert _{L^{2}}^{2} \\
&&+\delta \sum_{j}\left\Vert X_{j}f\right\Vert _{L^{2}}^{2}+4\delta ||\sqrt{a%
}\widehat{\nabla }v||_{L^{2}}^{2}+C_{\delta }\left\Vert \limfunc{Op}\left(
p\right) v\right\Vert _{H^{1}}^{2} \\
&&+C\left( \sqrt{\sum_{j}\left\Vert X_{j}v\right\Vert _{L^{2}}^{2}}%
\left\Vert v\right\Vert _{L^{2}}+\left\Vert v\right\Vert _{L^{2}}^{2}\right)
.
\end{eqnarray*}%
Combining this with the inequality 
\begin{equation*}
\sqrt{\sum_{j}\left\Vert X_{j}v\right\Vert _{L^{2}}^{2}}\left\Vert
v\right\Vert _{L^{2}}\leq \delta \sum_{j}\left\Vert X_{j}v\right\Vert
_{L^{2}}^{2}+C_{\delta }\left\Vert v\right\Vert _{L^{2}}^{2}\ 
\end{equation*}%
and the condition $\mathbf{Q}_{p}\approx a\mathbb{I}_{n-p+1}$ we obtain,
choosing $\delta $ smaller if necessary, 
\begin{eqnarray*}
\sum_{j}\left\Vert X_{j}v\right\Vert _{L^{2}}^{2}+||\sqrt{a}\widehat{\nabla }%
v||^{2} &\leq &\frac{1}{2}\left\Vert \eta _{1}\Lambda Lu\right\Vert
_{L^{2}}^{2}+\frac{1}{2}\left\Vert Ru\right\Vert _{L^{2}}^{2}+C_{\delta
}\left\Vert v\right\Vert _{L^{2}}^{2}+C_{\delta }\left\Vert \limfunc{Op}%
\left( p\right) v\right\Vert _{H^{1}}^{2} \\
&&+\delta \sum_{j}\left\Vert X_{j}v\right\Vert _{L^{2}}^{2}+\delta ||\sqrt{a}%
\widehat{\nabla }v||_{L^{2}}^{2}.
\end{eqnarray*}%
Absorbing the terms $\delta \sum_{j}\left\Vert X_{j}v\right\Vert
_{L^{2}}^{2} $ and $\delta ||\sqrt{a}\widehat{\nabla }v||_{L^{2}}^{2}$ into
the left hand side, and then using that the order of the error term $R$ is $%
-\varepsilon $, we obtain 
\begin{equation}
\sum_{j}\left\Vert X_{j}v\right\Vert _{L^{2}}^{2}+||\sqrt{a}\widehat{\nabla }%
v||_{L^{2}}^{2}\leq \left\Vert \eta _{1}\Lambda Lu\right\Vert _{L^{2}\left( 
\mathbb{R}^{n}\right) }^{2}+C\left\Vert v\right\Vert
_{L^{2}}^{2}+C\left\Vert u\right\Vert _{H^{-\varepsilon }}^{2}\ ,
\label{est:x-j}
\end{equation}%
where the term involving the $H^{1}$ norm of $\limfunc{Op}\left( p\right)
\Lambda u$ may be absorbed into $\left\Vert u\right\Vert _{H^{-\varepsilon
}}^{2}$ since $\Lambda $ may be made to be regularizing of arbitrary high
order in a conic neighborhood of the symbol $p$, by choosing $N_{0}$ to be
sufficiently large.\newline
Next we write 
\begin{eqnarray*}
\left\Vert v\right\Vert _{L^{2}}^{2} &=&\int_{\left\{ \xi \in \mathbb{R}%
^{n}:\left\vert \xi \right\vert \leq N\right\} }\left\vert \widehat{v}\left(
\xi \right) \right\vert ^{2}d\xi +\int_{\left\{ \xi \in \mathbb{R}%
^{n}:\left\vert \xi \right\vert >N\right\} }\left\vert \widehat{v}\left( \xi
\right) \right\vert ^{2}d\xi \\
&\leq &N^{2\gamma }\int_{\left\{ \xi \in \mathbb{R}^{n}:\left\vert \xi
\right\vert \leq N\right\} }\left\langle \xi \right\rangle ^{-2\gamma
}\left\vert \widehat{v}\left( \xi \right) \right\vert ^{2}d\xi +\frac{1}{%
w^{2}(N)}\int_{\left\{ \xi \in \mathbb{R}^{n}:\left\vert \xi \right\vert
>N\right\} }w^{2}\left( \left\langle \xi \right\rangle \right) \left\vert 
\widehat{v}\left( \xi \right) \right\vert ^{2}d\xi \\
&\leq &N^{2\gamma }\left\Vert u\right\Vert _{H^{0}}^{2}+\frac{1}{w^{2}(N)}%
\left\Vert w\left( \left\langle \xi \right\rangle \right) \widehat{v}\left(
\xi \right) \right\Vert _{L^{2}}^{2} \\
&\leq &N^{2\gamma }\left\Vert u\right\Vert _{H^{0}}^{2}+\frac{C}{w^{2}(N)}%
\left( \sum_{j}\left\Vert X_{j}v\right\Vert _{L^{2}}^{2}+||\sqrt{a}\widehat{%
\nabla }v||_{L^{2}}^{2}+\left\Vert v\right\Vert _{L^{2}}^{2}\right)
\end{eqnarray*}%
where for the last inequality we used (\ref{Christ_cond1-G}). Let $\delta
=C/w^{2}(N)$ and note that $\delta $ can be made arbitrarily small by
choosing $N$ sufficiently large, we combine the above equality with (\ref%
{est:x-j}) to obtain 
\begin{align*}
\left\Vert v\right\Vert _{L^{2}}^{2}& \leq C_{\delta }\left\Vert
u\right\Vert _{H^{0}}^{2}+\delta \left( \sum_{j}\left\Vert X_{j}v\right\Vert
_{L^{2}}^{2}+||\sqrt{a}\widehat{\nabla }v||_{L^{2}}^{2}+\left\Vert
v\right\Vert _{L^{2}}^{2}\right) \\
& \leq C_{\delta }\left\Vert u\right\Vert _{H^{0}}^{2}+\delta \left(
\left\Vert \eta _{1}\Lambda Lu\right\Vert _{L^{2}\left( \mathbb{R}%
^{n}\right) }^{2}+C\left\Vert v\right\Vert _{L^{2}}^{2}+C\left\Vert
u\right\Vert _{H^{-\varepsilon }}^{2}\right)
\end{align*}%
Choosing $\delta $ sufficiently small to absorb the norm $\left\Vert
v\right\Vert _{L^{2}}^{2}$ to the left hand side we conclude 
\begin{equation*}
\left\Vert \eta _{1}\Lambda u\right\Vert _{L^{2}\left( \mathbb{R}^{n}\right)
}^{2}=\left\Vert v\right\Vert _{L^{2}\left( \mathbb{R}^{n}\right) }^{2}\leq
C_{\gamma }\left\Vert \eta _{1}\Lambda Lu\right\Vert _{L^{2}\left( \mathbb{R}%
^{n}\right) }^{2}+C_{\gamma }\left\Vert u\right\Vert _{H^{0}\left( \mathbb{R}%
^{n}\right) }^{2},
\end{equation*}%
for a constant $C_{\gamma }$ depending on $\gamma $.
\end{proof}

\subsubsection{Removal of the smoothness assumption}

It remains to remove the smoothness assumption $u\in C^{2,\delta }\left(
U\right) $ in Lemma \ref{pre}, and to convert the above \emph{a priori}
estimate (\ref{pre hypo}) to the desired conclusion $\Lambda u\in H^{0}$ of
Theorem \ref{main_thm_Grushin}. For this we fix a strictly positive smooth
function $r\in C^{\infty }\left( \mathbb{R}^{n}\right) $ such that 
\begin{equation*}
r\left( \xi \right) \equiv \left\{ 
\begin{array}{ccc}
\left\vert \xi \right\vert ^{-1} & \text{\ for\ } & \left\vert \xi
\right\vert \geq 2 \\ 
1 & \text{\ for\ } & \left\vert \xi \right\vert \leq 1%
\end{array}%
\right. ,
\end{equation*}%
and we fix a large exponent $q$. For $\varepsilon >0$ small define a
mollified\ symbol 
\begin{eqnarray*}
\lambda _{\varepsilon }\left( x,\xi \right) &=&r_{\varepsilon }\left( \xi
\right)\cdot \lambda \left( x,\xi \right)=r\left( \varepsilon \xi \right)
^{q}\cdot \lambda \left( x,\xi \right) ; \\
\text{where }r_{\varepsilon }\left( \xi \right) &\equiv &r\left( \varepsilon
\xi \right) ^{q}.
\end{eqnarray*}%
where $\lambda \left( x,\xi \right) =\left\vert \xi \right\vert ^{\gamma
}e^{-N_{0}\left( \log \left\vert \xi \right\vert \right) \phi \left( x,\xi
\right) }$ for $\left\vert \xi \right\vert \geq e$ as in (\ref{def lambda}).
Let $\Lambda _{\varepsilon }=\limfunc{Op}\lambda _{\varepsilon }$. The
symbols $r_{\varepsilon }\left( \xi \right) $ satisfy 
\begin{equation}
\frac{\left\vert \partial _{\xi }^{\alpha }r_{\varepsilon }\right\vert }{%
r_{\varepsilon }}\leq C_{\alpha ,q}\left\vert \xi \right\vert ^{-\left\vert
\alpha \right\vert },\ \ \ \ \ \text{uniformly in }\varepsilon >0\text{ and }%
\xi \in \mathbb{R}^{n}.  \label{diff ineq}
\end{equation}

If $q$ is chosen sufficiently large relative to the order of the
distribution $u$, then $\Lambda _{\varepsilon }u\in C^{2}$ for all $%
\varepsilon >0$, and since $\Lambda _{\varepsilon }$ is elliptic of order $%
\gamma $ in a conic neighbourhood of $\left( x_{0},\xi _{0}\right) $, it
suffices to show that the $L^{2}$ norm of $\eta _{1}\Lambda _{\varepsilon }u$
remains uniformly bounded as $\varepsilon \searrow 0$. However, Lemma \ref%
{pre} fails to apply since we do not know that the distribution $u$ is a
function in $C^{2,\delta }\left( U\right) $, and we now work to circumvent
this difficulty.

The parameter $N_{0}$ in (\ref{def lambda}) can be chosen sufficiently large
that $\eta _{1}\Lambda Lu\in L^{2}$ because $\phi $ is strictly positive in
a conic neighbourhood of the $H^{\gamma }$ wave front set of $u$, and hence $%
\Lambda $ is regularizing there of order at least $\gamma -\sigma N_{0}$ for
some constant $\sigma >0$. The $L^{2}$ norm of $\eta _{1}\Lambda
_{\varepsilon }Lu$ is bounded uniformly in $\varepsilon >0$ and tends to the 
$L^{2}$ norm of $\eta _{1}\Lambda Lu$.

As in the proof of Lemma \ref{pre}, we have for each $\varepsilon >0$, an
operator $G_{\varepsilon }$ and an identity%
\begin{align*}
\left( L_{1}+G_{\varepsilon }\right) \eta _{1}\Lambda _{\varepsilon }u&
=\eta _{1}\Lambda _{\varepsilon }L_{1}u+R_{\varepsilon }u\ , \\
\left( L_{2}+E_{\varepsilon }\right) \eta _{1}\Lambda _{\varepsilon }u&
=\eta _{1}\Lambda _{\varepsilon }L_{2}u+R_{\varepsilon }u\ , \\
\left( R_{1}+J_{\varepsilon }\right) \eta _{1}\Lambda _{\varepsilon }u&
=\eta _{1}\Lambda _{\varepsilon }R_{1}u+R_{\varepsilon }u\ ,
\end{align*}
with both sides of the equation in $C^{2}$ for each $\varepsilon >0$.
Moreover, the differential inequalities (\ref{diff ineq}) ensure that the
proof of Lemma \ref{pre} carries through for each $\varepsilon >0$ with $%
\Lambda $ replaced by $\Lambda _{\varepsilon }$, \ so that $G_{\varepsilon }$
takes the form (\ref{form G}), i.e. 
\begin{eqnarray*}
G_{\varepsilon } &=&\sum_{j}B_{j,\varepsilon }\circ X_{j}+\sum_{j}X_{j}^{%
\func{tr}}\circ \widetilde{B_{j,\varepsilon }}+B_{0,\varepsilon }\ , \\
B_{0,\varepsilon } &\in &\limfunc{Op}\left( \mathcal{C}^{0,\delta }S_{1,\eta
}^{0,2}\right) \text{ and }B_{j,\varepsilon },\widetilde{B_{j,\varepsilon }}%
\in \limfunc{Op}\left( \mathcal{C}^{1,\delta }S_{1,\eta }^{0,1}\right) \ ,
\end{eqnarray*}%
where the pseudodifferential operator coefficients $B_{0,\varepsilon }$, $%
B_{j,\varepsilon }$ and $\widetilde{B_{j,\varepsilon }}$ lie uniformly in
the indicated operator classes. A similar argument holds for $%
E_{\varepsilon} $ and $J_{\varepsilon}$. All functions have sufficient \
differentiability for the proof of Lemma \ref{pre} to apply, and this proof,
together with the above identity, yield 
\begin{equation*}
\left\Vert \eta _{1}\Lambda _{\varepsilon }u\right\Vert _{L^{2}\left(
R\right) }\leq C\left\Vert \eta _{1}\Lambda _{\varepsilon }Lu\right\Vert
_{L^{2}\left( R\right) }+C\left\Vert u\right\Vert _{H^{0}\left( R\right) }\ ,
\end{equation*}%
uniformly in $\varepsilon >0$. We conclude as desired that the $L^{2}$ norm
of $\eta _{1}\Lambda _{\varepsilon }u$ remains bounded as $\varepsilon
\searrow 0$.

Thus we have proved that for any distribution $u\in \mathcal{D}^{\prime
}\left( V\right) $, and any $0<\gamma <\delta $, there is a symbol $\Lambda $
as in (\ref{def lambda}) that is elliptic of order $\gamma $ on the\ conical
set $\Gamma $, and satisfies 
\begin{equation*}
\left\Vert \eta _{1}\Lambda u\right\Vert _{L^{2}\left( R\right) }\leq
C\left\Vert \eta _{1}\Lambda Lu\right\Vert _{L^{2}\left( R\right)
}+C\left\Vert u\right\Vert _{H^{0}\left( R\right) }.
\end{equation*}%
The proof of Theorem \ref{main_thm_Grushin} is now complete.

Combined with the bootstrapping argument above, this shows that $u\in H_{%
\limfunc{loc}}^{s}\left( R\right) $ for all $s\in \mathbb{R}$. Indeed, $\eta
_{2}u\in H^{-M}\left( R\right) $ for some $M$ sufficiently large, and thus
we can begin the bootstrapping argument at $s=-M$.

\section{Proof of Theorem \protect\ref{HYPSOS G}}

We now prove Theorem \ref{HYPSOS G}. The first step is to use a
bootstrapping argument to reduce matters to the level of $L^{2}\left( 
\mathbb{R}^{n}\right) $. Consider the general second order divergence form
operator%
\begin{equation*}
Lu\left( x\right) \equiv \nabla ^{\func{tr}}A\left( x\right) \nabla u\left(
x\right) +D\left( x\right) u\left( x\right) ,
\end{equation*}%
where $A$ and $D$ are real and smooth, and where $A\left( x\right) $
satisfies appropriate form comparability conditions. In order to conclude
hypoellipticity of $L$ it is enough to show that there is $\gamma >0$ such
that for every $s\in \mathbb{R}$, we have the bootstrapping inequality 
\begin{equation*}
u\in H_{\func{loc}}^{s}\left( \mathbb{R}^{n}\right) \text{ and }Lu\in H_{%
\func{loc}}^{s+\gamma }\left( \mathbb{R}^{n}\right) \Longrightarrow u\in H_{%
\func{loc}}^{s+\gamma }\left( \mathbb{R}^{n}\right) \ \ \ \ \ \text{for all }%
s\in \mathbb{R}.
\end{equation*}%
Now with $\widehat{\Lambda _{s}}\left( \xi \right) \equiv \left(
1+\left\vert \xi \right\vert ^{2}\right) ^{\frac{s}{2}}$, and $\gamma >0$
fixed, it suffices to show%
\begin{equation*}
u\in H_{\func{loc}}^{0}\left( \mathbb{R}^{n}\right) \text{ and }\Lambda
_{s}L\Lambda _{-s}u\in H_{\func{loc}}^{\gamma }\left( \mathbb{R}^{n}\right)
\Longrightarrow u\in H_{\func{loc}}^{\gamma }\left( \mathbb{R}^{n}\right) \
\ \ \ \ \text{for all }s\in \mathbb{R}.
\end{equation*}%
For $s\geq 0$ we use 
\begin{equation*}
\Lambda _{s}L\Lambda _{-s}=\left( \Lambda _{s}L-L\Lambda _{s}\right) \Lambda
_{-s}+L=\left[ \Lambda _{s},L\right] \Lambda _{-s}+L,
\end{equation*}%
and for $s\leq 0$ we use%
\begin{equation*}
\Lambda _{s}L\Lambda _{-s}=-\Lambda _{s}\left( \Lambda _{-s}L-L\Lambda
_{-s}\right) +L=-\Lambda _{s}\left[ \Lambda _{-s},L\right] +L,
\end{equation*}%
to conclude that it suffices to prove%
\begin{eqnarray}
u &\in &H_{\func{loc}}^{0}\left( \mathbb{R}^{n}\right) \text{ and }\left[
\Lambda _{s},L\right] \Lambda _{-s}u\in H_{\func{loc}}^{\gamma }\left( 
\mathbb{R}^{n}\right) \Longrightarrow u\in H_{\func{loc}}^{\gamma }\left( 
\mathbb{R}^{n}\right) \ \ \ \ \ \text{for all }s\geq 0,  \label{suffices} \\
u &\in &H_{\func{loc}}^{0}\left( \mathbb{R}^{n}\right) \text{ and }\left[
\Lambda _{-s},L\right] \Lambda _{s}u\in H_{\func{loc}}^{\gamma }\left( 
\mathbb{R}^{n}\right) \Longrightarrow u\in H_{\func{loc}}^{\gamma }\left( 
\mathbb{R}^{n}\right) \ \ \ \ \ \text{for all }s\leq 0.  \notag
\end{eqnarray}

The second step is to use the sum of squares assumption in part (1) of
Theorem \ref{HYPSOS G} to show that it is sufficient to establish the
conditions of Theorem \ref{main_thm_Grushin}. So define 
\begin{equation}
\widetilde{G}\equiv \left[ \Lambda _{s},L\right] \Lambda _{-s}=\Lambda
_{s}L\Lambda _{-s}-L,  \label{G:tilde}
\end{equation}%
and suppose for the moment that the operator $L$ has the simple form 
\begin{equation}
L=\sum_{j}X_{j}^{\func{tr}}X_{j},  \label{L:simple}
\end{equation}%
where $L\in S_{1,0}^{2}$ is smooth and $X_{j}\in C^{2,\delta }$. We first
establish the properties of $\tilde{G}$ we need using the rough version of
asymptotic expansion from \cite{Saw} given in Theorem \ref{Saw calc} above,
which we repeat here for the reader's convenience.

Suppose $\sigma \in \mathcal{C}^{\nu }S_{1,\delta _{1}}^{m_{1}}$ and $\tau
\in \mathcal{C}^{M+\mu +\nu }S_{1,\delta _{2}}^{m_{2}}$ where $M$ is a
nonnegative integer, $0<\mu ,\delta _{1},\delta _{2}<1$, $\nu >0$ and $M+\mu
\geq m_{1}\geq 0$. Let $\delta \equiv \max \left\{ \delta _{1},\delta
_{2}\right\} $. Then 
\begin{eqnarray*}
\sigma \circ \tau &=&\sum_{\ell =0}^{M}\frac{1}{i^{\ell }\ell !}\nabla _{\xi
}^{\ell }\sigma \cdot \nabla _{x}^{\ell }\tau +E; \\
E &\in &\mathcal{O}_{\left( -\left( 1-\delta \right) \nu ,\nu \right)
}^{m_{1}+m_{2}+\left( M+\mu \right) \left( \delta _{2}-1\right) +\varepsilon
},\ \ \ \ \ \text{for every }\varepsilon >0.
\end{eqnarray*}

\begin{lemma}
\label{lem:conj} Let $L$ and $\widetilde{G}$ be as in (\ref{L:simple}) and (%
\ref{G:tilde}). Then 
\begin{eqnarray}
\widetilde{G} &=&\sum_{j}B_{j}\circ X_{j}+\sum_{j}X_{j}^{\func{tr}}\circ 
\widetilde{B_{j}}+B_{0}\ , \\
B_{0} &\in &\mathcal{O}_{\left( -\delta /2,\delta /2\right) }^{-\delta
/2+\varepsilon }\text{ for every }\varepsilon >0,\text{ and }B_{j},%
\widetilde{B_{j}}\in \limfunc{Op}\left( C^{1,\delta }S_{1,0}^{0}\right) \ . 
\notag
\end{eqnarray}
\end{lemma}

\begin{proof}
First we note that 
\begin{equation*}
\lbrack \Lambda _{s},L]=\sum_{j}[\Lambda _{s},X_{j}^{\func{tr}}]X_{j}+X_{j}^{%
\func{tr}}[\Lambda _{s},X_{j}],
\end{equation*}%
and so we investigate operators $[\Lambda _{s},X_{j}^{\func{tr}}]$ and $%
[\Lambda _{s},X_{j}]$. The analysis is similar, so we only give details for $%
[\Lambda _{s},X_{j}]$. Using Theorem (\ref{Saw calc}) with $m_{1}=s$, $%
m_{2}=1$, $M=1$, $\mu =1+\delta /2$, $\nu =\delta /2$ and $\delta
_{1}=\delta _{2}=0$ we have 
\begin{equation*}
\sigma ([\Lambda _{s},X_{j}])=C\nabla _{\xi }\left( 1+\left\vert \xi
\right\vert ^{2}\right) ^{\frac{s}{2}}\cdot \nabla _{x}\sigma (X_{j})+E,
\end{equation*}%
where $E\in \mathcal{O}_{\left( -\delta /2,\delta /2\right) }^{1+s-(2+\delta
/2)+\varepsilon }$. Composing with $\Lambda _{-s}$ and using $\limfunc{Op}%
\left( \nabla _{\xi }\left( 1+\left\vert \xi \right\vert ^{2}\right) ^{\frac{%
s}{2}}\right) =R^{-1}\circ \Lambda _{s}$, where $R^{-1}\in S_{1,0}^{-1}$, we
obtain 
\begin{equation*}
X_{j}^{\func{tr}}[\Lambda _{s},X_{j}]\Lambda _{-s}=X_{j}^{\func{tr}}\circ 
\widetilde{B_{j}}+R
\end{equation*}%
with $\widetilde{B_{j}}\in C^{1,\delta }S_{1,0}^{0}$ and $R\in \mathcal{O}%
_{\left( -\delta /2,\delta /2\right) }^{-\delta /2+\varepsilon }$.
\end{proof}

Now we start with an operator $L\in S_{1,0}^{2}$ of the more general form 
\begin{equation}
L=\sum_{j}X_{j}^{\func{tr}}X_{j}+A_{0}+\widehat{\nabla }^{\func{tr}}\cdot 
\mathbf{Q}_{p}\left( x\right) \widehat{\nabla },  \label{L:full}
\end{equation}%
where $X_{j}\in C^{2,\delta }$ and $A_{0}\in S_{1,0}^{1}$. Using Lemma \ref%
{lem:conj} for any operator $L$ in the form (\ref{L:full}) and Remark \ref%
{Grush-conj} we can show that the operator $\Lambda _{s}L\Lambda _{-s}$ has
the form 
\begin{equation*}
\Lambda _{s}L\Lambda _{-s}=\sum_{j}X_{j}^{\func{tr}}X_{j}+%
\sum_{j}B_{j}X_{j}+\sum_{j}X_{j}^{\func{tr}}\tilde{B}_{j}+B_{0}+R_{1}+%
\widehat{\nabla }^{\func{tr}}\mathbf{Q}_{p}\left( x\right) \widehat{\nabla },
\end{equation*}%
where $X_{j}$, $B_{j}$, $\tilde{B}_{j}$, and $B_{0}$ are as in Lemma \ref%
{lem:conj} and $R_{1}$ is as in Theorem \ref{main_thm_Grushin}. Thus to show
hypoellipticity of the operator (\ref{L:full}), it is sufficient to show
that it satisfies the hypotheses of Theorem \ref{main_thm_Grushin}, which
completes the second step of the proof.

We prepare for the final step of the proof with an auxiliary Lemma (see \cite%
[Lemma 5.1]{Chr}), and its corollary to be used later for showing condition (%
\ref{Christ_cond1-G}).

\begin{lemma}
\label{lem:aux1} Let $\varphi \in C_{0}^{2}(\mathbb{R}^{n})$, $f\in
C^{\infty }(\mathbb{R}^{n})$ simply positive, and $s>0$. Then for any $l\in
\{1,\dots ,n\}$ there exists a constant $C_{l}$ independent of $s$ such that 
\begin{equation}
||\varphi ||^{2}\leq C_{l}\left( \frac{1}{\tau ^{2}[\min_{|x|\geq s}f(x)]^{2}%
}+s^{2}\right) \left( ||\partial _{x_{l}}\varphi ||^{2}+\int \tau
^{2}f(x)^{2}\varphi (x)^{2}dx\right) ,  \label{bound_aux}
\end{equation}%
where the minimum is taken over all $x\in \mathrm{supp}\varphi $ s.t. $%
|x|\geq s$.
\end{lemma}

\begin{proof}
Fix $s>0$, for any $x\in \mathbb{R}^{n}$ we have 
\begin{align*}
\varphi (x)& =\varphi \left( x+s\frac{x_{l}}{|x_{l}|}\right)
-\int_{1}^{1+s/|x_{l}|}\frac{\partial \varphi }{\partial t}(x_{1},\dots
,x_{l-1},tx_{l},x_{l+1},\dots ,x_{n})dt \\
\varphi ^{2}(x)& \lesssim \varphi ^{2}\left( x+s\frac{x_{l}}{|x_{l}|}\right)
+\left( \int_{1}^{1+s/|x_{l}|}\nabla \varphi (x_{1},\dots
,x_{l-1},tx_{l},x_{l+1},\dots ,x_{n})\cdot (0,\dots ,x_{l},0,\dots
,0)dt\right) ^{2} \\
\int_{|x_{l}|\leq s}\varphi ^{2}(x)dx& \lesssim \int_{|x_{l}|\leq s}\varphi
^{2}\left( x+s\frac{x_{l}}{|x_{l}|}\right) dx \\
& \quad +\int_{|x_{l}|\leq s}\left( \int_{1}^{1+s/|x_{l}|}|\partial
_{l}\varphi (x_{1},\dots ,x_{l-1},tx_{l},x_{l+1},\dots
,x_{n})t^{3/4}x_{l}|^{2}dt\int_{1}^{1+s/|x_{l}|}t^{-3/2}dt\right) dx \\
\int_{|x_{l}|\leq s}\varphi ^{2}(x)dx& \lesssim \int_{s\leq |x_{l}|\leq
2s}\varphi ^{2}\left( x\right) dx+\int_{|x_{l}|\leq
s}\int_{1}^{1+s/|x_{l}|}|\partial _{l}\varphi (x_{1},\dots
,x_{l-1},tx_{l},x_{l+1},\dots ,x_{n})t^{3/4}x_{l}|^{2}dtdx.
\end{align*}%
Switching the order of integration in the last term on the right and making
a change of variables $y=(x_{1},\dots ,x_{l-1},tx_{l},x_{l+1},\dots ,x_{n})$
we obtain 
\begin{equation*}
\int_{|x_{l}|\leq s}\int_{1}^{1+s/|x_{l}|}|\partial _{l}\varphi (x_{1},\dots
,x_{l-1},tx_{l},x_{l+1},\dots ,x_{n})t^{3/4}x_{l}|^{2}dtdx\leq
\int_{1}^{\infty }\int_{|y_{l}|\leq 2s}|\partial _{l}\varphi
(y)y_{l}|^{2}t^{-1/2}\frac{dy}{t}dt\lesssim s^{2}\int |\partial _{l}\varphi
(y)|^{2}dy,
\end{equation*}%
which combining with the above gives 
\begin{equation*}
\int_{|x_{l}|\leq s}\varphi ^{2}(x)dx\lesssim \int_{s\leq |x|\leq 2s}\varphi
^{2}\left( x\right) dx+s^{2}\int |\partial _{l}\varphi (x)|^{2}dx.
\end{equation*}%
Finally, 
\begin{equation*}
\int_{|x_{l}|\geq s}\tau ^{2}f(x)^{2}\varphi (x)^{2}dx\geq \tau
^{2}[\min_{|x_{l}|\geq s}f(x)]^{2}\int_{|x_{l}|\geq s}\varphi ^{2}(x)dx,
\end{equation*}%
and thus altogether 
\begin{equation*}
\int \varphi ^{2}(x)dx\lesssim \frac{1}{\tau ^{2}[\min_{|x_{l}|\geq
s}f(x)]^{2}}\int_{|x|\geq s}\tau ^{2}f(x)^{2}\varphi (x)^{2}dx+s^{2}\int
|\partial _{l}\varphi (x)|^{2}dx
\end{equation*}%
which implies (\ref{bound_aux}).
\end{proof}

\begin{lemma}
\label{lem:cond1} Let $\varphi$ and $f$ as in Lemma \ref{lem:aux1}. There
exists a strictly positive continuous function $w$ satisfying $w(\tau)\to
\infty$ as $\tau\to\infty$ such that for every $l\in \{1,\dots,n\}$ and some
constant $C_{l}>0$ 
\begin{equation}
\int w(\tau)^{2}\varphi(x)^{2}dx\leq C_{l}\int\left(|\partial_{l}
\varphi(x)|^{2}+\tau^2f(x)^{2}\varphi(x)^{2}\right)dx.
\end{equation}
\end{lemma}

\begin{proof}
For all $s\geq 0$ define 
\begin{equation*}
f_{0}(s)\equiv \min_{x\in \mathrm{supp}\varphi :|x|\geq s}f(x),
\end{equation*}%
and note that $f_{0}(0)=0$, $f_{0}(s)>0$ for $s\neq 0$, and $f_{0}$ is
nondecreasing on $[0,\infty )$. Let $r=r(\tau )>0$ be the unique point
satisfying 
\begin{equation}
\frac{1}{r}=\tau f_{0}(r).  \label{r}
\end{equation}%
Define the function $w$ by 
\begin{equation*}
w(\tau )=\inf_{0<s<\infty }\left( \frac{1}{s}+\tau f_{0}(s)\right) ,
\end{equation*}%
since $1/s$ is nonincreasing and $f_{0}(s)$ nondecreasing in $s$ we have $%
w(\tau )\approx 1/r$ where $r$ is given by (\ref{r}). Therefore, $w(\tau
)\rightarrow \infty $ as $\tau \rightarrow \infty $ and using (\ref%
{bound_aux}) with $s=r$ we obtain 
\begin{align*}
\int w(\tau )^{2}\varphi (x)^{2}dx& \leq C_{l}\frac{1}{r^{2}}\left( \frac{1}{%
\tau ^{2}f_{0}(r)^{2}}+r^{2}\right) \int \left( |\partial _{l}\varphi
(x)|^{2}+\tau ^{2}f(x)^{2}\varphi (x)^{2}\right) dx \\
& \leq C_{l}\int \left( |\partial _{l}\varphi (x)|^{2}+\tau
^{2}f(x)^{2}\varphi (x)^{2}\right) dx.
\end{align*}
\end{proof}

\subsection{Sufficiency}

We can now proceed to complete the sufficiency part of Theorem \ref{HYPSOS G}%
. We note that without loss of generality we may assume that the diagonal
entries $\lambda _{k}\left( \tilde{x}\right) $ are smooth. Indeed, from $%
A\left( x\right) \sim D_{\mathbf{\lambda }}\left( x\right) $ we obtain $%
A\left( x\right) \sim A_{\func{diag}}\left( x\right) $ and hence%
\begin{equation}
\lambda _{k}\left( \tilde{x}\right) \approx a_{k,k}\left( x\right) \approx
a_{k,k}\left( \tilde{x},0,0\right) ,  \label{lambda a}
\end{equation}%
where the functions $a_{k,k}\left( \tilde{x},0,0\right) $ are smooth for $%
1\leq k\leq n$ by assumption.

\begin{proof}[Proof of sufficiency in Theorem \protect\ref{HYPSOS G}]
Let $(\xi _{1},\dots ,\xi _{m},\eta _{m+1},\dots ,\eta _{n})$ denote the
dual variables, and denote $\xi =(\xi _{1},\dots ,\xi _{m})$, $\eta =(\eta
_{m+1},\dots ,\eta _{n})$, $\tilde{x}=(x_{1},\dots ,x_{m})$. Define 
\begin{equation*}
R=\{(x,\xi ,\eta ):x=0,\xi =0,\ \eta _{m+1},\dots ,\eta _{n}>0\}.
\end{equation*}%
%
%
%
%
%
%
%
%
%
%
%
%
%
%
%
%
The principal symbol of $L$ vanishes on the manifold $\tilde{x}=\xi =0$, so
it suffices to prove that $Lu\in H^{s}(\mathfrak{N}\left( R\right) )\implies
u\in H^{s}(\mathfrak{N}\left( R\right) )$ for some conical neighbourhood $%
\mathfrak{N}\left( R\right) $ of the ray $R$. We start with verifying
condition (\ref{Christ_cond1-G}). Let $\mathcal{F}(u)(\tilde{x},\eta )$ be
the partial Fourier transform of $u$ in $n-m$ variables $\eta $, then from
Lemma \ref{lem:cond1} with $x=\tilde{x}$ and $\varphi (\tilde{x})=\mathcal{F}%
(u)(\tilde{x},\eta )$, we have for $k=m+1,\dots ,p-1$ 
\begin{equation*}
\int w(\eta _{k})^{2}\mathcal{F}(u)(\tilde{x},\eta )^{2}d\tilde{x}\leq C\int
\left( |\nabla _{\tilde{x}}\mathcal{F}(u)(\tilde{x},\eta )|^{2}+\eta
_{k}^{2}\lambda _{k}(\tilde{x})\mathcal{F}(u)(\tilde{x},\eta )^{2}\right) d%
\tilde{x},
\end{equation*}%
and for $k=p,\dots ,n$ 
\begin{equation*}
\int w(\eta _{k})^{2}\mathcal{F}(u)(\tilde{x},\eta )^{2}d\tilde{x}\leq C\int
\left( |\nabla _{\tilde{x}}\mathcal{F}(u)(\tilde{x},\eta )|^{2}+\eta
_{k}^{2}\lambda _{p}(\tilde{x})\mathcal{F}(u)(\tilde{x},\eta )^{2}\right) d%
\tilde{x},
\end{equation*}%
where $w(s)\rightarrow \infty $ as $s\rightarrow \infty $. Adding the
inequalities together gives 
\begin{align*}
& \int w(|\eta |)^{2}\mathcal{F}(u)(\tilde{x},\eta )^{2}d\tilde{x} \\
& \qquad \leq C\int \left( |\nabla _{\tilde{x}}\mathcal{F}(u)(\tilde{x},\eta
)|^{2}+\left[ \sum_{k=m+1}^{p-1}\eta _{k}^{2}\lambda _{k}(\tilde{x}%
)+\sum_{k=p}^{n}\eta _{k}^{2}\lambda _{p}(\tilde{x})\right] \mathcal{F}(u)(%
\tilde{x},\eta )^{2}\right) d\tilde{x},
\end{align*}%
where $w(|\eta |)\rightarrow \infty $ as $|\eta |\rightarrow \infty $.
Combining with the first line in (\ref{add_cond_gen G}) we obtain 
\begin{equation*}
\int_{\mathbb{R}^{n}}w(|(\eta )|)^{2}\mathcal{F}(u)(\tilde{x},\eta )^{2}d%
\tilde{x}d\eta \leq C\sum_{j}\left\Vert X_{j}u\right\Vert ^{2}+C\left\Vert 
\sqrt{\lambda _{p}}\widehat{\nabla }u\right\Vert ^{2},
\end{equation*}%
which gives upon using the first condition in (\ref{add_cond_gen G}) again 
\begin{align*}
\int \min \{|(\xi ,\eta )|,w(|(\xi ,\eta )|)\}^{2}\hat{u}(\xi ,\eta
)^{2}d\xi d\eta & \lesssim \int_{|\xi |\leq |\eta |}w(|\eta |)^{2}\left\vert 
\hat{u}(\xi ,\eta )\right\vert ^{2}d\xi d\eta +\int_{|\xi |\geq |\eta |}|\xi
|^{2}\left\vert \hat{u}(\xi ,\eta )\right\vert ^{2}d\xi d\eta \\
& \leq C\sum_{j}||X_{j}u||^{2}+C\left\Vert \sqrt{\lambda _{p}}\widehat{%
\nabla }u\right\Vert ^{2}.
\end{align*}

We proceed to verify (\ref{superlog-G}) with $p\equiv 0$ and $\psi $
constructed below. Since the principal symbol of the operator vanishes on $%
\mathbb{R}^{n-m}\times \mathbb{R}^{n-m}$, namely when $\tilde{x}=\xi =0$, we
need to localize matters to a strip $\left\vert \left( \tilde{x},\frac{\xi }{%
\left\vert \eta \right\vert }\right) \right\vert <\rho $ where $\psi $
enjoys favorable commutation relations with the symbol $\sigma \left(
X_{j}\right) $ of the vector field $X_{j}$. So let $p\equiv 0$ and let $\rho
>0$. Let $\psi \in C^{\infty }(T^{\func{tr}}V)$ be homogeneous of degree $0$
with respect to $\left( \xi ,\eta \right) $ and satisfy 
\begin{equation*}
\begin{cases}
\psi =1, & \quad \text{if}\ \left\vert \left( x,\frac{\xi }{\left\vert \eta
\right\vert }\right) \right\vert \geq 3\rho \\ 
\psi =0, & \quad \text{if}\ \left\vert \left( x,\frac{\xi }{\left\vert \eta
\right\vert }\right) \right\vert \leq \rho \\ 
\psi =\psi (x_{m+1},\dots ,x_{n}), & \quad \text{if}\ \left\vert \left( 
\tilde{x},\frac{\xi }{\left\vert \eta \right\vert }\right) \right\vert \leq
2\rho%
\end{cases}%
.
\end{equation*}%
Thus $\psi $ is $1$ outside a large ball of radius $3\rho $, vanishes inside
a small ball of radius $\rho $, and makes the transition from $0$ to $1$ in
the strip while depending only on the variables $\tilde{x}$ in the strip $%
\left\vert \left( \tilde{x},\frac{\xi }{\left\vert \eta \right\vert }\right)
\right\vert \leq 2\rho $. In the strip $\left\vert \left( \tilde{x},\frac{%
\xi }{\left\vert \eta \right\vert }\right) \right\vert <\rho $, $\psi $ is a
function of variables $x_{m+1},\dots ,x_{n}$ only, and the main step of
Christ's application of his theorem occurs now: for each $j=1,\dots ,k$
there exist $a_{\ell }^{j}\left( \tilde{x}\right) $, $\ell =m+1,\dots ,n$
such that 
\begin{align*}
& \{\psi ,\sigma (X_{j})\}=i\sum_{\ell =m+1}^{n}a_{\ell }^{j}\left( \tilde{x}%
\right) \,\partial _{x_{l}}\psi , \\
\left\vert a_{\ell }^{j}\left( \tilde{x}\right) \right\vert & \lesssim \sqrt{%
\lambda _{\ell }\left( \tilde{x}\right) },\ \ell =m+1,\dots ,p-1,\ \
\left\vert a_{\ell }^{j}\left( \tilde{x}\right) \right\vert \lesssim \sqrt{%
\lambda _{p}(\tilde{x})},\ \ell =p,\dots ,n
\end{align*}%
using conditions (\ref{add_cond_gen G}), and 
\begin{equation*}
\{\psi ,\eta \}=i\hat{\nabla}\psi .
\end{equation*}%
Using the condition $\left\vert \xi \right\vert \leq \rho \left\vert \eta
\right\vert $ this gives for each $j=1,\dots ,N$ 
\begin{align*}
\left\Vert \limfunc{Op}\left[ \log \langle (\xi ,\eta )\rangle \{\psi
,\sigma (X_{j})\}\right] u\right\Vert ^{2}& \lesssim \sum_{\ell
=m+1}^{p-1}\left\Vert \limfunc{Op}\left[ \sqrt{\lambda _{\ell }(\tilde{x})}%
\log \langle \eta \rangle \right] u\right\Vert ^{2}+\left\Vert \limfunc{Op}%
\left[ \sqrt{\lambda _{p}\left( \tilde{x}\right) }\log \langle \eta \rangle %
\right] u\right\Vert ^{2} \\
& =\int \Lambda _{\func{sum}}\left( \tilde{x}\right) \log \langle \eta
\rangle ^{2}\mathcal{F}(u)(\tilde{x},\eta )^{2}d\tilde{x}d\eta ,
\end{align*}%
and 
\begin{align*}
\left\Vert \sqrt{\mathbf{Q}_{p}}\limfunc{Op}\left[ \log \langle (\xi ,\eta
)\rangle \{\psi ,\eta \}\right] u\right\Vert ^{2}& \lesssim \left\Vert \sqrt{%
\lambda _{p}}\limfunc{Op}\left[ \log \langle \eta \rangle \right]
u\right\Vert ^{2} \\
& \lesssim \int \Lambda _{\func{sum}}\left( \tilde{x}\right) \log \langle
\eta \rangle ^{2}\mathcal{F}(u)(\tilde{x},\eta )^{2}d\tilde{x}d\eta ,
\end{align*}%
upon using the definition of $\Lambda _{\func{sum}}\left( \tilde{x}\right) $%
. To show (\ref{superlog-G}) it is therefore sufficient to establish the
first inequality in the following display (since the second follows directly
from (\ref{add_cond_gen G})) 
\begin{align}
& \int \log \langle \eta \rangle ^{2}\Lambda _{\func{sum}}\left( \tilde{x}%
\right) \mathcal{F}(u)(\tilde{x},\eta )^{2}d\tilde{x}d\eta \lesssim \delta
\int \left\vert \nabla _{\tilde{x}}\mathcal{F}(u)(\tilde{x},\eta ,\tau
)\right\vert ^{2}d\tilde{x}d\eta d\tau  \notag \\
& \qquad +\delta \int \left[ \sum_{k=m+1}^{p-1}\eta _{k}^{2}\lambda _{k}(%
\tilde{x})+\sum_{k=p}^{n}\eta _{k}^{2}\lambda _{p}(\tilde{x})\right] 
\mathcal{F}(u)(\tilde{x},\eta )^{2}d\tilde{x}d\eta +C_{\delta }||u||^{2}
\label{bound_n'} \\
& \qquad \lesssim \delta \sum_{j=1}^{N}||X_{j}u||^{2}+\delta ||\sqrt{\lambda
_{p}}\hat{\nabla}u||^{2}+C_{\delta }||u||^{2}.  \notag
\end{align}%
Using the definitions of $\Lambda _{\func{sum}}\left( \tilde{x}\right) $ and 
$\Lambda _{\func{product}}\left( \tilde{x}\right) $ we conclude that it is
sufficient to show 
\begin{equation}
(\log \tau )^{2}||\sqrt{\Lambda _{\func{sum}}}\varphi ||^{2}\leq \delta
(\tau )||\nabla _{\tilde{x}}\varphi ||^{2}+\delta (\tau )\tau ^{2}||\sqrt{%
\Lambda _{\func{product}}}\varphi ||^{2},\quad \text{ for all}\ \ \varphi
\in C_{0}^{1}(\mathbb{R}^{m}),  \label{suffic3_nd}
\end{equation}%
where $\delta (\tau )\rightarrow 0$ as $\tau \rightarrow \infty $. Indeed, (%
\ref{suffic3_nd}) together with the bound $0\leq \lambda _{j}\leq 1$ implies 
\begin{align*}
& \int \log \langle \eta \rangle ^{2}\Lambda _{\func{sum}}\varphi (\tilde{x}%
)^{2}d\tilde{x}\leq \delta (\langle \eta \rangle )||\nabla _{\tilde{x}%
}\varphi ||^{2}+\delta (\langle \eta \rangle )\langle \eta \rangle ^{2}||%
\sqrt{\Lambda _{\func{product}}}\varphi ||^{2} \\
& \quad \leq \delta (\langle \eta \rangle )||\nabla _{\tilde{x}}\varphi
||^{2}+\delta (\langle \eta \rangle )\left[ \sum_{k=m+1}^{p-1}|\eta
_{k}|^{2}||\sqrt{\lambda _{k}}\varphi ||^{2}+\sum_{k=p}^{n}|\eta _{k}|^{2}||%
\sqrt{\lambda _{p}}\varphi ||^{2}+\right] +C_{\delta }||\varphi ||^{2}.
\end{align*}%
This implies (\ref{bound_n'}) by splitting the region of integration into $%
|\eta |$ sufficiently large so that $\delta (\langle \eta \rangle )\leq
\delta $, and the region where $|\eta |$ is bounded, and thus the left hand
side of (\ref{bound_n'}) is bounded by $C||u||^{2}$.

To establish (\ref{suffic3_nd}), we first recall for convenience the Koike
condition 
\begin{equation}
\lim_{\tilde{x}\rightarrow 0}\mu (|\tilde{x}|,\sqrt{\Lambda _{\func{sum}}}%
)\ln \Lambda _{\func{product}}(\tilde{x})=0.  \label{log assump_min}
\end{equation}%
Now let $\phi \in C_{0}^{1}(B(0,r))$. Then we then have with $\phi _{\tilde{y%
}}\left( \rho \right) \equiv \phi \left( \rho \tilde{y}\right) $,%
\begin{eqnarray}
&&\int_{|\tilde{x}|\leq r}\Lambda _{\func{sum}}\left( \tilde{x}\right) \phi (%
\tilde{x})^{2}d\tilde{x}=\int_{|\tilde{x}|\leq r}\Lambda _{\func{sum}}\left( 
\tilde{x}\right) (r-|\tilde{x}|)^{2}\frac{\phi (\tilde{x})^{2}}{(r-|\tilde{x}%
|)^{2}}d\tilde{x}  \label{claim} \\
&\leq &\mu \left( r,\sqrt{\Lambda _{\func{sum}}}\right)^{2} \int_{|\tilde{x}%
|\leq r}\frac{\phi (\tilde{x})^{2}}{(r-|\tilde{x}|)^{2}}d\tilde{x}=\mu
\left( r,\sqrt{\Lambda _{\func{sum}}}\right)^{2} \int_{\mathbb{S}%
^{m-1}}\left\{ \int_{0}^{r}\left( \frac{1}{r-\rho }\int_{\rho }^{r}\phi _{%
\tilde{y}}^{\prime }\left( \rho \right) \right) ^{2}\rho ^{m-1}d\rho
\right\} d\tilde{y}  \notag \\
&\leq &\mu \left( r,\sqrt{\Lambda _{\func{sum}}}\right)^{2} \int_{\mathbb{S}%
^{m-1}}\left\{ 4\int_{0}^{r}\phi _{\tilde{y}}^{\prime }\left( \rho \right)
^{2}\rho ^{m-1}d\rho \right\} d\tilde{y}\leq 4\mu \left( r,\sqrt{\Lambda _{%
\func{sum}}}\right)^{2} \int \left\vert \nabla _{\tilde{x}}\phi (\tilde{x}%
)\right\vert ^{2},  \notag
\end{eqnarray}%
where in the last line we have applied Hardy's inequality.

Fix $\varphi \in C_{0}^{1}(\mathbb{R}^{m})$ as in (\ref{suffic3_nd}). Let $%
\chi \in C_{0}^{1}(\mathbb{R}^{1})$ satisfy $\chi (t)=1$ for $\left\vert
t\right\vert \leq 1$ and $\chi (t)=0$ for $\left\vert t\right\vert \geq 2$,
and define the function 
\begin{equation}
\nu (\tilde{x})\equiv \chi (\tau \Lambda _{\func{product}}(\tilde{x})),
\label{def:nu'}
\end{equation}%
and the set 
\begin{equation*}
I(\tau )\equiv \{\tilde{x}\in \func{Supp}\varphi :\tau \Lambda _{\func{%
product}}(\tilde{x})>1\}.
\end{equation*}%
We can write 
\begin{equation}
\int \Lambda _{\func{sum}}\left( \tilde{x}\right) \varphi (\tilde{x})^{2}d%
\tilde{x}\leq 2\int \Lambda _{\func{sum}}\left( \tilde{x}\right) \nu (\tilde{%
x})^{2}\varphi (\tilde{x})^{2}d\tilde{x}+2\int \Lambda _{\func{sum}}\left( 
\tilde{x}\right) (1-\nu (\tilde{x}))^{2}\varphi (\tilde{x})^{2}d\tilde{x}.
\label{step1'}
\end{equation}%
To estimate the second integral we notice that it vanishes outside the set $%
I(\tau )$ and thus 
\begin{eqnarray}
(\log \tau )^{2}\int \Lambda _{\func{sum}}\left( \tilde{x}\right) (1-\nu (%
\tilde{x}))^{2}\varphi (\tilde{x})^{2}d\tilde{x} &\leq &(\log \tau
)^{2}\int_{I(\tau )}\Lambda _{\func{product}}(\tilde{x})\varphi (\tilde{x}%
)^{2}d\tilde{x}  \label{step2'} \\
&=&\delta \left( \tau \right) \tau ^{2}\int \Lambda _{\func{product}}(\tilde{%
x})\varphi (\tilde{x})^{2}d\tilde{x},  \notag
\end{eqnarray}%
where $\delta \left( \tau \right) =(\log \tau )^{2}\tau ^{-1}\rightarrow 0$
as $\tau \rightarrow \infty $.

To estimate the first integral on the right hand side of (\ref{step1'}) we
define 
\begin{equation*}
r(\tau )\equiv \sup \{|\tilde{y}|:\tilde{y}\in \func{Supp}\varphi :\tau
\Lambda _{\func{product}}(\tilde{y})\leq 2\}.
\end{equation*}%
Since $\func{Supp}\varphi $ is compact, the supremum above is attained at
some point $\tilde{z}\in \func{Supp}\varphi $, and moreover we have both%
\begin{equation*}
\left\vert \tilde{z}\right\vert =r\text{ and }\tau =\frac{2}{\Lambda _{\func{%
product}}\left( \tilde{z}\right) }.
\end{equation*}%
Thus $\ln \tau \approx \ln \frac{1}{\Lambda _{\func{product}}\left( \tilde{z}%
\right) }$ and so%
\begin{equation*}
\mu \left( r\left( \tau \right) ,\sqrt{\Lambda _{\func{sum}}}\right) \ln
r\left( \tau \right) \approx \mu \left( \left\vert \tilde{z}\right\vert ,%
\sqrt{\Lambda _{\func{sum}}}\right) \ln \frac{1}{\Lambda _{\func{product}%
}\left( \tilde{z}\right) }.
\end{equation*}

The Koike condition condition (\ref{log assump_min}) now implies 
\begin{equation}
\lim_{\tau \rightarrow \infty }\mu (r\left( \tau \right) ,\sqrt{\Lambda _{%
\func{sum}}})\ln r\left( \tau \right) =\lim_{\tilde{x}\rightarrow 0}\mu (|%
\tilde{x}|,\sqrt{\Lambda _{\func{sum}}})\ln \frac{1}{\Lambda _{\func{product}%
}(\tilde{x})}=0,  \label{g_lim'}
\end{equation}%
since $r\left( \tau \right) \rightarrow 0$ as $\tau \rightarrow \infty $. We
now need to combine this result with (\ref{claim}) to obtain the desired
estimate. Let $\phi (\tilde{x})=\nu (\tilde{x})\varphi (\tilde{x})$. Then
using the definition of $\nu (\tilde{x})$ in (\ref{def:nu'}) we obtain 
\begin{align*}
\int |\nabla _{\tilde{x}}\phi (\tilde{x})|^{2}d\tilde{x}& \leq C\int |\nabla
_{\tilde{x}}\nu (\tilde{x})|^{2}\varphi (\tilde{x})^{2}d\tilde{x}+C\int \nu (%
\tilde{x})^{2}|\nabla _{\tilde{x}}\varphi (\tilde{x})|^{2}d\tilde{x} \\
& \leq C\tau ^{2}\int_{I(\tau )}|\nabla _{\tilde{x}}\Lambda _{\func{product}%
}(\tilde{x})|^{2}\varphi (\tilde{x})^{2}d\tilde{x}+C\int |\nabla _{\tilde{x}%
}\varphi (\tilde{x})|^{2}d\tilde{x} \\
& \leq C\tau ^{2}\int_{I(\tau )}\Lambda _{\func{product}}(\tilde{x})\varphi (%
\tilde{x})^{2}d\tilde{x}+C\int |\nabla _{\tilde{x}}\varphi (\tilde{x})|^{2}d%
\tilde{x},
\end{align*}%
where in the last inequality we used the Malgrange inequality, see e.g. \cite%
[Lemme I]{Gla}, applied to $\Lambda _{\func{product}}(\tilde{x}%
)=\dprod\limits_{k=m+1}^{p}\lambda _{k}\left( \tilde{x}\right) $, where the
functions $\lambda _{k}$ are smooth by (\ref{lambda a}). Finally, from the
definition of $r$ and (\ref{def:nu'}) it follows that%
\begin{equation*}
\func{Supp}\phi \subset \func{Supp}\nu\subset \left\{ \tilde{y}:\tau <\frac{2%
}{\Lambda _{\func{product}}(\tilde{y})}\right\} \subset B\left( 0,r\left(
\tau \right) \right) .
\end{equation*}%
since if $\left\vert \tilde{x}\right\vert >r\left( \tau \right) $, then $%
\tau \Lambda _{\func{product}}(\tilde{y})>2$ by the definition of $r\left(
\tau \right) $.

Combining the above estimate with (\ref{g_lim'}) and (\ref{claim}) we
conclude that 
\begin{eqnarray*}
(\log \tau )^{2}\int \Lambda _{\func{sum}}(\tilde{x})\nu (\tilde{x}%
)^{2}\varphi (\tilde{x})^{2}d\tilde{x} &=&(\log \tau )^{2}\int \Lambda _{%
\func{sum}}(\tilde{x})\phi (\tilde{x})^{2}d\tilde{x} \\
&\leq &\delta (\tau )\left( \tau ^{2}\int \Lambda _{\func{product}}(\tilde{x}%
)\varphi (\tilde{x})^{2}d\tilde{x}+\int |\nabla _{\tilde{x}}\varphi (\tilde{x%
})|^{2}d\tilde{x}\right)
\end{eqnarray*}%
with $\delta (\tau )=C\mu (r,\sqrt{\Lambda _{\func{sum}}})^{2}(\log \tau
)^{2}\rightarrow 0$ as $\tau \rightarrow \infty $. Together with (\ref%
{step2'}) this gives (\ref{suffic3_nd}).
\end{proof}

\subsection{Sharpness}

We now turn to the sharpness portion of Theorem {\ref{HYPSOS G}}. If the
Koike condition (\ref{log assump_min}) fails, then 
\begin{eqnarray*}
0 &<&\lim \sup_{\tilde{x}\rightarrow 0}\mu \left( \left\vert \tilde{x}%
\right\vert ,\sqrt{\Lambda _{\func{sum}}}\right) \ln \frac{1}{\Lambda _{%
\func{product}}(\tilde{x})} \\
&=&\lim \sup_{\tilde{x}\rightarrow 0}\mu \left( \left\vert \tilde{x}%
\right\vert ,\sqrt{\sum_{k=m+1}^{p}\lambda _{k}\left( \tilde{x}\right) }%
\right) \ln \dprod\limits_{k=m+1}^{p}\frac{1}{\lambda _{k}\left( \tilde{x}%
\right) } \\
&\leq &\lim \sup_{\tilde{x}\rightarrow 0}\mu \left( \left\vert \tilde{x}%
\right\vert ,\sum_{k=m+1}^{p}\sqrt{\lambda _{k}\left( \tilde{x}\right) }%
\right) \sum_{j=m+1}^{p}\ln \frac{1}{\lambda _{j}\left( \tilde{x}\right) } \\
&\leq &\sum_{k,j=m+1}^{p}\lim \sup_{\tilde{x}\rightarrow 0}\mu \left(
\left\vert \tilde{x}\right\vert ,\sqrt{\lambda _{k}\left( \tilde{x}\right) }%
\right) \ln \frac{1}{\lambda _{j}\left( \tilde{x}\right) }
\end{eqnarray*}%
shows that $p>m+1$ (since $\lim \sup_{\tilde{x}\rightarrow 0}\mu \left(
\left\vert \tilde{x}\right\vert ,\sqrt{\lambda _{p}\left( \tilde{x}\right) }%
\right) \ln \frac{1}{\lambda _{p}\left( \tilde{x}\right) }=0$) and that
there is a pair of distinct indices $k,j\in \left\{ m+1,...,p\right\} $ such
that 
\begin{equation*}
\lim \sup_{\tilde{x}\rightarrow 0}\mu \left( \left\vert \tilde{x}\right\vert
,\sqrt{\lambda _{k}\left( \tilde{x}\right) }\right) \ln \frac{1}{\lambda
_{j}\left( \tilde{x}\right) }>0.
\end{equation*}%
Our sharpness assertion in Theorem {\ref{HYPSOS G} now follows immediately
from Proposition \ref{L1 L2} and Theorem \ref{strongly mon} below.}

{To prove the Proposition and Theorem, }we will need the following lemma
(see Hoshiro \cite[(2.7)]{Hos}), whose short proof we include here for the
reader's convenience.

\begin{lemma}[T. Hoshiro \protect\cite{Hos}]
\label{lem:Hosh}Let $L$ be a hypoelliptic operator on $\mathbb{R}^{n}$. For
any multiindex $\beta $ and any subsets $\Omega ,\Omega ^{\prime }$ of $%
\mathbb{R}^{n}$ such that $\Omega ^{\prime }\Subset \Omega $, there exists $%
N\in \mathbb{N}$ and $C>0$ such that 
\begin{equation}
||D^{\beta }u||_{L^{2}(\Omega ^{\prime })}^{2}\leq C\left( \sum_{|\alpha
|\leq N}||D^{\alpha }Lu||_{L^{2}(\Omega )}^{2}+||u||_{L^{2}(\Omega
)}^{2}\right) \quad \forall u\in C^{\infty }(\overline{\Omega }).
\label{hypo_bound}
\end{equation}
\end{lemma}

\begin{proof}
Fix $\Omega^{\prime }\Subset \Omega$ and consider the set 
\begin{equation*}
S\equiv \{u\in L^{2}\left( \Omega ^{\prime }\right) :D^{\alpha }Lu\in
L^{2}(\Omega ^{\prime })\ \text{for all multiindices}\ \alpha \}.
\end{equation*}%
The family of seminorms $||u||_{L^{2}(\Omega ^{\prime })},\ ||D^{\alpha
}Lu||_{L^{2}(\Omega ^{\prime })},\ |\alpha |\in \mathbb{N}$, makes it a Fr%
\'{e}chet space. Since $L$ is hypoelliptic we have $S\subset C^{\infty
}(\Omega ^{\prime })$, and in particular $S\subset C^{M}(\Omega ^{\prime })$
for any $M>0$. Now consider the inclusion map 
\begin{equation*}
T:S\rightarrow C^{M}(\Omega ^{\prime }),
\end{equation*}%
we claim $T$ is closed. Indeed, suppose $\{u_{n}\}\subset S$ satisfies $%
u_{n}\rightarrow u$ in $S$ and $u_{n}\rightarrow v$ in $C^{M}(\Omega
^{\prime })$, in particular, $u_{n}\rightarrow u$ in $L^{2}(\Omega ^{\prime
})$ and $u_{n}\rightarrow v$ in $L^{\infty }(\Omega ^{\prime })$. Then for
any $n\in \mathbb{N}$ 
\begin{equation*}
||u-v||_{L^{2}(\Omega ^{\prime })}\leq ||u-u_{n}||_{L^{2}(\Omega ^{\prime
})}+||u_{n}-v||_{L^{2}(\Omega ^{\prime })}\leq ||u-u_{n}||_{L^{2}(\Omega
^{\prime })}+||u_{n}-v||_{L^{\infty }(\Omega ^{\prime })}\left\vert \Omega
^{\prime }\right\vert ^{\frac{1}{2}}\rightarrow 0\ \text{as}\ n\rightarrow
\infty .
\end{equation*}%
This implies $u=v$, i.e. $T$ is closed. By the closed graph theorem $T$ is
continuous, and therefore there exists $N\in\mathbb{N}$ and $C>0$ such that 
\begin{equation*}
||u||_{C^{M}(\Omega^{\prime })}\leq C\left(\sum_{|\alpha|\leq
N}||D^{\alpha}Lu||_{L^{2}(\Omega^{\prime })}^{2}+||u||_{L^{2}(\Omega^{\prime
})}^{2}\right).
\end{equation*}
Since the choice of $M$ was arbitrary, this implies (\ref{hypo_bound}).
\end{proof}

\begin{proposition}
\label{L1 L2}Fix distinct indices $k,j\in \left\{ m+1,...,p\right\} $ where $%
p>m+1$. Define 
\begin{align*}
L_{1}& \equiv \frac{\partial ^{2}}{\partial x_{1}^{2}}+\dots +\frac{\partial
^{2}}{\partial x_{m}^{2}}+\lambda _{k}\left( x_{1},\dots ,x_{m}\right) \frac{%
\partial ^{2}}{\partial x_{k}^{2}}+\lambda _{j}(x_{1},\dots ,x_{m})\frac{%
\partial ^{2}}{\partial x_{j}^{2}}, \\
L_{2}& \equiv \frac{\partial ^{2}}{\partial x_{1}^{2}}+\dots +\frac{\partial
^{2}}{\partial x_{m}^{2}}+\sum_{i=m+1}^{p}\lambda _{i}\left( x_{1},\dots
,x_{m}\right) \frac{\partial ^{2}}{\partial x_{i}^{2}}+\sum_{i=p+1}^{n}%
\lambda _{p}\left( x_{1},\dots ,x_{m}\right) \frac{\partial ^{2}}{\partial
x_{i}^{2}}.
\end{align*}%
If $L_{1}$ is not hypoelliptic in $\mathbb{R}^{m+2}$, then $L_{2}$ is not
hypoelliptic in $\mathbb{R}^{n}$.
\end{proposition}

\begin{proof}
Suppose $L_{1}$ is not hypoelliptic in $\mathbb{R}^{m+2}$, i.e. there exists
a non smooth function $u=u(x_{1},\dots ,x_{m},x_{k},x_{j})$ such that $%
L_{1}u\in C^{\infty }(\mathbb{R}^{m+2})$. If we define the function $v$ by 
\begin{equation*}
v(x_{1},\dots ,x_{n})=u(x_{1},\dots ,x_{m},x_{k},x_{j}),
\end{equation*}%
then $v$ is not smooth since $u$ is not smooth. However, 
\begin{equation*}
L_{2}v(x_{1},\dots ,x_{n})=L_{1}u(x_{1},\dots ,x_{m},x_{k},x_{j})
\end{equation*}%
and is therefore smooth in $\mathbb{R}^{n}$.
\end{proof}

\begin{theorem}
\label{strongly mon}Suppose that $h,f\in C^{\infty }\left( \mathbb{R}%
^{m}\right) $ are \emph{strongly monotone}, i.e. 
\begin{equation*}
f(z)\leq f(x)\ \text{and}\ h(z)\leq h(x)\ \ \text{for all}\ z\in B\left(
0,|x|\right) ,
\end{equation*}%
and satisfy $h(x),\ f\left( x\right) \geq 0$ and $h(0)=f\left( 0\right) =0$
for all $x\in \mathbb{R}^{m}$. Define 
\begin{equation}
\mu (t,h)\equiv \max \{h(z)(t-|z|):0\leq |z|\leq t\}.
\end{equation}%
and suppose in addition that 
\begin{equation}
\liminf_{x\rightarrow 0}\mu (|x|,h)\ln f(x)\neq 0.  \label{decay}
\end{equation}%
Then the operator 
\begin{equation*}
\mathcal{L}\equiv \Delta _{x}+f^{2}\left( x\right) \frac{\partial ^{2}}{%
\partial y^{2}}+h^{2}(x)\frac{\partial ^{2}}{\partial t^{2}}
\end{equation*}%
\textbf{fails} to be $C^{\infty }$-hypoelliptic in $\mathbb{R}^{m+2}$.
\end{theorem}


\begin{proof}
For $a,\eta >0$ consider the second order operator $L_{\eta }\equiv
-\Delta_{x}+f^{2}\left( x\right) \eta ^{2}$ and the eigenvalue problem 
\begin{eqnarray*}
&L_{\eta }v\left( x,\eta \right) =\lambda \ h^{2}(x)v\left( x,\eta \right)
,&\ \ \ \ \ x\in B(0,a) , \\
&v(x)=0 ,&\ \ \ \ \ x\in \partial B(0,a).
\end{eqnarray*}
The least eigenvalue is given by the Rayleigh quotient formula 
\begin{eqnarray}
\lambda _{0}\left( a,\eta \right) &=&\inf_{\varphi\left( \neq 0\right) \in
C_{0}^{\infty }\left(B\right) }\frac{\left\langle L_{\eta
}\varphi,\varphi\right\rangle _{L^{2}}}{\left\langle
h^{2}\varphi,\varphi\right\rangle _{L^{2}}}  \notag \\
&=&\inf_{\varphi\left( \neq 0\right) \in C_{0}^{\infty }\left( B\right) }%
\frac{ \int_{B}\left\vert \nabla\varphi\right\vert\left( x\right)
^{2}dx+\int_{B}f^{2}\left( x\right) \eta ^{2}\varphi\left( x\right) ^{2}dx}{%
\int_{B}h(x)^{2}\varphi(x)^{2}dx}.  \label{Rayleigh}
\end{eqnarray}
Next, from (\ref{decay}) it follows that there exists $\varepsilon>0$ and
sequences $\{a_n\},\ \{b_n\}\subset\mathbb{R}^{m}$ s.t. $|a_n|<|b_n|\leq 1$, 
$b_n\to 0$, and 
\begin{equation}  \label{decay_impl}
h(a_n)\left(|b_n|-|a_n|\right)|\ln f(b_n)|\geq \varepsilon,\ \quad\forall
n\in\mathbb{N}.
\end{equation}
Let 
\begin{equation*}
\eta_n=\frac{1}{f(b_n)}\to\infty\quad\text{as}\ \ n\to\infty,
\end{equation*}

By strong monotonicity of $f$ and $h$ we have 
\begin{equation*}
\eta_{n}f(x)\leq 1,\ \ h(x)\geq h(a_n)\ \ \forall x\in R_{n}\equiv\{x\in 
\mathbb{R}^{m}: |a_n|\leq |x|\leq |b_n|\}.
\end{equation*}
This implies using (\ref{Rayleigh}) 
\begin{align*}
\lambda _{0}\left( |b_{n}|,\eta_{n} \right)&\leq \inf_{\varphi\left( \neq
0\right) \in C_{0}^{\infty }\left( R_{n}\right) }\frac{\left\langle
L_{\eta_{n} }\varphi,\varphi\right\rangle _{L^{2}}}{\left\langle
h^{2}\varphi,\varphi\right\rangle _{L^{2}}} \\
&\leq h(a_n)^{-2}\inf_{\varphi\left( \neq 0\right) \in C_{0}^{\infty }\left(
R_{n}\right)
}\{(\left\Vert\nabla\varphi\right\Vert^{2}+\left\Vert\varphi\right%
\Vert^{2})/\left\Vert\varphi\right\Vert^{2}\} \\
&\leq h(a_n)^{-2} (C(|b_n|-|a_n|)^{-2}+1)\leq C|\ln f(b _n)|^{2}=
C(\ln\eta_n)^{2},
\end{align*}
where we used (\ref{decay_impl}) and the definition of $\eta_n$ for the last
two inequalities. It also follows from (\ref{Rayleigh}) and the fact that $%
|b_n|\leq 1$ 
\begin{equation}  \label{lambda bound}
\lambda _{0}\left( 1,\eta_{n} \right) \leq \lambda _{0}\left(
|b_{n}|,\eta_{n}\right)\leq C_{1}(\ln\eta_n)^{2}.
\end{equation}
Now let $v_{0}\left( x,\eta_{n} \right) $ be an eigenfunction on the ball $%
B=B(0,1) $ associated with $\lambda _{0}\left( 1,\eta_{n} \right) $ i.e. 
\begin{equation*}
-\Delta v_{0}\left( x,\eta_{n}\right) =\left[ \lambda _{0}\left(
1,\eta_{n}\right)h^{2}(x) -f^{2}\left( x\right) \eta_{n}^{2}\right]
v_{0}\left( x,n\right) ,
\end{equation*}
and normalized so that 
\begin{equation}
\left\Vert v_{0}\left( \cdot ,\eta_{n} \right) \right\Vert _{L^{2}\left(
B\right) }=1.  \label{normalized}
\end{equation}
We first claim that 
\begin{equation}  \label{norm2}
\left\Vert v_{0}\left( \cdot ,\eta_{n} \right) \right\Vert _{L^{2}\left(
(1/2)B\right) }\to 1\quad\text{as}\ n\to\infty.
\end{equation}
Indeed, we have 
\begin{align*}
&\inf_{1/2<|x|<1}f^{2}(x)\eta_{n}^{2}\int_{1/2<|x|<1}|v_{0}\left( x
,\eta_{n} \right)|^{2}dx \\
&\quad\leq \int_{B}f^{2}(x)\eta_{n}^{2}|v_{0}\left( x ,\eta_{n}
\right)|^{2}dx \\
&\quad\leq \int_{B}|\nabla v_{0}\left( x ,\eta_{n}
\right)|^{2}dx+\int_{B}f^{2}(x)\eta_{n}^{2}|v_{0}\left( x ,\eta_{n}
\right)|^{2}dx \\
&\quad=\lambda_{0}\left( 1,\eta_{n}\right)\int_{B}h^{2}(x)|v_{0}\left( x
,\eta_{n} \right)|^{2}dx \leq C \lambda_{0}\left( 1,\eta_{n}\right).
\end{align*}
Dividing both sides by $\inf_{1/2<|x|<1}f^{2}(x)\eta_{n}^{2}$ and using (\ref%
{lambda bound}) we obtain that 
\begin{equation*}
\int_{1/2<|x|<1}|v_{0}\left( x ,\eta_{n} \right)|^{2}dx\to 0\quad\text{as}%
\quad n\to \infty,
\end{equation*}
which implies (\ref{norm2}).

Define a sequence of functions 
\begin{equation*}
u_{n}(x,y,t)=e^{iy\eta _{n}+\sqrt{\lambda _{0}\left( 1,\eta _{n}\right) }%
t}v_{0}\left( x,\eta _{n}\right) .
\end{equation*}%
Then 
\begin{equation*}
\mathcal{L}u_{n}=\left( \Delta v_{0}\left( x,\eta _{n}\right) -\eta
_{n}^{2}f^{2}(x)v_{0}\left( x,\eta _{n}\right) +\lambda _{0}\left( 1,\eta
_{n}\right) v_{0}\left( x,\eta _{n}\right) \right) e^{iy\eta _{n}+\sqrt{%
\lambda _{0}\left( 1,\eta _{n}\right) }t}=0.
\end{equation*}%
Now, let $V=B(0,1)\times \lbrack -\pi ,\pi ]\times \lbrack -\delta ,\delta ]$
and $V^{\prime }=B(0,1/2)\times \lbrack -\pi /2,\pi /2]\times \lbrack
-\delta /2,\delta /2]$ for some $\delta >0$. We have using (\ref{norm2}) 
\begin{equation*}
||\partial _{y}^{k}u_{n}||_{L^{2}(V^{\prime })}^{2}=\eta
_{n}^{2k}||u_{n}||_{L^{2}(V^{\prime })}^{2}\geq \pi \eta
_{n}^{2k}\int_{1/2B}\int_{0}^{\delta /2}e^{2\sqrt{\lambda _{0}\left( 1,\eta
_{n}\right) }t}|v_{0}\left( x,\eta _{n}\right) |^{2}dtdx\geq C\eta _{n}^{2k},
\end{equation*}%
where the constant $C$ is independent of $k$ and $n$. On the other hand,
using (\ref{lambda bound}) 
\begin{equation*}
||u_{n}||_{L^{2}(V)}^{2}\leq Ce^{2\sqrt{\lambda _{0}\left( 1,\eta
_{n}\right) }\delta }\leq C\eta _{n}^{2\sqrt{C_{1}}\delta }.
\end{equation*}%
Since $\eta _{n}\rightarrow \infty $ as $n\rightarrow \infty $, these two
inequalities contradict (\ref{hypo_bound}) for $k>\sqrt{C_{1}}\delta $, and
thus by Lemma \ref{lem:Hosh}\ the operator $\mathcal{L}$ is not hypoelliptic.
\end{proof}

\section{Proof of Theorem \protect\ref{HYP G}}

Finally, we prove Theorem \ref{HYP G} by showing that the requirements of
Theorem \ref{HYPSOS G} are satisified. Let $L$ be as in (\ref{L form' G}).
We apply Theorem \ref{final n Grushin} to obtain $\mathbf{A}%
=\sum_{j=1}^{N}Y_{j}Y_{j}^{\func{tr}}+A_{p}$, and write the second order
term in $L$ as 
\begin{equation*}
\nabla ^{\limfunc{tr}}\mathbf{A}\nabla =\sum_{j=1}^{N}\nabla ^{\limfunc{tr}%
}Y_{j}Y_{j}^{\func{tr}}\nabla =\sum_{j=1}^{N}X_{j}^{\func{tr}}X_{j}+\widehat{%
\nabla }^{\func{tr}}\mathbf{Q}_{p}\widehat{\nabla }\ ,\ \ \ \ \ \text{where }%
X_{j}=Y_{j}^{\func{tr}}\nabla ,
\end{equation*}%
and then note that condition (\ref{log assump_n' G}) is satisfied by the
assumption (\ref{star}) of Theorem \ref{HYP G}. Moreover, condition (\ref%
{add_cond_gen G}) follows from (\ref{more}).

\section{Open problems}

\subsection{First problem}

In Theorem \ref{HYP G} we have shown that the Koike condition is sufficient
for the hypoellipticity of an operator $L$ with $n\times n$ matrix $A\left(
x\right) $ satisfying certain conditions on both its diagonal and
nondiagonal entries. However, in the converse direction we only showed that
failure of the Koike condition implies failure of hypoellipticity if in
addition $L$\ is diagonal with strongly monotone entries. In fact the proof
shows that we need only assume in addition that $A\left( x\right) $ has the
block form%
\begin{equation*}
A\left( x\right) =\left[ 
\begin{array}{ccccc}
\left[ 
\begin{array}{ccc}
a_{1,1}\left( x\right) & \cdots & a_{n,1}\left( x\right) \\ 
\vdots & \ddots & \vdots \\ 
a_{1,n}\left( x\right) & \cdots & a_{m,m}\left( x\right)%
\end{array}%
\right] & \mathbf{0}_{m\times 1} & \mathbf{0}_{m\times 1} & \cdots & \mathbf{%
0}_{m\times 1} \\ 
\mathbf{0}_{1\times m} & a_{m+1,m+1}\left( x\right) & 0 & \cdots & 0 \\ 
\mathbf{0}_{1\times m} & 0 & a_{m+2,m+2}\left( x\right) &  & 0 \\ 
\vdots & \vdots & \vdots & \ddots & \vdots \\ 
\mathbf{0}_{1\times m} & 0 & 0 & \cdots & a_{n,n}\left( x\right)%
\end{array}%
\right] .
\end{equation*}%
where just $a_{m+1,m+1}\left( x\right) $ and $a_{n,n}\left( x\right) $ are
assumed to be strongly monotone and satisfy (\ref{decay}).

\begin{problem}
Is the Koike condition actually necessary and sufficient for hypoellipticity
under the assumptions of Theorem \ref{HYP G}, \textbf{without} assuming the
above block form for $A\left( x\right) $?
\end{problem}

\subsection{Second problem}

Recall that the main theorem in \cite{KoRi} extends Kohn's theorem in \cite%
{Koh} to apply with finitely many blocks instead of the two blocks used in 
\cite{Koh}. These operators are restricted by being of a certain block form,
but they are more general in that the elliptic blocks are multiplied by
smooth functions that are positive outside the origin, and have more
variables than in our theorems, and furthermore that need not be finite sums
of squares of regular functions.

\begin{problem}
Can Theorem \ref{HYPSOS G} be extended to more general operators that
include the operators appearing in \cite{KoRi}?
\end{problem}

\subsection{Third problem}

What sort of smooth lower order terms of the form $B\left( x\right) \nabla $
and $\nabla ^{\func{tr}}C\left( x\right) $ can we add to the operator $L$ in
the main Theorem \ref{HYP G}? The natural hypothesis to make on the vector
fields $B\left( x\right) \nabla $ and $C\left( x\right) \nabla $ is that
they are subunit with respect to $\nabla ^{\func{tr}}A\left( x\right) \nabla 
$. However, if we use Theorem \ref{HYPSOS G} in the proof, we require more,
namely that $B\left( x\right) \nabla $ and $C\left( x\right) \nabla $ are
linear combinations, with $C^{2,\delta }$ coefficients, of the $C^{2,\delta
} $ vector fields $X_{j}\left( x\right) $ arising in the sum of squares
Theorem \ref{final n Grushin}, something which seems difficult to arrange
more generally.


\begin{thebibliography}{AkKoRi}
\bibitem[AkKoRi]{AkKoRi} \textsc{T. Akhunov, L. Korobenko and C. Rios, }%
\textit{Hypoellipticity of Fed\u{\i}i's type operators under Morimoto's
logarithmic condition, }J. Pseudo-Differ. Oper. Appl. \textbf{10} (2019),
649-688.

\bibitem[BoCoRo]{BoCoRo} \textsc{J. Bochnak, M. Coste and M.-F.Roy}, \textit{%
G\'{e}om\'{e}trie alg\'{e}brique r\'{e}elle}, Springer-Verlag, Berlin, 1987.

\bibitem[Bon]{Bon} \textsc{J.-M. Bony,} \textit{Sommes de Carr\'{e}s de
fonctions d\'{e}rivables}, Bull. Soc. math. France \textbf{133} (4), 2005,
p. 619--639.

\bibitem[Bou]{Bou} \textsc{G. Bourdaud,} $L^{p}$\textit{\ estimates for
certain nonregular pseudo-differential operators,} Comm. Partial
Differential Equations \textbf{7} (1982), 1023-1033.

\bibitem[Bru]{Bru} \textsc{G. Brumfiel}, \textit{Partially ordered rings and
semi-algebraic geometry}, LondonMathematical Society Lecture Note Series,
vol. \textbf{37}, Cambridge University Press, Cambridge-New York, 1979.

\bibitem[Chr]{Chr} {\textsc{M. Christ, }}\textit{Hypoellipticity in the
infinitely degenerate regime, Complex Analysis and Geometry, Ohio State
Univ. Math. Res. Instl Publ. \textbf{9}, }Walter de Gruyter, New York
(2001), 59-84\textit{.}

\bibitem[Chr2]{Chr2} {\textsc{M. Christ, }\textit{Hypoellipticity:
geometrization and speculation, }}\texttt{arXiv.9801.142v1}.

\bibitem[Fe]{Fe} {\textsc{V. S. Fedi\u{\i},} \textsc{\ }\textit{On a
criterion for hypoellipticity},\textit{\ }Math. USSR Sbornik\textit{\ }%
\textbf{14} (1971), 15-45.}

\bibitem[FePh]{FePh} \textsc{C. Fefferman and D. H. Phong,} \textit{On
positivity of pseduo-differential operators},\textit{\ }Proc. Natl. Acad.
Sci. USA vol. \textbf{75}, No. 10, p 4673-4674, October 1978.

\bibitem[Gla]{Gla} \textsc{Georges Glaeser,} \textit{Racine carr\'{e}e d'une
fonction diff\'{e}rentiable}, Annales de l'institut Fourier, tome \textbf{13}%
, no 2 (1963), p. 203-210.

\bibitem[Gua]{Gua} \textsc{P. Guan,}\textit{\ }$C^{2}$\textit{\ \emph{a
priori} estimates for degenerate Monge-Amp\`{e}re equations}, Duke Math. J. 
\textbf{86} (1997), 323-346.

\bibitem[Ho]{Ho} \textsc{L. Hormander,} \textit{Hypoelliptic second order
differential equations}, Acta. Math\textit{. }\textbf{119} (1967), 141-171.

\bibitem[Hos]{Hos} \textsc{T. Hoshiro,} \textit{Hypoellipticity for
infinitely degenerate elliptic and parabolic operators of second order}, J.
Math. Kyoto Univ. (JMKYAZ) \textbf{28}-4 (1988), 615-632.

\bibitem[Koh]{Koh} \textsc{J. J. Kohn}, \textit{Hypoellipticity of Some
Degenerate Subelliptic Operators, }Journal of functional analysis \textbf{%
159 }(1998), p. 203-216.

\bibitem[Koi]{Koi} \textsc{M. Koike, }\textit{A note on hypoellipticity of
degenerate elliptic operators, }Publ. RIMS, Kyoto Univ. \textbf{27} (1991),
995-1000.

\bibitem[KoRi]{KoRi} \textsc{L. Korobenko and C. Rios, }\textit{%
Hypoellipticity of a class of infinitely degenerate second order operators
and systems,} \texttt{arXiv:1301.2339v2}.

\bibitem[KoSa1]{KoSa1} \textsc{L. Korobenko and E. Sawyer,} \textit{Sums of
squares I: scalar\ functions, }\texttt{arXiv.}

\bibitem[KoSa2]{KoSa2} \textsc{L. Korobenko and E. Sawyer,} \textit{Sums of
squares II: matrix\ functions, }\texttt{arXiv.}

\bibitem[KuStr]{KuStr} {\textsc{S. Kusuoka and D. Stroock,} \textsc{\ }%
\textit{Applications of the Malliavin Calculus II}, J. Fac. Sci. Univ. Tokyo 
\textbf{\ 32} (1985), 1--76. }

\bibitem[Mor]{Mor} {\textsc{Y. Morimoto,} \textsc{\ }\textit{%
Non-Hypoellipticity for Degenerate Elliptic Operators}, Publ. RIMS, Kyoto
Univ. \textbf{22} (1986), 25-30; and \textit{Erratum to "Non-Hypoellipticity
for Degenerate Elliptic Operators",} Publ. RIMS, Kyoto Univ. \textbf{34}
(1994), 533-534.}

\bibitem[Pie]{Pie} \textsc{F. Pieroni}, \textit{On the real algebra of
Denjoy-Carleman classes}, Sel. math., New ser. \textbf{13} (2007), 321--351.

\bibitem[RiSaWh]{RiSaWh} {\textsc{C. Rios, E. Sawyer and R. Wheeden, }%
\textit{\ Hypoellipticity for Infinitely Degenerate Quasilinear Equations
and the Dirichlet Problem}, Journal d'Analyse Math\'{e}matique \textbf{119}
(2013), 1 - 62.}

\bibitem[Saw]{Saw} \textsc{E. Sawyer,} \textit{A symbolic calculus for rough
pseudodifferential operators}, Ind. U. Math. J. \textbf{45} (1996), 289- 332.

\bibitem[Ste]{Ste} \textsc{E. M. Stein,} \textit{Harmonic Analysis:
real-variable methods, orthogonality, and oscillatory integrals},\textit{\ }%
Princeton University Press, Princeton, N. J., 1993.

\bibitem[Tat]{Tat} \textsc{D. Tataru,} \textit{On the Fefferman-Phong
inequality and related problems}, preprint.

\bibitem[Tay]{Tay} \textsc{M. E. Taylor}, \textit{Partial differential
equations I, II and III}, Applied Mathematical Sciences \textbf{115}, 
\textbf{116} and \textbf{117}, Springer, New York, 1996.

\bibitem[Tre]{Tre} \textsc{F. Treves,} \textit{Introduction to
pseudodifferential and Fourier integral operators vol. 1\ and 2}, Univ. Ser.
in Math. Plenum Press, New York and London 1980.
\end{thebibliography}
\end{document}